\documentclass{article}
\usepackage{amsmath}
\usepackage{ulem}
\usepackage{graphicx,graphics}
\usepackage{accents}
\usepackage[mathscr]{eucal}
\usepackage{amsthm}
\usepackage{amssymb}
\usepackage{enumerate}
\usepackage{cases}
\usepackage{subcaption}
\usepackage{pgf,tikz}
 \usepackage{xcolor}
\usepackage{csquotes}
 \usepackage{hyperref}
 
 \usepackage{comment}
\usepackage{cite}

\definecolor{dhcol}{rgb}{0,0.5,0}

\definecolor{sccol}{rgb}{0,0,0.5}

\definecolor{amcol}{rgb}{0.5,0,0}

\usepackage{manfnt} % Define textdbend
  \newcommand{\Becommente}[1]{{\color{purple}{#1}}}

\newcommand*{\RR}{{\mathbb{R}}}

\def\N{\mathbb{N}}
\def\Z{\mathbb{Z}}
\def\C{\mathbb{C}}
\def\R{\mathbb{R}}
\def\cA{\mathcal{A}}

\def\cg{\mathcal{G}^\sharp}
\def\cG{\mathcal{G}}
\def\dsp{\displaystyle}

\def\thetag{\boldsymbol{\theta}}

\def\kg{\mathbf{k}}

\def\xg{\mathbf{x}}
\def\vg{\mathbf{v}}

\newtheorem{theorem}{\bf Theorem}[section]
\newtheorem{definition}{\bf Definition}[section]
\newtheorem{lemma}{\bf Lemma}[section]
\newtheorem{proposition}{\bf Proposition}[section]
\newtheorem{remark}{\bf Remark}[section]
\newtheorem{corollary}{\bf Corollary}[section]

\begin{document}

\title{Guided modes in a hexagonal periodic graph like domain}

\author{B\'erang\`ere Delourme\footnote{LAGA (UMR7539), Université Sorbonne Paris Nord, Villetaneuse, France, delourme@math.univ-paris13.fr}  \hspace{0.1cm}and  Sonia Fliss\footnote{POEMS, CNRS, INRIA, ENSTA Paris, Institut Polytechnique de Paris, 91120 Palaiseau, France, sonia.fliss@ensta-paris.fr}
}
\maketitle

% Sets running headers as well as PDF title and authors
%%%% Subject entries to be placed here %%%%
{\bf Subjects:} Analysis of PDEs, Spectral Theory\\

%%%% Keyword entries to be placed here %%%%

{\bf Abstract:} This paper deals with the existence of guided waves and edge states in particular two-dimensional media obtained by perturbing  a reference periodic medium with honeycomb symmetry. This reference medium is  a thin periodic domain (the thickness is denoted $\delta>0$) with an hexagonal structure, which is close to an honeycomb quantum graph. %It is well known now that such hexagonal properties could create conical crossings (also named Dirac points), which are the source of fundamental phenomena in condensed Matter physics (e.g. graphene and topological insulators). 
In a first step, we show the existence of Dirac points (conical crossings) at arbitrarily large frequencies if $\delta$ is chosen small enough. We then perturbe the domain by cutting the perfectly periodic medium  along the so-called zig-zag direction, and we consider either Dirichlet or Neumann boundary conditions on the cut edge. In the two cases, we prove  the existence of edges modes as well as their robustness  with respect to some perturbations, namely the location of the cut and the thickness of the perturbed edge. In particular, we show that different locations of the cut lead  to almost-non dispersive edge states, the number of locations increasing with the frequency. All the results are obtained via asymptotic analysis and semi-explicit computations done on the limit quantum graph. 
Numerical simulations illustrate the theoretical results. 
\\

{\bf Acknowledgements: }{The authors acknowledge David Gontier and Antonin Coutand for interesting and fruitful discussions. The authors would like to thank the Isaac Newton Institute for Mathematical Sciences, Cambridge, for support and hospitality during the programme " Mathematical theory and applications of multiple wave scattering" where work on this paper was undertaken. This work was supported by EPSRC grant no EP/R014604/1.
}

% REQUIRED

%%%%%%%%%%%%%%%%%%%%%%%%%%%

%%%%%%%%%% Insert the texts which can accomdate on firstpage in the tag "fmtext" %%%%%

\section{Introduction}
The propagation of waves in periodic media has known a regain of interest the past decades, in optics for micro and nano-technology. Indeed, in some frequency ranges, periodic structures behave as insulators or filters: the corresponding monochromatic waves, also called Floquet modes, cannot propagate in the bulk. The study of these modes is, from a mathematical point of view, related to the spectrum of the underlying operator that presents a so-called band structure: the spectrum may contain some forbidden frequency intervals, called gaps. Even if necessary conditions for the existence of gaps are not known, in lots of papers sufficient conditions are proposed. Let us mention, for instance, that playing with the (high) contrast of the materials \cite{Figotin:1996a,Figotin:1996b,Post:2003,lipton:2017} or the shape of the boundary of the medium \cite{Cardone:2010,Nazarov:2010a,Article_DFJV_Part1}, gaps can be created. 

In Material Science, the spectral study of the graphene, a two dimensional material with a honeycomb structure, which is well described using a tight binding model, has explained its remarkable conductivity properties and its behaviour as a topological insulator in presence of a magnetic field.  Indeed, the associated tight-binding model has a band structure consisting of two dispersion surfaces which conically touch at Dirac points, around the so-called Fermi (or Dirac) energy \cite{Fefferman:2012,Fefferman:2014}. Dirac points have been shown to appear for a large class of honeycomb Schrodinger operators \cite{Fefferman:2012,berkolaiko:2018}. Analogous properties have been proven for another class of elliptic operators of divergence form and with honeycomb symmetry, see \cite{Lee:2019, ammari:2020,Cassier:2021}. This is of particular interest in order to create engineered honeycomb media, also called artificial graphene, in order to reproduce the remarkable topological properties in another context, for photonics \cite{Plotnik:2014,Rechtsman:2013, Noh:2018,Ozawa:2019}, acoustics \cite{Torrent:2012,Zhang:2018,Coutant:2021,Coutant:2022} or elastic \cite{Wang:2015} applications.

 The first aim of this paper is to complement the references mentioned above by proving existence of several Dirac points at different energy, or, in our context, different frequencies. To be more specific, we consider the Laplace operator with Neumann boundary condition in a ladder-like periodic domain with a honeycomb symmetry and we use a standard approach of asymptotic analysis that consists in deducing properties of the operator from the ones of the limit operator when the thickness of the rung tends to 0. The limit domain consists on a honeycomb periodic graph and the limit operator on the second order derivative operator on each edge of the graph together with so-called Kirchhoff conditions at its vertices. The spectrum of the limit operator can be explicitly determined  (see for instance \cite{KuchmentBookGraph,KuchmentPost,KuchmentQuantumSurvey}). Note that in this paper we revisit the result for the quantum graph operator in order to show existence of Dirac points for our 2D operator.. 

 Another phenomenon, which is of great interest in Condensed matter physics, in Optics or Acoustics, is the propagation of energy along a line defect or an edge. 
Indeed the presence of a boundary, an interface or more generally a line perturbation in a periodic medium may create energy localization. This is directly linked to the possible presence of discrete spectrum when perturbing a perfectly periodic operator. Such phenomena can be exploited in quantum, electronic or photonic device design. In the mathematical literature, sufficient conditions on the periodic media and the perturbations have been proposed in order to ensure the existence of such localized and guided waves (see for instance \cite{KuchmentOng:2010,BrownHoangPlumWood2014,BrownHoangPlumWood2015,Article_DFJV_Part1}. Existence of edge states in graphene has been first studied in \cite{Nakada:1996,Fujita:1996} where the importance of the shape of the edge has been highlighted (the so-called zigzag and armchair edges were studied). In \cite{Fefferman:2022,Fefferman:2022b} existence of edge states for any "rational" edge has been investigated. Let us also mention \cite{Lee:2019} showing existence of edge states for photonic graphene. {Some  recent results dealing with the existence of edge spectrum for more general interfaces can be found in~\cite{gontier2021spectral, drouot2023topological}: However, the nature of the edge spectrum, in particular the localization (along the interface) of the associated eigenmodes, is not yet understood}.

  The second aim of this paper is to show existence of edge states or guided modes when our domain is perturbed in the zigzag direction. More precisely, we consider the half-space problem  obtained by cutting the periodic domain along the zizgag direction, and we impose either homogeneous Neumann or  homogeneous Dirichlet conditions on the new part of the boundary. Note that the associated edge states correspond respectively to antisymmetric and symmetric guided waves for the mirror symmetrized medium. We first study the classical zigzag edge, that we show to be robust with respect to local perturbations on the thickness of  the rungs near the edge. Then, following  the arguments used for the study of edge states in presence of dislocations in \cite{Gontier:2020}, we are able to study existence of edge states for any position of the cutting (but still in the same direction), going from the zigzag edge to the so-called bearded edge. We recover in particular the unconventional non dispersive  edge states observed in \cite{Plotnik:2014}, and we show that such phenomenon also occurs at high frequencies, for several locations of the cut, the number of locations increasing with the frequency.

 This paper is organized as follows. In Section~\ref{sec:Section1}, we present the problem under consideration (unperturbed and perturbed geometries) and give the main results. Then, Section~\ref{sec:spectrum_optot} is dedicated to the proof of existence of Dirac points. The existence of guided waves is studied Section~\ref{sec_guidedmodes}.  Numerical illustrations are given in Section~\ref{sec:numericalResultSpectreEssentiel} (essential spectrum) and Section~\ref{sec:numericalResultSpectreDiscret} (edge states and guided modes). Technical results are postponed in Appendices.

\section{Model problem}\label{sec:Section1}
\subsection{Geometry of the domains}\label{sub:geometry}
\subsubsection{The infinite periodic graph and the corresponding fatten graph like domain}

Let us first introduce a hexagonal periodic medium $\Omega_\delta$ that consists of the plane $\RR^2$ minus an infinite set of equi-spaced hexagonal perfect conductor obstacles. The distance between  neighboring obstacles is supposed to be small and is denoted $\delta$, see Figure~\ref{fig:omega} where $\Omega_\delta$ lies in the grey region. 
\begin{figure}[h]
	\begin{center}
	 \includegraphics[width=0.9\textwidth,trim=0cm 0cm 0cm 0cm,clip]{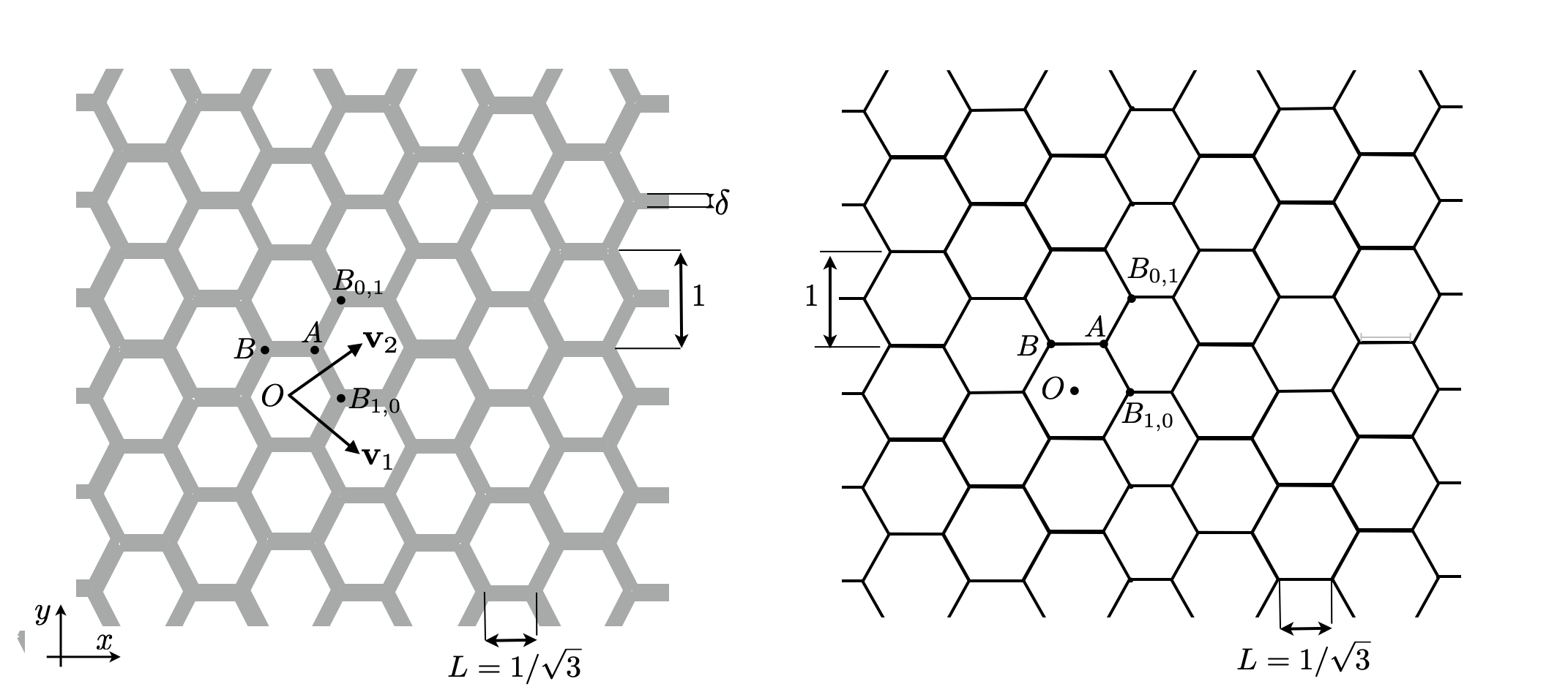}
	 \caption{The hexagonal periodic medium $\Omega_\delta$ (left), the associated quantum graph $\cG$ (right)}
	 \label{fig:omega}
 	\end{center}\vspace{-0.5cm}
\end{figure}

In order to give a precise definition of $\Omega_\delta$, let us first describe the associated quantum graph that we denote $\cG$. We first introduce the two directions of periodicity and the associated Bravais lattice $\Lambda$
  \begin{equation}\label{eq:dir-per}
    \vg_1 := (\frac{\sqrt{3}}{2}, -\frac{1}{2}),\;
    \vg_2 := (\frac{\sqrt{3}}{2}, \frac{1}{2})\quad \text{and}\quad \Lambda:=\Z\vg_1+\Z\vg_2
  \end{equation}
as well as its dual basis $(\vg_1^\ast, \vg_2^\ast)$ defined by $\vg_i \cdot \vg_j^\ast =  \delta_{ij},\,i,j \in \{ 1, 2\}$ the reciprocal lattice $\Lambda^\ast$
 \begin{equation} \label{e1stare2star}
 \vg_1^\ast = (\frac{1}{\sqrt{3}},-1), \quad \vg_2^\ast =  (\frac{1}{\sqrt{3}},1) \quad \text{and}\quad\Lambda^\ast:=\Z\vg_1^\ast+\Z\vg_2^\ast.
 \end{equation}
Let us introduce the two "generator" vertices  
$$
A:= (\frac{L}2,\frac{1}{2}) \mbox{ and } \quad {B} := (-\frac{L}2,\frac{1}{2}), 
$$ 
with $L:=1/\sqrt{3}$ corresponding to the distance between A and B, the set of "A-points", $A+\Z\vg_1+\Z\vg_2$, the set of "B-points", $B+\Z\vg_1+\Z\vg_2$ composed respectively by the points
\[
	\forall m,n\in\Z,\quad A_{m,n}:=A+m\vg_1+n\vg_2,\quad B_{m,n}:=B+m\vg_1+n\vg_2,
\]
and finally the three oriented "generator" edges (see Figure \ref{fig:percell} (left))
\begin{equation}\label{orientedgeneratorEdges}
  \begin{array}{l}
  e_0 = \Big\{ \xg \in \R^2, \; \mbox{s.t.} \;  \xg = A\,(1-{t}/{L}) + B\,{t}/{L}, \quad t \in (0,L)\Big\}, \\
   e_1  = \Big\{ \xg \in \R^2, \; \mbox{s.t.} \;  \xg = A\,(1-{t}/{L}) +B_{1,0}\,{t}/{L}, \quad t \in (0,L)\Big\}, \\
   e_2 = \Big\{ \xg \in \R^2, \; \mbox{s.t.} \;  \xg = A\,(1-{t}/{L})+B_{0,1}\,{t}/{L}, \quad t \in (0,L)\Big\} .
   \end{array}\end{equation}
The periodicity cell $\cg$ is then defined by 
$
\cg := \overline{e_0}  \cup \overline{e_1} \cup \overline{e_2} 
$ 
and the infinite periodic graph $\cG$ is defined as the union of all the translations of the periodicity cell 
\begin{equation}\label{DefintionGrapheG}
\cG := \bigcup_{(n,m) \in \Z^2} \cg + n \vg_1+m\vg_2.
\end{equation}

\noindent Finally, we shall denote by $\mathcal{V}$  the set of the vertices of the graph, i.e. the union of the sets of A-points and B-points
$
	\mathcal{V}:=\{A_{n,m}, B_{n,m},\;n,m\in\Z\},
$ and by  $\mathcal{E}$ the set of its edges
\begin{equation}\label{DefinitionSetofEdges}
	\mathcal{E}:=\{e_j+\Z\vg_1+\Z\vg_2,\;j=0,1,2\}.
\end{equation}
We will introduce functions defined on the graph which, on each edge, can be identified to 1-D functions, using the following parametrization of the edges : for all $n,m\in\Z$
  \begin{equation}\label{eq:param}
  \begin{array}{l}
  e_0+n \vg_1 + m \vg_2 = \Big\{ \xg \in \R^2, \; \mbox{s.t.} \;  \xg = A_{n,m}\,(1 - {t}/{L}) + B_{n,m}\,{t}/{L}, \quad t \in (0,L)\Big\}, \\[2ex]
   e_1+n \vg_1 + m \vg_2  = \Big\{ \xg \in \R^2, \; \mbox{s.t.} \;  \xg = A_{n,m}\,(1-  {t}/{L})) +B_{n+1,m}\,{t}/{L}, \quad t \in (0,L)\Big\}, \\[2ex]
   e_2+n \vg_1 + m \vg_2 = \Big\{ \xg \in \R^2, \; \mbox{s.t.} \;  \xg = A_{n,m}\,(1- {t}{L})+B_{n,m+1}\,{t}/{L}, \quad t \in (0,L)\Big\} .
   \end{array}\end{equation}
   In the sequel, the identification of two functions defined in different edges is done using this parametrization.
   \\\\
   Finally, the domain $\Omega_\delta$ for $\delta$ small enough is defined as
  \begin{equation}\label{eq:Omega_delta}
  \Omega_\delta := \left\{ \xg = (x, y) \in \R^2, d(\xg, \cG  ) <\delta \right\},
  \end{equation}
  where $d$ denotes the euclidian distance.
\\\\
  The particularity of $\cG$ and $\Omega_\delta$ is that they admit the so-called honeycomb symmetry defined as follows
   \begin{definition}[The honeycomb symmetry]\label{def:honeycomb}
   Let $\mathcal{O}\subset\R^2$. We say that $\mathcal{O}$ satisfies a honeycomb symmetry if 
   \begin{enumerate}
   \item $\mathcal{O}$ is periodic in the $\vg_1$ and $\vg_2$ directions :
   $
    \mathcal{O}+\vg_1=\mathcal{O}+\vg_2=\mathcal{O}.
   $
   \item $\mathcal{O}$ is stable over the symmetry $S$ with respect to the origin $(0,0)$, i.e. \begin{equation}\label{symmetry}S:{\bf x}\mapsto -{\bf x}.\end{equation} More precisely,
   $
   	\forall {\bf x}\in \R^2,\quad {\bf x}\in\mathcal{O}\;\Leftrightarrow S{\bf x}\in\mathcal{O}.
   $ 
   \item $\mathcal{O}$ is stable over the rotation $R$ of center $(0,0)$ and angle $2 \pi /3$, i.e.s
   \begin{equation}\label{MatriceRotation}
    R:{\bf x} \mapsto \left[  \begin{matrix} \cos(2 \pi/3) &   -\sin(2 \pi/3)  \\
    \sin(2 \pi/3) \ & \cos(2 \pi/3) 
     \end{matrix} \right]{\bf x}
   \end{equation}
   More precisely, 
   $
   	\forall {\bf x}\in \R^2,\quad {\bf x}\in\mathcal{O}\;\Leftrightarrow R{\bf x}\in\mathcal{O}.
   $
   \end{enumerate}
   \end{definition}
   We can then introduce natural linear transformations acting on functions defined in open sets with honeycomb symmetry. In the following, $L^2_\text{loc}(\mathcal{O})$ stands for the set of functions which are locally $L^2$.
   \begin{definition}
    Let $\mathcal{O}\subset\R^2$ with a honeycomb symmetry. We define 
    \begin{enumerate}
      \item the symmetry operator $\mathcal{S}: L^2_\text{loc}(\mathcal{O}) \rightarrow L^2_\text{loc}(\mathcal{O})$ defined by 
      \begin{equation}\label{DefinitionSym}
     \forall u\in  L^2_\text{loc}(\mathcal{O}),\quad\mathcal{S}u({\bf x})=u(S{\bf x}),\; {\bf x}\in \mathcal{O}.
      \end{equation} 
      where $S$ is the symmetry transformation defined in \eqref{symmetry};
      \item the rotation operator $\mathcal{R}:L^2_\text{loc}(\mathcal{O})\rightarrow L^2_\text{loc}(\mathcal{O})$ defined by
      \begin{equation}\label{eq:rotation_op}
        \forall u\in L^2_\text{loc}(\mathcal{O}),\quad\mathcal{R}u({\bf x})=u(R^*{\bf x}),\; {\bf x}\in \mathcal{O}
      \end{equation}
      where $R^*$ is the adjoint of the rotation $R$ defined in~\eqref{MatriceRotation}.
    \end{enumerate} 
   \end{definition}
   Let us note that these transformations depend obviously on $\mathcal{O}$ (typically $\cG$ and $\Omega_\delta$), but in this paper, we will use abusively the same notation $\mathcal{S}$ and $\mathcal{R}$ for any $\mathcal{O}$.
  
  Note finally that since $\mathcal{R}^3=\mathcal{I}$ where $\mathcal{I}$ stands for the identity operator, $\mathcal{R}$ is unitary with eigenvalues $e^{2 \imath \pi s \theta}$  where $s\in\{0,1,2\}$ and the associated eigenspaces are defined by
    \begin{equation}\label{eq:L2s}
      \forall s\in\{0,1,2\},\quad L^2_s(\mathcal{O}):=\{u\in L^2_\text{loc}(\mathcal{O}),\; \mathcal{R}u=e^{2 \imath \pi  s\theta}u \}.
    \end{equation}

Let us now introduce  ${\cal C}_\delta^\sharp$ the periodicity cell of $\Omega_\delta$ which is the union of the three fattened versions of the edges $e_0,\,e_1,\,e_2$, 
$
   	{\cal C}_\delta^\sharp=e_{0,\delta}\cup e_{1,\delta}\cup e_{2,\delta}
    $
   where $e_{0,\delta}$ is the polygon delimited by $A$, $A+\delta(-\sqrt{3}/2,1/2)$, $B+\delta(\sqrt{3}/2,1/2)$,\Becommente{ $B$,} $B+\delta(\sqrt{3}/2,-1/2)$ and 
   $A+\delta(-\sqrt{3}/2,-1/2)$,  $e_{1,\delta}:=Re_{0,\delta}+\vg_2$ and $e_{2,\delta}:=R^*e_{0,\delta}+\vg_2-\vg_1=R^*e_{1,\delta}+\vg_2$ where $R$ is the rotation defined in 
   \eqref{MatriceRotation}. In what follows, we will identify functions defined on $e_{j,\delta},\,j\in\{0,1,2\}$ in the following sense
   \begin{equation}\label{identify_fatedges}
    \begin{array}{l}
 u\in L^2(e_{0,\delta}),\,v\in L^2(e_{1,\delta}),\quad u=v\quad\Leftrightarrow \quad u({\bf x})=v(R{\bf x}+\vg_2), \; {\bf x}\in \mathcal{O},\\
 u\in L^2(e_{0,\delta}),\,v\in L^2(e_{2,\delta}),\quad u=v\quad\Leftrightarrow \quad u({\bf x})=v(R^*{\bf x}+\vg_2-\vg_1), \; {\bf x}\in \mathcal{O},
\\
 u\in L^2(e_{1,\delta}),\,v\in L^2(e_{2,\delta}),\quad u=v\quad\Leftrightarrow \quad u({\bf x})=v(R{\bf x}+\vg_2), \; {\bf x}\in \mathcal{O}.
    \end{array}
   \end{equation}
In other words, with this identification, we keep the parametrization from a A-point to a B-point, as in \eqref{orientedgeneratorEdges}.\vspace{-0.5cm}
\begin{figure}[h]
	\begin{center}
	 \includegraphics[width=8 cm]{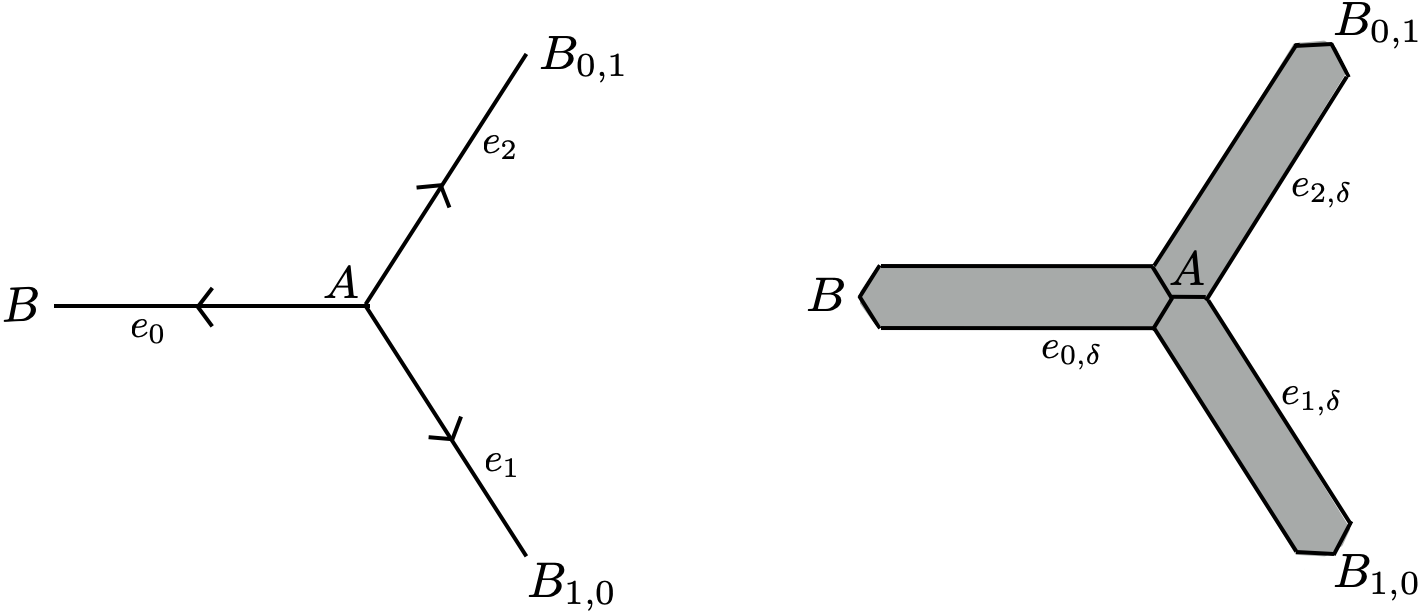}
	 \caption{ The periodicity cell $\cg$ and its three oriented edges (left); The periodicity cell ${\cal C}_\delta^\sharp$ and the three fattened edges (right)}
	 \label{fig:percell}
 	\end{center}
\end{figure}\vspace{-1cm}

\subsubsection{Zigzag {truncated} domains}\label{sec:DefinitionZigzagPerturbeddomain}

We want to study the existence of guided modes or edge modes when cutting the domain in the certain direction and possibly perturbing it along its boundary.  The cutting direction is in this paper the so-called zigzag direction, as in \cite{Nakada:1996,Fujita:1996,Plotnik:2014}: the $\vg_1$-direction or equivalently, the $\vg_2$-direction or equivalently  the vertical direction ${\bf e}_y=\vg_2-\vg_1$. Without loss of generality, we focus on this last case in the paper, see Figure~\ref{fig:lineicdefect}. Let us now describe the different perturbed configurations we shall focus on.\\

{\bf Case 1.} The first \Becommente{truncated} domain is the half-domain with a classical zigzag edge. It consists in cutting our domain at the abscissa $x = -L$ along the direction
\begin{equation}\label{eq:defvgz}
 \vg_2 -\vg_1  = \mathbf{e}_y.
\end{equation}
We obtain the domain $\Omega_{\delta}^{0}$ defined by
\begin{equation}\label{eq:defOmegaDelta0}
\Omega_{\delta}^{0}:=\Omega_\delta\cap\{x> -L\},
\end{equation}
and the corresponding truncated graph $\mathcal{G}_0$ given by
\begin{equation}\label{eq:defGraph0}
\mathcal{G}_0:=\mathcal{G}\cap\{x> -L\}.
\end{equation}
The edge is said to be zigzag because of the shape of the lateral edges of  $\mathcal{G}_0$ (namely the union of the edges $[A_{m-1,-m}, B_{m, -m}]$ and $[ B_{m, -m}, A_{m,-(m+1)}]$ for $m\in \Z$). 

{\bf Case 2.} The second \Becommente{truncated and perturbed} domain we consider is a perturbation of $\Omega_{\delta}^{0}$ obtained by modifying the width of those zigzag edges from $\delta$ to $\mu \delta$ (for a given positive parameter $\mu$). We obtain the domain $\Omega_{\delta,\mu}^{0}$ defined by
\begin{equation}\label{eq:defOmegaDelta0mu}
\Omega_{\delta,\mu}^{0}:= \left\{ \xg = (x, y) \in \R^2, d(\xg , \mathcal{G}_0) \leq d_\mu(\xg) \delta \right\},
\end{equation}
where the function $d_\mu : \R^2 \mapsto \R$ is defined by
\begin{equation}\label{eq:d_mu}
 d_\mu(x, y) = \begin{cases}
 \mu & \mbox{ if }  x < -\frac{L}{2}, \\
 1 &  \mbox{ otherwise. }  
 \end{cases} 
\end{equation}

{\bf Case 3.} Finally, in our third perturbation, we still cut the domain in the ${\bf e}_y$-direction but the cut location changes. For this, we introduce the parameter  $t \in [0, 2L]$ and we define
\begin{equation}
\Omega_{\delta}^{t}:=\Omega_\delta\cap\{x> \alpha(t)\},
\end{equation}
where
\begin{equation}\label{eq:defalphat}
\alpha(t) = \begin{cases}
\dsp -L + \frac{t}{2} & \text{ if } t \in [0, L],\\
\dsp -\frac{L}{2} + (t-L) & \text{ if } t  \in [L, 2L].\\
\end{cases} 
\end{equation}
The associated truncated graph $\mathcal{G}_t$ is defined as
 \begin{equation}\label{eq:defGrapht}
\mathcal{G}_t:=\mathcal{G}\cap\{x>  \alpha(t)\}.
\end{equation}
We remark that $\mathcal{G}_0$ (resp. $\Omega_{\delta}^{0}$) coïncides with $\mathcal{G}_{2L}$ (resp. $\Omega_{\delta}^{2L}$) up to a translation of vector $\vg_2$ or $\vg_1$.
Note that the structure of the edge is completely different if $t\in [0,L]$ or $t\in[L,2L]$. In Condensed Matter Physics literature, $\mathcal{G}_L$ is often called the "bearded zigzag edge" \cite{Fujita:1996,Nakada:1996}. Note that compared to the tight binding model where only two zigzag edges can be considered (ordinary and bearded), in our case, a family of zigzag edges, parametrized by $t$ is relevant.
\begin{figure}[h]
	\begin{center}
	 \includegraphics[width=\textwidth]{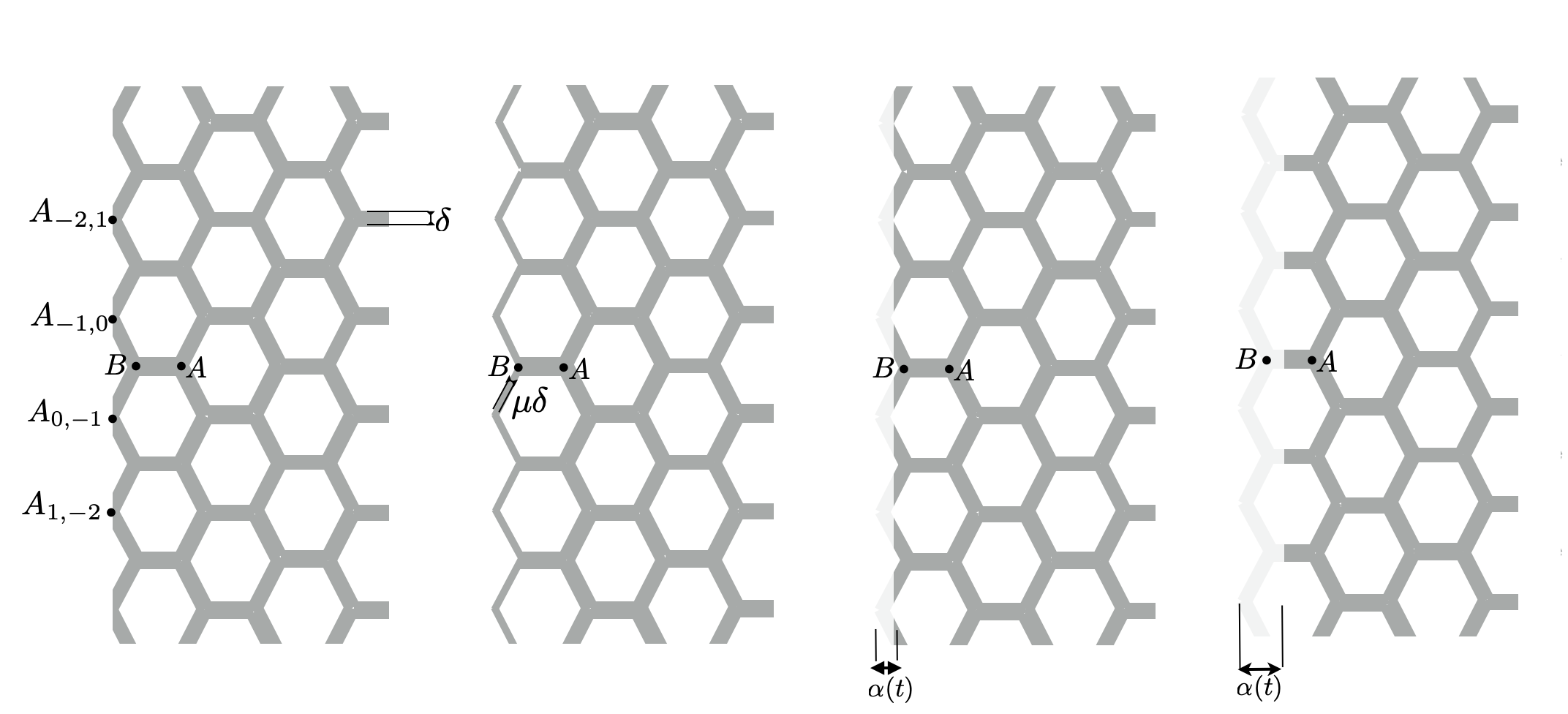}
	 \caption{The zigzag perturbed domains: (Case 1) the half-domain with a classical zigzag edge $\Omega_\delta^0$ (left), (Case 2) the perturbed half-domain $\Omega_{\delta,\mu}^{0}$ for $\mu<1$ (second figure) and (Case 3) the half-domain with a different cut position $\Omega_{\delta}^{t}$ for $t\in (0,L)$ (third figure) and $\Omega_{\delta}^{t}$ for $t\in (L,2L)$ (right). See \eqref{eq:defalphat} for the definition of $\alpha(t)$.}\vspace{-0.5cm}
	  	 \label{fig:lineicdefect}
 	\end{center}
\end{figure}

For the sake of brevity, the domains  $\Omega_{\delta,\mu}^{0}$ and $\Omega_{\delta}^{t}$ are denoted $\Omega_{\delta,\mu}^{t}$ in the rest of this section.
For any $t$ and $\mu$, the domain $\Omega_{\delta,\mu}^{t}$ is 1-periodic in the ${\bf e}_y$-direction. We denote $\widehat\Omega_{\delta,\mu}^{t}:=\Omega_{\delta,\mu}^{t}\cap\{ -1/2<y<1/2\}$ one of its period. 
Besides, we denote by  $\Gamma_t^\delta$ the "edge" part of $\partial \Omega_{\delta,\mu}^{t}$ located on the truncated interface: 
\begin{equation}\label{definitionGammat}
\Gamma_t^\delta = \partial \Omega_{\delta,\mu}^{t} \cap \{x = \alpha(t) \}.
\end{equation} 

\subsection{Mathematical formulation}\label{sec:mathFormulation}
We are interested in the existence of  guided modes or edge states, that is to say solutions of the homogeneous wave equation propagating along the edge (i.e the left boundary), (see e.g. \cite[Section 2]{Article_DFJV_Part1}). In other words, for a fixed wavenumber $\beta\in\RR$, we look for couples $(u_{\delta,\mu}^{t}(\beta),\lambda_{\delta,\mu}^{t}(\beta))\in H^1_\text{loc}(\Omega_{\delta,\mu}^{t})\times\RR^+$ such that
\begin{equation}\label{eq:guidedmodes}
\begin{cases}
	-\triangle u_{\delta,\mu}^{t}=\lambda_{\delta,\mu}^{t}\,u_{\delta,\mu}^{t}&\text{in}\;\Omega_{\delta,\mu}^{t},
	\\
	\partial_n u_{\delta,\mu}^{t}=0&\text{on}\;\partial\Omega_{\delta,\mu}^{t} \setminus \Gamma_t^\delta,\\
	\partial_n u_{\delta,\mu}^{t} = 0 \; \mbox{ or }  \; u_{\delta,\mu}^{t} = 0  & \mbox{ on }\Gamma_t^\delta,
	\end{cases}
\end{equation}
such that $u_{\delta,\mu}^{t}\in H^1(\widehat\Omega_{\delta,\mu}^{t})$ and $u_{\delta,\mu}^{t}$ is $\beta$-quasi-periodic in the ${\bf e}_y$-direction, which means
\begin{equation}\label{eq:guidedmodes2}
	\forall {\bf x}\in\Omega_{\delta,\mu}^{t},\quad u_{\delta,\mu}^{t}({\bf x}+{\bf e}_y)=e^{2\imath\pi\beta}\,u_{\delta,\mu}^{t}({\bf x}).
\end{equation}
 By periodicity, it is easy to see that it suffices to consider $\beta\in(-1/2,1/2]$. Moreover,  if $(u,\lambda)$ satisfies \eqref{eq:guidedmodes} with $u$ $\beta$-quasi-periodic then $(\overline{u},\lambda)$ satisfies also \eqref{eq:guidedmodes}, $\overline{u}$ being $(-\beta)$-quasi-periodic.Therefore, we only have to consider $\beta\in[0,1/2]$. 

 We emphasize (again) that imposing $\partial_n u_{\delta,\mu}^{t} = 0$ on $\Gamma^\delta_t$ yields to consider symmetric guided modes in the domain $\mathcal{T}^\delta$ obtained by attaching to  $\Omega_{\delta,\mu}^{t}$ its image by mirror symmetry while imposing $ u_{\delta,\mu}^{t} = 0$ on $\Gamma^\delta_t$ yields to consider antisymmetric ones.

 For all $\beta\in[0,1/2]$, this problem is linked to the discrete spectrum of the self-adjoint and non-negative operators ${N}_{\delta,\mu}^{t}(\beta)$  and  ${D}_{\delta,\mu}^{t}(\beta)$ defined as follows:
\begin{equation}
	\label{eq:opN_perturbe_def}
	\begin{array}{|l}
	{N}_{\delta,\mu}^{t}(\beta) :  u  \rightarrow -\triangle u,\\[5pt]
	 D({N}_{\delta,\mu}^{t}(\beta))=\{u\in H^1(\triangle,\widehat\Omega_{\delta,\mu}^{t})\cap  H^1_\text{\text{loc}}(\triangle,\Omega_{\delta,\mu}^{t}),\\[5pt]
   \hspace{2 cm} \partial_n u=0\text{ on }\partial\Omega_{\delta}^{\mu,t}, \quad
	 u\text{ satisfies \eqref{eq:guidedmodes2}}\},
	 \end{array}
\end{equation}
where $H^1(\triangle,\mathcal{O}):=\{u\in H^1(\mathcal{O}),\;\triangle u \in L^2(\mathcal{O})\}$, for any open subset $\mathcal{O}$ of $\R^2$ and
\begin{equation}
	\label{eq:opD_perturbe_def_D}
	\begin{array}{|l}
	{D}_{\delta,\mu}^{t}(\beta) : u  \rightarrow -\triangle u,\\[5pt]
	 D({D}_{\delta,\mu}^{t}(\beta))=\{u\in H^1(\triangle,\widehat\Omega_{\delta,\mu}^{t})\cap  H^1_\text{\text{loc}}(\triangle,\Omega_{\delta,\mu}^{t}),\\[5pt]
   \hspace{2 cm} \partial_n u=0\text{ on }\partial\Omega_{\delta}^{\mu,t} \setminus \Gamma_t^\delta, \quad  u=0\text{ on } \Gamma_t^\delta, \quad
	 u\text{ satisfies \eqref{eq:guidedmodes2}}\},
	 \end{array}
\end{equation}
Note that in the definition of the operators ${N}_{\delta,\mu}^{t}(\beta)$  or  ${D}_{\delta,\mu}^{t}(\beta)$, the only difference is in the boundary conditions on $\Gamma_t^\delta$ (homogeneous Neumann boundary conditions for ${N}_{\delta,\mu}^{t}(\beta)$ and homogeneous Dirichlet ones for ${D}_{\delta,\mu}^{t}(\beta)$). 
% \Becommente{
% \noindent In this work, we shall exhibit conditions on the perturbations that ensure existence of a discrete spectrum for the operators ${N}_{\delta,\mu}^{t}(\beta) $ or ${D}_{\delta,\mu}^{t}(\beta) $. Remark that we have performed a similar study for square lattices with a lineic perturbation \cite{Article_DFJV_Part1}: only the distance between two consecutive obstacle in a line is changed from $\delta$ to $\mu\delta$ (similarly to case 2). The condition on the perturbation for existence of guided  modes is simple: if $\mu\in(0,1)$, for $\delta$ small enough, the operator  exist guided modes for any $\beta$ which are dispersive. In this work, the result is different : there exists guided modes in Case 2 for any $\mu$, these guided modes being almost non dispersive. The differences might be due to the presence of Dirac points (see Definition~\ref{DefinitionDirac})  in the essential spectrum of the underlying operator. \\
% }

The essential spectrum of  both ${N}_{\delta,\mu}^{t}(\beta)$ and ${D}_{\delta,\mu}^{t}(\beta) $ is linked to the essential spectrum of the operator defined on the whole hexagonal periodic domain $\Omega_\delta$, namely
 \begin{equation}
 	\label{eq:opA_def}
 	A_{\delta}=-\triangle,\quad
 	 D(A_{\delta})=\{u\in H^1(\triangle,\Omega_\delta),\; \partial_n u|_{\partial\Omega_\delta}=0
 	 \}.
 \end{equation}
 \subsection{Main results}

 The first main result is, that for any $\delta$, the operator	$A_{\delta}$ has a certain number of Dirac points located at different frequencies (see figure~\ref{fig:dispersionSurf1} for a numerical illustration). Let us define, for all $n\in \N$,
 \begin{equation}\label{eq:lambda_n_star}
  \omega_n^*:= \frac{\pi}{2L}+\frac{n\pi}{L},\quad \mbox{ and } \quad \lambda_n^*:=(\omega_n^*)^2.
\end{equation}
 \begin{theorem}\label{th:Diracpoints_delta}
  There exists $\delta_0$ such that for all $\delta<\delta_0$, \Becommente{there exists $M_{\delta} \in \mathbb{N}$ such that} the  spectrum of the operator $A^{\delta}$ contains $M_\delta>0$ Dirac points located near $
    \lambda_n^*$ with $0\leq n\leq M_\delta$. \Becommente{Moreover, $\lim_{\delta \rightarrow 0}M_{\delta}=+\infty$}.
\end{theorem}
We deduce from the previous result the presence of 'gaps' (\emph{i.e.} intervals included in the complementary of the essential spectrum) in the essential spectrum of ${N}_{\delta,\mu}^{t}(\beta) $ and of ${D}_{\delta,\mu}^{t}(\beta)$), see figure~\ref{fig:VpPlateZigZag} for a numerical illustration of the essential spectrum with respect to $\beta$.
\begin{theorem}\label{th:ess_spec_beta_delta}
  For all $\beta\in[0,1/2]\setminus\{\frac{1}3\}$, there exists $\delta_0>0$ such that for all $\delta<\delta_0$, $\exists M_\delta \in \N$ such that for all $0\leq n\leq M_\delta$,  for any $t\in[0,2L]$,  for any $\mu>0$, there exists a gap $I^n_\delta(\beta)$ containing $\lambda_n^*$ in the spectrum of ${N}_{\delta,\mu}^{t}(\beta) $ and ${D}_{\delta,\mu}^{t}(\beta)$.
\end{theorem} 
We can then study, for any $\beta\in[0,1/2]\setminus\{\frac{1}3\}$ and any $\delta$ small enough, the existence of eigenvalues of the operator ${N}_{\delta,\mu}^{t}(\beta)$ and ${D}_{\delta,\mu}^{t}(\beta)$ in $I^n_\delta(\beta)$, see figure~\ref{fig:VpPlateZigZag} for a numerical illustration of the eigenvalue with respect to $\beta$ and figures~\ref{fig:VecteurPropreZigZag} and \ref{fig:VecteurPropreZigZagDirichlet} for illustrations of eigenvectors.
\begin{theorem}[Existence of guided modes of ${N}_{\delta,\mu}^{t}(\beta) $  and  ${D}_{\delta,\mu}^{t}(\beta) $ for Case 1 and Case 2]\label{th:guided_case12_delta}
  Let $t=0$ and $\mu>0$.\vspace{-0.1cm}
  \begin{itemize}
  \item For all $\beta\in(\frac{1}3,\frac{1}2)$, there exists $\delta_1\leq \delta_0$, where $\delta_0$ is the one of Theorem \ref{th:ess_spec_beta_delta},  such that for all $\delta<\delta_1$ and for all $0\leq n\leq M_\delta$, the operator ${N}_{\delta,\mu}^{t}(\beta)$, defined in \eqref{eq:opN_perturbe_def}, has an eigenvalue $\lambda_{\delta,\mu}^{n,N}(\beta)$ which is at first order independent of $\beta$ and $\mu$. More precisely, there exists a constant $C(n,\beta,\mu)$ which depends on $n$, $\beta$ and $\mu$ such that
    \[
     | \lambda_{\delta,\mu}^{n,N}(\beta) - \lambda_n^* |\leq  C({n,\beta,\mu})\,\sqrt{\delta}.
    \] 
     \item For all $\beta\in(0,\frac{1}3)$, there exists $\delta_1\leq \delta_0$, where $\delta_0$ is the one of Theorem \ref{th:ess_spec_beta_delta},  such that for all $\delta<\delta_1$ and for all $0\leq n\leq M_\delta$, the operator ${D}_{\delta,\mu}^{t}(\beta)$, defined in \eqref{eq:opN_perturbe_def}, has an eigenvalue $\lambda_{\delta,\mu}^{n,D}(\beta)$ which is at first order independent of $\beta$ and $\mu$. More precisely, there exists a constant $C(n,\beta,\mu)$ which depends on $n$, $\beta$ and $\mu$ such that
    \[
     | \lambda_{\delta,\mu,n}^{n,D}(\beta) - \lambda_n^* |\leq  C({n,\beta,\mu})\,\sqrt{\delta}.
    \] 
     \end{itemize}
\end{theorem}
\Becommente{For the Case 3 ($\mu=1$), the point of view is different. Whereas in the previous result, for a fixed truncation $t=0$, we study the dispersion curves $\beta \mapsto \lambda_n(\beta)$, for a fixed $\beta$ and  for a fixed value $t$, we exhibit the  values of $t$ for which $\lambda$ is an eigenvalue of the operator ${D}_{\delta,\mu}^{t}(\beta)$  (resp. ${N}_{\delta,\mu}^{t}(\beta)$).}

\begin{theorem}[Existence of guided modes for Case 3]\label{th:guided_case3_delta}
  Let $t\in[0,2L]$, $\mu=1$ and $\beta\in(0,\frac{1}2)\setminus\{\frac{1}3\}$. Suppose that $\delta\leq \delta_0$, where $\delta_0$ is given in Theorem \ref{th:ess_spec_beta_delta}.\vspace{-0.1cm}
  \begin{itemize}
  \item  Let $\lambda\in I^n_\delta(\beta)\setminus\lambda_n^* $. There exists $\delta_1\leq \delta_0$, such that, for any $\delta< \delta_1$, there exist (at least) $(2n+1)$ values of $t$ (depending on $\beta$), $t^D_{\delta,1}(\beta), \ldots, t^D_{\delta,2n+1}(\beta)$ (resp. $t^N_{\delta,1}(\beta), \ldots, t^N_{\delta,2n+1}(\beta)$) such that $\lambda$ is an eigenvalue of the operator ${D}_{\delta,\mu}^{t}(\beta)$ (resp. ${N}_{\delta,\mu}^{t}(\beta)$).
   \item   Let $\lambda = \lambda_n^* $. There exists $\delta_1\leq \delta_0$ such that, for any $\delta< \delta_1$,  there exist (at least) $2n$ values of $t$ depending on $\beta$,  $t^D_{\delta,1}(\beta), \ldots, t^D_{\delta,2n}(\beta)$ (resp. $t^N_{\delta,1}(\beta), \ldots, t^N_{\delta,2n}(\beta)$) such that $\lambda$ is an eigenvalue of the operator ${D}_{\delta,\mu}^{t}(\beta)$ (resp. ${N}_{\delta,\mu}^{t}(\beta)$). Moreover,  for $J\in\{D,N\}$, there exist  2 pairs of $2n$ points  $(t^J_{\pm,1}$, $\cdots$,  $t^J_{\pm,n})$  such that
   $$
   \lim_{\delta \rightarrow 0} {t^J_{\delta,k}}(\beta)  = \begin{cases}
   t_{-,k}^J & \mbox{ if } \beta \in (0, 1/3), \\
   t_{+,k}^J & \mbox{ if } \beta \in (1/3, 1/2)
   \end{cases} 
   $$ 
  \end{itemize}
\end{theorem}
\Becommente{
\begin{remark}
Theorem~\ref{th:guided_case3_delta} proves, for given $\lambda$ and $\delta$, the existence of values of $t$ such that $\lambda$ is an eigenvalue of the operator ${D}_{\delta,\mu}^{t}(\beta)$(or ${N}_{\delta,\mu}^{t}(\beta)$). We can  'invert' the relation, and consider the $2n+1$ curves $t \mapsto \lambda(t)$, see Figure~\ref{fig:FlotSpectraux} where those curves  (for ${N}_{\delta,\mu}^{t}(\beta)$) are represented. Finally for a given truncation $t$, we could represent the dispersion curves $\beta \mapsto \lambda_n(\beta)$ and exhibit almost flat curves, see Figure \ref{fig:FlatEigenvalues}. The results of Case 1 and 2 (see Theorem \ref{th:guided_case12_delta}) can be extended for other truncations.
\end{remark} 
}
  The three previous theorems are proved using a standard approach of asymptotic analysis. In a nutshell, we first identify the limit of the operators ${N}_{\delta,\mu}^{t}(\beta) $ (resp.  ${D}_{\delta,\mu}^{t}(\beta)$)  as $\delta$ tends to 0. Then, we make explicit computations of the spectrum of the limit operators. Standard results of~\cite{KuchmentQuantumgraphs1,Post:2006} (see also ~\cite{Article_DFJV_Part1} for an application to square graph-like domains) ensures the convergence of the spectrum of ${N}_{\delta,\mu}^{t}(\beta) $ (resp. ${D}_{\delta,\mu}^{t}(\beta)$) to the one of the limit operator.  
  \Becommente{
 \begin{remark} In that paper, we restrict ourselves to the Laplacian operator. The extension of the previous result to the operators of the form $-\rho^{-1} \Delta$ is an interesting question. The result of presence of Dirac points could be derived using  the analysis of~\cite{KuchmentPost} (for the limit operator with a potential) together with asymptotic analysis. However, for the edge states, the extension of the existence result is less clear and has to be investigated.
 \end{remark} 
 }
Theorem~\ref{th:Diracpoints_delta} is proven in Section~\ref{sec:spectrum_optot}, while Theorems~\ref{th:ess_spec_beta_delta}-\ref{th:guided_case12_delta}-\ref{th:guided_case3_delta} are proven in Section~\ref{sec_guidedmodes}.
 \section{Spectrum of the operator $A_\delta$ (proof of Theorem~\ref{th:Diracpoints_delta})}\label{sec:spectrum_optot}
 \subsection{Band structure of the spectrum and the hexagonal Brillouin zone}
The operator $A_\delta$ defined in \eqref{eq:opA_def}
 is self-adjoint and non negative. The Floquet-Bloch theory shows that the spectrum of this periodic elliptic operator is reduced to its essential spectrum which has a band structure \cite{Eastham:1973,Kuchment:1993,Reed:1972}. Let us recall this result.
 
For a fixed ${\bf k}=(k_x,k_y)\in\RR^2$, let us define the set of  locally $L^2$ functions which are ${\bf k}\cdot \vg_i$ quasi-periodic in the direction $\vg_i$ for $i\in\{1,2\}$
 \begin{equation}
 L^2_{\kg}(\Omega_\delta) = \left\{ f \in L^2_{loc}(\Omega_\delta) \mbox{ s.t. }  f(\cdot + \vg) = f e^{2\imath\pi \kg \cdot  \vg} \; \forall \vg \in \Lambda  \right\}.
 \end{equation}
It is easy to see that this space can be identified to $L^2({\cal C}_\delta^\sharp)$ through the ${\bf k}$-quasi-periodic extension operator ${E}_{\bf k}:L^2({\cal C}_\delta^\sharp)\rightarrow  L^2_{\kg}(\Omega_\delta)$ defined by
\begin{equation}\label{eq:extensionop}
	 \forall f\in L^2({\cal C}_\delta^\sharp),\; \forall {\bf x}\in {\cal C}_\delta^\sharp,\;\forall \vg \in \Lambda,\;\;[{E}_{\bf k}f]({\bf x} + \vg) = f({\bf x}) e^{2\imath\pi \kg \cdot  \vg} \; 
\end{equation}
and we have
\[
	f\in L^2_{\kg}(\Omega_\delta)\quad\Leftrightarrow\quad f={E}_{\bf k}[f|_{{\cal C}_\delta^\sharp}]
\]
 We  equip $ L^2_{\kg}(\Omega_\delta)$ with the scalar product of $L^2({\cal C}_\delta^\sharp)$ and the associated norm. 
Let us define the set of  locally $H^1$ functions which are ${\bf k}\cdot \vg_i$ quasi-periodic in the direction $\vg_i$ for $i\in\{1,2\}$
  \begin{equation}\label{eq:H1_k}
  H^1_{\kg}(\Omega_\delta) = \left\{ f \in L^2_{\kg}(\Omega_\delta) \mbox{ s.t. }   \nabla f \in L^2_{\kg}(\Omega_\delta)^2  \right\}.
  \end{equation}
  The space $H^1_{\kg}(\Omega_\delta)$ is a closed subspace of $H^1(\Omega_\delta)$ so we equip $H^1_{\kg}(\Omega_\delta)$ with the scalar product of $H^1({\cal C}_\delta^\sharp)$ and the associated norm. 
  Finally, let us introduce the space
  \begin{equation}
  H^1_{\kg}(\Omega_\delta,\Delta) = \left\{ f \in H^1_{\kg}(\Omega_\delta) \mbox{ s.t. }   \Delta f \in L^2_{\kg}(\Omega_\delta)  \right\},
  \end{equation}
 which, for the same reason than for the previous spaces, can be equipped with the scalar product of $H^1({\cal C}_\delta^\sharp,\Delta)$ and the associated norm.
We introduce now the 'reduced' operator $A_\delta(\kg)$ defined as follows
 \begin{equation}
 A_\delta(\kg) = - \Delta ,\quad  
 {D}(A_\delta(\kg))=\left\{  v \in H^1_{\kg}(\Omega_\delta,\Delta),\partial_n v = 0 \; \mbox{on} \; \partial \Omega_\delta  \right\}.
 \end{equation}
For any $\kg \in \R^2$, the operator $A_\delta(\kg)$ is self-adjoint, non negative, and has a compact resolvent. Consequently, its spectrum consists of an increasing sequence of non-negative eigenvalues $(\lambda_{\delta,n}(\kg))_{n\in \N^\ast}$ that  tends to $+\infty$ as $n$ tends to $+\infty$. The mappings ${\bf k}\mapsto \lambda_{\delta,n}(\kg)$ are called the dispersive surfaces, they are Lipschitz-continuous functions (which can be shown by using a min-max characterization of the eigenvalues).  By definition of the dual basis, we have
\[
	\forall \kg \in \R^2,\;\forall (m,n)\in\Z^2,\quad L^2_{\kg+m \vg_1^\ast + n \vg_2^\ast}(\Omega_\delta)= L^2_{\kg}(\Omega_\delta)
\]
a similar property holding also for $H^1_{\kg}(\Omega_\delta)$ and $H^1_{\kg}(\Omega_\delta,\Delta)$. This implies that
\[
	\forall \kg \in \R^2,\;\forall (m,n)\in\Z^2,\quad A(\kg + m \vg_1^\ast + n \vg_2^\ast) = A(\kg).
\]
Hence, it suffices to consider the vectors $\kg$ varying over a periodicity cell. A natural choice could be to consider the parallelogram $\{\kg\in \R^2, \kg=k_1\vg_1^\ast +k_2\vg_2^\ast,\; k_1,k_2\in (-1/2,1/2)\}$. But in order to take advantage of the rotation and symmetry property, a more common choice is the so-called  Brillouin zone $\mathcal{B}$ consisting, here, on a regular hexagon containing the points ${\bf k}\in\R^2$ such that $\kg+\Lambda^*$ is invariant by $R$ and that are closer to the origin (see Figure~\ref{fig:Brillouin}). 
\begin{figure}[htbp]
  \centering\vspace{-0.3cm}
  \includegraphics[width=3cm]{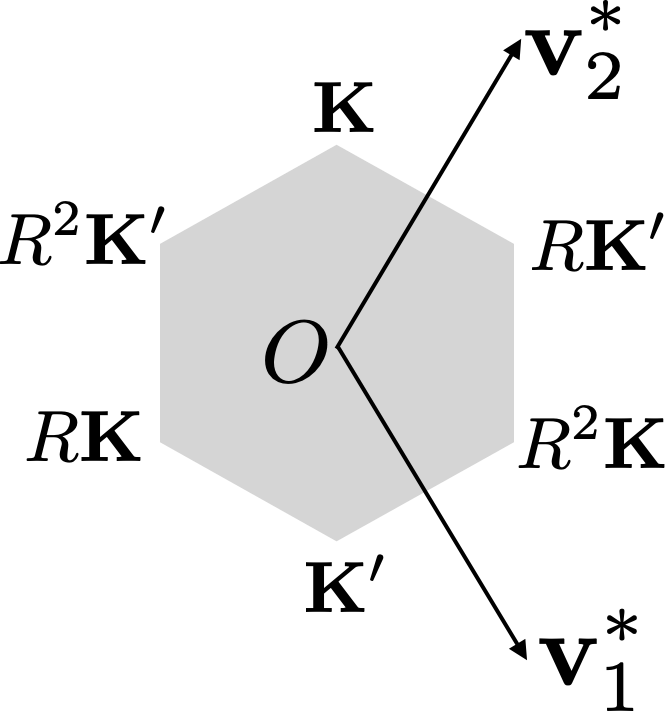}
  \caption{The hexagonal Brillouin zone $\mathcal{B}$ and its 6 vertices. }\label{fig:Brillouin}\vspace{-0.3cm}
\end{figure}
The 6 vertices delimiting the Brillouin zone are defined by
 \begin{equation}\label{eq:vertex_B}
 {\bf K}:= \frac{1}3(\vg_2^\ast-\vg_1^\ast),\; R{\bf K},\; R^2{\bf K},\; {\bf K}':= -{\bf K},\; R{\bf K}',\; R^2{\bf K}',
 \end{equation}
 Note that, since $R{\bf K} = {\bf K}-\vg_2^\ast$, $R^2{\bf K} = {\bf K}+\vg_1^\ast$,  $R{\bf K}' = {\bf K}'+\vg_2^\ast$ and $R^2{\bf K}' = {\bf K}'-\vg_1^\ast$, we have 
 \begin{equation}\label{eq:prop_vertex}
  \begin{array}{l}
     L^2_{{\bf K}^*}(\Omega_\delta)= L^2_{{\bf K}}(\Omega_\delta)\quad\text{and}\quad A({\bf K}^*)=A({\bf K})\;\;\text{for}\;{\bf K}^*\in\{R{\bf K},R^2{\bf K}\}, \\[3pt]
     L^2_{{\bf K}^*}(\Omega_\delta)= L^2_{{\bf K}'}(\Omega_\delta)\quad\text{and}\quad A({\bf K}^*)=A({\bf K}')\;\;\text{for}\;{\bf K}^*\in\{R{\bf K}',R^2{\bf K}'\},
  \end{array}
 \end{equation}
 and since ${\bf K}'=S\,{\bf K}$, we have
 \begin{equation}\label{eq:prop_vertex2}
  L^2_{{\bf K}'}(\Omega_\delta)= \mathcal{S}\,L^2_{{\bf K}}(\Omega_\delta)\quad\text{and}\quad A({\bf K}')=\mathcal{S}A({\bf K})\mathcal{S}.
\end{equation}
 \noindent Finally, the (essential) spectrum of $A_\delta$ is given by
 $$
 \sigma(A_\delta) = \bigcup_{\kg \in \mathcal{B}} \sigma(A_\delta(\kg)) = \bigcup_{\kg \in \mathcal{B}} \bigcup_{n \in \N} \lambda_{\delta,n}(\kg).
 $$ 
 \subsection{Orthogonal decomposition of $L^2_{{\bf K}}(\Omega_\delta)$}

 In order to give more information on the structure of the spectrum near the vertices of $\mathcal{B}$, we will use a particular decomposition of $L^2_{\kg}(\Omega_\delta)$ for ${\bf K}^*\in\{{\bf K},{\bf K}'\}$ (and by \eqref{eq:prop_vertex} this decomposition holds for the other vertices of $\mathcal{B}$). This decomposition is already used and proven in \cite{Fefferman:2012,Lee:2019,Cassier:2021}. This decomposition is linked to the rotation operator $\mathcal{R}$ defined in
   \eqref{eq:rotation_op} and its eigenspaces $L^2_s(\Omega_\delta)$, $s\in\{0,1,2\}$ defined in \eqref{eq:L2s} for $\mathcal{O}=\Omega_\delta$. Let us define the following spaces
   \begin{equation}
    \label{eq:L2s_K}
   \forall {\bf K}^*\in\{{\bf K},{\bf K}'\}, \; \forall s\in\{0,1,2\},\quad  L^2_{{\bf K}^*,s}(\Omega_\delta):= L^2_{{\bf K}^*}(\Omega_\delta)\cap L^2_s(\Omega_\delta).
   \end{equation}
   By using that $R^*e_{0,\delta}=e_{2,\delta}+\vg_1-\vg_2$, $R^*e_{1,\delta}=e_{0,\delta}+\vg_1-\vg_2$ and $R^*e_{2,\delta}=e_{1,\delta}+\vg_1-\vg_2$ and ${\bf K}\cdot(\vg_1-\vg_2)=-2/3$, we obtain the following characterization:
   \begin{equation}\label{eq:carac_L2sK}
    \begin{array}{lcl}
      u\in L^2_{{\bf K},0}(\Omega_\delta)\quad\Leftrightarrow\quad u={E}_{\bf K}[u|_{{\cal C}_\delta^\sharp}]\;\;\text{and}\; \;u\big|_{e_{0,\delta}}=e^{-2\imath \pi/3}\,u\big|_{e_{1,\delta}}=e^{2\imath \pi/3}\,u\big|_{e_{2,\delta}},\\[5pt]
      u\in L^2_{{\bf K},1}(\Omega_\delta)\quad \Leftrightarrow\quad u={E}_{\bf K}[u|_{{\cal C}_\delta^\sharp}]\;\;\text{and}\;\; u\big|_{e_{0,\delta}}=u\big|_{e_{1,\delta}}=u\big|_{e_{2,\delta}},\\[5pt]
      u\in L^2_{{\bf K},2}(\Omega_\delta)\quad \Leftrightarrow\quad u={E}_{\bf K}[u|_{{\cal C}_\delta^\sharp}]\;\;\text{and}\;\; u\big|_{e_{0,\delta}}=e^{2\imath \pi/3}\,u\big|_{e_{1,\delta}}=e^{-2\imath \pi/3}\,u\big|_{e_{2,\delta}},
    \end{array}
   \end{equation}
   where ${E}_{\bf K}$ is defined in \eqref{eq:extensionop} and where we have identified functions defined on $e_{i,\delta}$ using the identification \eqref{identify_fatedges}.
   \begin{lemma}\label{lem:decomp_L2omega}
	  For all ${\bf K}^*\in\{{\bf K},{\bf K}'\}$, the space $L^2_{{\bf K}^*}(\Omega_\delta)$ admits the following orthogonal decomposition: 
	 \begin{equation}\label{DecompositionL2Epais}
	 L^2_{{\bf K}^*}(\Omega_\delta) = L^2_{{\bf K}^*,0}(\Omega_\delta) \oplus L^2_{{\bf K}^*, 1}(\Omega_\delta) \oplus L^2_{{\bf K}^*,2}(\Omega_\delta),
	 \end{equation}
	 where the $\oplus$ sign stands for the orthogonality decomposition with respect to the scalar product on $L^2({\cal C}_\delta^\sharp)$.
	 \end{lemma}
	 \begin{proof}
    Let us show \eqref{DecompositionL2Epais} for ${\bf K}^*={\bf K}$ and using \eqref{eq:prop_vertex2}, it is easy to deduce the result for ${\bf K}^*={\bf K}'$.
		 Since $1+e^{2 \imath\pi/3}+e^{-2\imath \pi/3}=0$, we have for all $u\in L^2(\Omega_\delta)$
		 \[ 	u=\frac{1}{3}(u+\mathcal{R}u+\mathcal{R}^2u)+\frac{1}{3}(u+e^{2 \imath\pi/3}\mathcal{R}u+e^{-2 \imath\pi/3}\mathcal{R}^2u)+\frac{1}{3}(u+e^{-2 \imath\pi/3}\mathcal{R}u+e^{2 \imath\pi/3}\mathcal{R}^2u).
		 \]
		 Since $\mathcal{R}^3=\mathcal{I}$, the first term of the right hand side is in $L^2_{ 0}(\Omega_\delta)$, the second one is in $L^2_{1}(\Omega_\delta)$ and the last one is in $L^2_{2}(\Omega_\delta)$. Moreover, if $u\in L^2_{{{\bf K}}}(\Omega_\delta)$ then 
		 \[
		 \forall {\bf x}\in \Omega_\delta,\;\vg\in\Lambda,\quad \mathcal{R}u({\bf x}+{\bf v})=u(R^*({\bf x}+{\bf v}))=u(R^*({\bf x}))e^{2\imath\pi R{\bf K}\cdot \vg}
		 \]
		 where we have used in the last equality that ${\bf v}\in\Lambda\,\Rightarrow \, R^*\vg \in\Lambda$ and ${\bf K}\cdot R^*\vg = R{\bf K}\cdot \vg$. Using \eqref{eq:prop_vertex}, we deduce that $\mathcal{R}u\in L^2_{{\bf K}}(\Omega_\delta)$.
		 Now, let us show that the decomposition is orthogonal with respect to the scalar product of $L^2({\cal C}_\delta^\sharp)$. Let $s\neq s'\in\{0,1,2\}$,  $u\in 	L^2_{{\bf K},s}(\Omega_\delta)$ and $u'\in L^2_{{\bf K},s'}(\Omega_\delta)$ then using \eqref{eq:carac_L2sK} we have \vspace{-0.2cm}
		 \[
		 	\int_{{\cal C}_\delta^\sharp}u\overline{u'}=\sum_{j=0}^2 \int_{e_{j,\delta}}u\overline{u'}=(1+e^{2\imath\pi/3}+ e^{-2\imath\pi/3})\int_{e_{0,\delta}}u\overline{u'}=0.\vspace{-0.1cm}
		 \]
		 \end{proof}\vspace{-0.3cm}
By noting that for ${\bf K}^*\in\{{\bf K},{\bf K}'\}$ and each $s\in\{0,1,2\}$, each space	$D(A_\delta({\bf K}^*))\cap L^2_{{\bf K}^*,s}(\Omega_\delta)$	is stable by the operator $A_\delta({\bf K}^*)$, by denoting
\begin{equation}\label{eq:A_s}
\forall {\bf K}^*\in\{{\bf K},{\bf K}'\},\;\forall s\in\{0,1,2\},\quad A_{\delta,s}({\bf K}^*):=A_\delta({\bf K}^*)\big|_{L^2_{{{\bf K}^*},s}(\Omega_\delta)},
\end{equation}
we deduce easily the following decomposition.
\begin{corollary}\label{cor:decomp_A}
For ${\bf K}^*\in\{{\bf K},{\bf K}'\}$, the operator $A_\delta({\bf K}^*)$ can be decomposed as follows
\begin{equation}\label{eq:decomp_A}
     A_\delta({\bf K}^*)=A_{\delta,0}({\bf K}^*)\oplus A_{\delta,1}({\bf K}^*) \oplus A_{\delta,2}({\bf K}^*),
\end{equation}
where $A_{\delta,s}({\bf K}^*)$ for $s\in\{0,1,2\}$  is defined in \eqref{eq:A_s}.
\end{corollary}
Let us now relate the eigenvalues and eigenvectors of $A_{\delta,1}({\bf K}^*)$ to the ones of $A_{\delta,2}({\bf K}^*)$. To do so, we use  the symmetry operator $\mathcal{S}$ defined in  \eqref{DefinitionSym} for $\mathcal{O}=\Omega_\delta$.
We can now state the following result.
 \begin{proposition}\label{sym_eigenvectors}
 Let ${\bf K}^*\in\{{\bf K},{\bf K}'\}$. $(\lambda_\delta, \phi_{\delta,1})$ is an eigenpair of $A_{\delta,1}({\bf K}^*)$ if and only if $(\lambda_\delta, \phi_{\delta,2}:=\overline{\mathcal{S}\phi_{\delta,1}})$ is an eigenpair of $A_{\delta,2}({\bf K}^*)$. Moreover
\begin{equation}
    \int_{\mathcal{C}_\delta^\sharp}\phi_{\delta,s}({\bf x})\,\overline{\nabla \phi_{\delta,s}({\bf x})}d{\bf x}=0\quad\text{for }\;s\in\{1,2\},
\end{equation}
and there exists $v_\delta$ such that 
\begin{equation}\label{eq:v_delta}
   \int_{\mathcal{C}_\delta^\sharp}\nabla \phi_{\delta,1}({\bf x})\,\overline{\phi_{\delta,2}({\bf x})}d{\bf x}=  {-} \int_{\mathcal{C}_\delta^\sharp}\phi_{\delta,1}({\bf x})\,\overline{\nabla \phi_{\delta,2}({\bf x})}d{\bf x}=v_\delta (1,\imath)^T.
\end{equation}
 \end{proposition}\vspace{-0.5cm}
 \begin{proof}
  Let us show the result for ${\bf K}^*={\bf K}$ and using \eqref{eq:prop_vertex2}, it is easy to deduce it for ${\bf K}^*={\bf K}'$.
First, by using \eqref{eq:prop_vertex2}, $u \in L^2_{{\bf K}}(\Omega_\delta)\;\Leftrightarrow\;\mathcal{S} u \in L^2_{{\bf K}'}(\Omega_\delta)\;\Leftrightarrow\; \overline{\mathcal{S} u} \in L^2_{{\bf K}}(\Omega_\delta) $. Moreover, $u \in L^2_{1}(\Omega_\delta)\;\Leftrightarrow\;\mathcal{S} u \in L^2_{1}(\Omega_\delta)\;\Leftrightarrow\; \overline{\mathcal{S} u} \in L^2_{2}(\Omega_\delta) $. Since, the operator $A_{\delta}({\bf K})$ commutes with $\mathcal{S}$, we deduce that $\phi_{\delta,1}$ is an eigenvector of $A_{\delta,1}({\bf K})$ associated with the eigenvalue $\lambda_\delta$ if and only if $\phi_{\delta,2}:=\overline{\mathcal{S}\phi_{\delta,1}}$ is an eigenvector of $A_{\delta,2}({\bf K})$ associated with the same eigenvalue.

Let $s\in\{1,2\}$, by applying the change of variable ${\bf x}\mapsto R{\bf x}$, we obtain
\begin{eqnarray*}
\mathbb{U}_s:=\int_{\mathcal{C}_\delta^\sharp}\phi_{\delta,s}({\bf x})\,\overline{\nabla \phi_{\delta,s}({\bf x})}\,d{\bf x}&=&\int_{R\mathcal{C}_\delta^\sharp}\phi_{\delta,s}(R^*{\bf x})\,\overline{[\nabla \phi_{\delta,s}](R^*{\bf x})}\,d{\bf x}\quad(\text{since}\,|\text{det}R|=1\,)\\
&=&\int_{R\mathcal{C}_\delta^\sharp}\phi_{\delta,s}({\bf x})\,{R}^*\overline{[\nabla \phi_{\delta,s}]({\bf x})}\,d{\bf x}\quad\text{since $\phi_{\delta,s},\,\nabla \phi_{\delta,s} \in \,L^2_s(\Omega_\delta)$}\\
&=&{R}^*\int_{\mathcal{C}_\delta^\sharp}\phi_{\delta,s}({\bf x})\,\overline{[\nabla \phi_{\delta,s}]({\bf x})}\,d{\bf x}\quad\text{since $\phi_{\delta,s}\in \,H^1_{{\bf K}}(\Omega_\delta)$}.
\end{eqnarray*}
The vector $\mathbb{U}_s$ satisfies then ${R}\mathbb{U}_s=\mathbb{U}_s$ and since $1$ is not an eigenvalue of ${R}$, we deduce that $\mathbb{U}_s=0.$
Since $ \phi_{\delta,2}=\overline{\mathcal{S}\phi_{\delta,1}}$ and $ \phi_{\delta,1}=\overline{\mathcal{S}\phi_{\delta,2}}$, we have
\[
 \int_{\mathcal{C}_\delta^\sharp}\phi_{\delta,1}({\bf x})\,\overline{\nabla \phi_{\delta,2}({\bf x})}d{\bf x}=-\int_{S\mathcal{C}_\delta^\sharp}\overline{\phi_{\delta,2}({\bf x})}\,{\nabla \phi_{\delta,1}({\bf x})}d{\bf x}=-\int_{\mathcal{C}_\delta^\sharp}\overline{\phi_{\delta,2}({\bf x})}\,{\nabla \phi_{\delta,1}({\bf x})}d{\bf x},
\]
where we have used that $\phi_{\delta,1}$ and $\phi_{\delta,2}$ are in $H^1_{{\bf K}}(\Omega_\delta)$ and $Se_{0,\delta}=e_{0,\delta}+\vg_1-\vg_2$, $Se_{1,\delta}=e_{1,\delta}-\vg_2$ and $Se_{2,\delta}=e_{2,\delta}+\vg_1-2\vg_2$. Moreover, by 
applying the change of variable ${\bf x}\mapsto R{\bf x}$ and by using similar arguments than for the computation of $\mathbb{U}_s$, we obtain 
\begin{eqnarray*}
\mathbb{V}:=\int_{\mathcal{C}_\delta^\sharp}\phi_{\delta,1}({\bf x})\,\overline{\nabla \phi_{\delta,2}({\bf x})}\,d{\bf x}&=&\int_{R\mathcal{C}_\delta^\sharp}\phi_{\delta,1}(R^*{\bf x})\,\overline{[\nabla \phi_{\delta,2}](R^*{\bf x})}\,d{\bf x}\\
&=&\int_{R\mathcal{C}_\delta^\sharp}e^{2 \imath \pi/3}\phi_{\delta,1}({\bf x})\,{R}^*e^{-4\imath \pi/3}\overline{[\nabla \phi_{\delta,2}]({\bf x})}\,d{\bf x},\\
&=&{R}^*e^{-2 \imath \pi/3}\int_{\mathcal{C}_\delta^\sharp}\phi_{\delta,1}({\bf x})\,\overline{[\nabla \phi_{\delta,2}]({\bf x})}\,d{\bf x}.
\end{eqnarray*}
The vector $\mathbb{V}$ satisfies ${R}\mathbb{V}=e^{-2 \imath \pi /3}\mathbb{V}$, which means that it is collinear to $(1,\imath)^T$. 
 \end{proof}\vspace{-0.5cm}
   \begin{remark}\label{RemAutreDefinitionvdelta} Naturally, since $\phi_{\delta, 2}({\bf x}) = \overline{\phi_{\delta, 1}(S{\bf x})}$, $v_\delta$ can equivalently be defined as
  $$
  v_\delta (1,\imath)^T= \int_{\mathcal{C}_\delta^\sharp}\phi_{\delta,1}(S{\bf x})\,\nabla (\phi_{\delta,1}({\bf x}))d{\bf x}
   $$
  \end{remark}\vspace{-0.3cm}
We investigate now the existence of Dirac points in the spectrum of $A_\delta$ in the neighborhood of ${\bf K}^*\in\{{\bf K},{\bf K}'\}$, for $\delta$ small enough.
Let us first recall the definition of Dirac points.
\begin{definition}[Dirac points]\label{DefinitionDirac}
The pair $({\bf K}^*,\lambda^*)\in \mathcal{B}\times\R^+$ is a Dirac point if there exists $n\in\N$ such that $\kg\mapsto \lambda_{n,\delta}(\kg)$ and $\kg\mapsto \lambda_{n+1,\delta}(\kg)$ satisfies 
\begin{itemize}
    \item $\lambda^*=\lambda_{n,\delta}({\bf K}^*)=\lambda_{n+1,\delta}({\bf K}^*)$ is an eigenvalue of multiplicity 2 of $A_\delta({\bf K}^*)$;
    \item there exists a constant $\alpha^*>0$ such that 
\begin{equation}\label{eq:Dirac points}
\begin{array}{|l}
\lambda_{n,\delta}(\kg)=\lambda^*-\alpha^*|\kg-{\bf K}^*| + o(\|\kg-{\bf K}^*\|)\\
\lambda_{n+1,\delta}(\kg)=\lambda^*+\alpha^*|\kg-{\bf K}^*| + o(\|\kg-{\bf K}^*\|)
\end{array}\end{equation}
\end{itemize}
\end{definition}
The following result, which is an analogue of Theorem 4.1 of \cite{Fefferman:2012} and Theorem 2 of \cite{Lee:2019}, provide sufficient conditions of existence of Dirac Points in our context.
\begin{proposition}\label{PropositionDiracgrapheEpais}
Let $\delta>0$ and ${\bf K}^*\in\{{\bf K},{\bf K}'\}$. Let $\lambda_\delta$ \Becommente{be}  an eigenvalue of multiplicity 1 of $A_{\delta,1}({\bf K}^*)$ and $\phi_{\delta,1}$ be an associated eigenvector such that $\|\phi_{\delta,1} \|_{L^2(\mathcal{C}^\sharp_\delta)}=1$. 
Suppose that $\lambda_\delta$ is not an eigenvalue of $A_{\delta,0}({\bf K}^*)$ and that $v_\delta$ defined in \eqref{eq:v_delta} does not vanish. Then, $A_\delta$ admits a Dirac point in the neighborhood of ${\bf K}^*$ with $\alpha^*=4\pi v_\delta$ in \eqref{eq:Dirac points}.
\end{proposition}
The proof of Proposition~\ref{PropositionDiracgrapheEpais} is given in Annex \ref{annex:proof_propDirac}. We have adapted the one of \cite{Fefferman:2012,Lee:2019} replacing an operator formalism with a bilinear form one, which is necessary for our problem in order to take into account the boundary conditions.{The demonstration of Theorem \ref{th:Diracpoints_delta} finally only consists in verifying  assumptions of Proposition \ref{PropositionDiracgrapheEpais}. This is done using asymptotic arguments that require first to identify the limit operator.} 
% we show that for $\delta$ small enough, the  spectrum of the operator $A^{\delta}$ contains a certain number of Dirac points located near the  limit ones.  It suffices to show that the assumptions of Proposition \ref{PropositionDiracgrapheEpais} are satisfied. This will be done by an asymptotic analysis (see ...\commente{A completer}).

\subsection{The limit graph and the associated limit operator}
\subsubsection{Definition of the limit operator and convergence properties}

As $\delta$ tends to $0$, the domain $\Omega_\delta$ tends to the periodic quantum graph $\cG$ (see Definition~\eqref{DefintionGrapheG} and Figure~\ref{fig:omega}). In order to introduce the formal limit of the operator $A_\delta$, let us introduce the functional spaces (remind that $\mathcal{E}$ is the set of the edges of $\cG$, see~\eqref{DefinitionSetofEdges})
 $$
 L^2(\cG) = \{ u,\quad u \in L^2(e) \;\;\forall e \in \mathcal{E},\quad \|u\|_{L^2(\cG)}^2:= \sum_{e \in \mathcal{E}} \| u \|_{L^2(e)}^2 < + \infty\},
 $$ 
 $$
 H^1(\cG) = \{ u \in \mathcal{C}(\cG), \quad u \in H^1(e) \;\;\forall e \in \mathcal{E},\quad  \|u\|_{H^1(\cG)}^2:= \sum_{e \in \mathcal{E}} \| u \|_{H^1(e)}^2 < + \infty\},
 $$ 
   $$H^2(\cG)=\{u\in H^1(\cG),\quad u\in H^2(e)\;\;\forall e\in \mathcal{E},\quad \sum_{e \in \mathcal{E}} \| u \|_{H^2(e)}^2 < + \infty\}$$
 where $\mathcal{C}(\cG)$ denotes the set of continuous functions defined on $\cG$.
 Here and in what follows $u'$ (resp. $u''$) denotes the function defined on the graph $\cG$ by taking the derivative (resp. the second derivative) of the restriction of $u$ on each edge $e$ with respect to the local variable $t$, introduced in the parametrization of the edges described in \eqref{eq:param}, see also Figure~\ref{fig:percell}. 
 The domain of the limit operator $\cA$ is given by
 \begin{equation}\label{eq:Domaine_A_graph}
 {D}(\cA) = \{ u \in H^2(\cG), \; \sum_{e \in \mathcal{E}(M) } u\big|_{e}'(M) = 0 \;\; \forall M \in \mathcal{V} \},
 \end{equation}
 where $\mathcal{E}(M)$  stands for the set  of edges adjacent to the vertex $M$. The condition
 \begin{equation}\label{eq:Kirch}
 \sum_{e\in \mathcal{E}(M) } u\big|_{e}'(M)=0,
 \end{equation}
is the so-called Kirchhoff Law \cite{KuchmentQuantumSurvey,KuchmentQuantumgraphs1}. We define finally the limit operator $\cA$ as follows
 \begin{equation}\label{eq:op_graph}
 \forall u\in {D}(\cA),\quad [\cA u]|_{e} = - \partial_t^2[u|_{e}] \quad \forall e \in \mathcal{E}.  
 \end{equation}
 See \cite{KuchmentQuantumgraphs1,AdrienRapportRecherche,Adrien} for more details on this derivation.
As previously, the spectrum of $\cA$ can be characterized by using the Floquet-Bloch theory.  Let us introduce the space
 \begin{equation}
 L^2_{\kg}(\cG) = \left\{ f,\quad f \in L^2(e) \;\;\forall e \in \mathcal{E},\quad  f(\cdot + \vg) = f e^{2\imath\pi \kg \cdot  \vg} \;\; \forall \vg \in \Lambda  \right\},
 \end{equation}
which can be identified to $L^2(\cG^\sharp)$ through a ${\bf k}$-quasi-periodic extension operator $\mathcal{E}_{\bf k}:L^2(\cG^\sharp)\rightarrow L^2_{\kg}(\cG)$ defined by
 \begin{equation}\label{eq:extensionop_1D}
 	 \forall f\in L^2(\cG^\sharp),\quad \forall {\bf x}\in \cG^\sharp,\quad[\mathcal{E}_{\bf k}f]({\bf x} + \vg) = f({\bf x}) e^{2\imath\pi \kg \cdot  \vg} \;\; \forall \vg \in \Lambda,
 \end{equation}
 and we have
 \[
 	f\in L^2_{\kg}(\cG)\quad\Leftrightarrow\quad f=\mathcal{E}_{\bf k}[f|_{\cG^\sharp}].
 \]
  We  equip $ L^2_{\kg}(\cG)$ with the scalar product of $L^2(\cG^\sharp)$ and the associated norm. 
 Let us define the set of  locally $H^1$ functions which are ${\bf k}\cdot \vg_i$ quasi-periodic in the direction $\vg_i$ for $i\in\{1,2\}$
   \begin{equation}
   H^1_{\kg}(\cG) = \left\{ u \in \mathcal{C}(\cG),\quad u \in L^2_{\kg}(\cG),\quad u' \in L^2_{\kg}(\cG) \right\}.
   \end{equation}
   The space $H^1_{\kg}(\cG)$ is a closed subspace of $H^1(\cG)$ so we equip $H^1_{\kg}(\cG)$ with the scalar product of $H^1(\cG^\sharp)$ and its associated norm.

For any $\kg\in\R^2$, the reduced operator $\cA(\kg)$ is defined  
\begin{equation}\label{eq:op_k_graph}
 \begin{array}{|l}
\dsp  {D}(\cA(\kg)) = \{ u \in H^1_{\kg}(\cG),\quad  u'' \in L^2_\kg(\cG), \;   \sum_{e \in \mathcal{E}(M) } u\big|_{e}'(M) = 0\; \; \forall M \in \mathcal{V} \},\\[9pt] \dsp\forall u\in {D}(\cA(\kg)),\; [\cA(\kg) u]\big|_{e} = - \partial_t^2[u\big|_{e}] \quad \forall e\in \mathcal{E}.
 \end{array}
\end{equation}
One can show that for any $\kg \in \R^2$, $\cA(\kg)$ is self-adjoint non negative with compact resolvent. As a result, its spectrum consists of an increasing sequence of  non negative eigenvalues $(\lambda_n(\kg))_{n\in \N}$ that tends to $+\infty$ as $n$ tends to $+\infty$.  The mappings $\kg \mapsto \lambda_n(\kg)$ are Lipschitz continuous functions and are the dispersive surfaces of the operator $\cA$. Again,  it suffices to consider that $\kg$ varies over the Brillouin zone $
\mathcal{B}$. The spectrum of the operator $\cA$ is then given by
 $$
 \sigma(\cA) = \bigcup_{\kg \in \mathcal{B}} \sigma(\cA(\kg)) = \bigcup_{\kg \in \mathcal{B}} \bigcup_{n \in \N} [\lambda_{n}(\kg)].
 $$ 
Actually, the spectrum of the operator $\cA(\kg)$ can be computed explicitly. This was done in  \cite[Lemma 3.1-Lemma 3.5]{KuchmentPost}. This enables to establish the conical behaviour of some dispersive surfaces in the vicinity of the vertices of the hexagonal Brillouin zone. For the sake of completeness, we repeat  the main steps of the proof.
 \begin{proposition}\label{CharacSpectrumGraphHexa}For all $\kg\in\mathcal{B}$, we have
	 \begin{equation}\label{eq:disp_rel}
	 	\lambda \in \sigma(\cA(\kg)), \,\lambda\geq 0\quad\Leftrightarrow\quad \sqrt{\lambda} L \in \N^\ast\pi\;\text{or}\; \cos^2 \sqrt{\lambda}  L = \frac{1}{9} | 1 +  e^{2\imath\pi \vg_1\cdot\kg} + e^{2\imath\pi\vg_2\cdot\kg}|^2  .
	 \end{equation}
 Moreover, for ${\bf K}^*\in\{{\bf K},{\bf K}'\}$, we have that for all $n\in\N$ $({\bf K}^*,\lambda_n^*$ where $\lambda_n^*$ is given in \eqref{eq:lambda_n_star},  is a Dirac point. More precisely, for ${\bf K}^*\in\{{\bf K},{\bf K}'\}$ and for all $n\in\N$
 \begin{equation}\label{eq:Dirac_graph}
 	\forall \kg\in\mathcal{B},\quad\begin{array}{|l}
\lambda_{3n}(\kg) = \displaystyle \lambda_n^*  - \alpha_n  \|\kg-{\bf K}^*\| + {O}(\| \kg-{\bf K}^* \|^2), \\[5pt]
	
\lambda_{3n+1}(\kg) =\displaystyle \lambda_n^*   +  \alpha_n  \|\kg-{\bf K}^*\| + {O}(\| \kg-{\bf K}^* \|^2),\\[5pt]
		\end{array}. 
 \end{equation}
 with $\alpha_n =  (2n+1)\pi^2/L$.
 \end{proposition}\vspace{-0.3cm}
   To prove the existence of Dirac point, we could reproduce the analysis made in Proposition~\ref{sym_eigenvectors}. But it is possible, to use direct computation as it is done in ~\cite{KuchmentPost} and reproduced in Appendix \ref{annex:proof_propquantum}.

 \subsubsection{Symmetry properties of the eigenvalues and eigenvectors\label{SectionSymmetryGraph}}

 As for $L^2_{{\bf K}^*}(\Omega_\delta)$, we can introduce  a particular decomposition of $L^2_{{\bf K}^*}(\cG)$ for ${{\bf K}^*}\in\{{\bf K},{\bf K}'\}$ linked to the rotation operator $\mathcal{R}$ defined in \eqref{eq:rotation_op} with $\mathcal{O}=\cG$. 
 As in \eqref{eq:L2s_K}, let us define the spaces
   \begin{equation}
    \label{eq:L2s_K_graph}
   \forall {\bf K}^*\in\{{\bf K},{\bf K}'\}, \; \forall s\in\{0,1,2\},\quad  L^2_{{\bf K}^*,s}(\cG):= L^2_{{\bf K}^*}(\cG)\cap L^2_s(\cG),
   \end{equation}
   where $L^2_s(\cG)$, $s\in\{0,1,2\}$ are defined in \eqref{eq:L2s} for $\mathcal{O}=\cG$. As in \eqref{eq:carac_L2sK}, we have the following characterization
   \begin{equation}\label{eq:carac_L2sK_graph}
    \begin{array}{lcl}
      u\in L^2_{{\bf K},0}(\cG)\quad \Leftrightarrow \quad u=\mathcal{E}_{\bf K} [u|_{\cG^\sharp}]\;\;\text{and}\; \;u\big|_{e_{0}}=e^{-2\imath \pi/3}\,u\big|_{e_{1}}=e^{2\imath \pi/3}\,u\big|_{e_{2}},\\[5pt]
      u\in L^2_{{\bf K},1}(\cG)\quad \Leftrightarrow\quad u=\mathcal{E}_{\bf K} [u|_{\cG^\sharp}]\;\;\text{and}\; \; u\big|_{e_{0}}=u\big|_{e_{1}}=u\big|_{e_{2}},\\[5pt]
      u\in L^2_{{\bf K},2}(\cG)\quad \Leftrightarrow \quad u=\mathcal{E}_{\bf K} [u|_{\cG^\sharp}]\;\;\text{and} \;\; u\big|_{e_{0}}=e^{2\imath \pi/3}\,u\big|_{e_{1}}=e^{-2\imath \pi/3}\,u\big|_{e_{2}},
    \end{array}
   \end{equation}
   where $ \mathcal{E}_{\bf K}$ is defined in \eqref{eq:extensionop_1D} and where we have identified functions defined on the edges $e_i$ using the parametrization \eqref{orientedgeneratorEdges}.
 We have the following decomposition, which could be proven as Lemma \ref{lem:decomp_L2omega}.
 \begin{lemma}
 For all ${\bf K}^*\in\{{\bf K},{\bf K}'\}$, the space $L^2_{{\bf K}^*}(\cG)$ admits the following orthogonal decomposition 
 \begin{equation}
L^2_{{\bf K}^*}(\cG) =L^2_{{\bf K}^*,0}(\cG)\oplus L^2_{{\bf K}^*,1}(\cG) \oplus L^2_{{\bf K}^*,2}(\cG),
 \end{equation} 
 where $\oplus$ sign stands for the orthogonal decomposition with respect to the scalar product on $L^2(\cG^\sharp)$.
 \end{lemma}
Since, for ${\bf K}^*\in\{{\bf K},{\bf K}'\}$ and $s\in\{0,1,2\}$, each space $D(\cA({\bf K}^*))\cap L^2_{{\bf K}^*,s}(\cG)$ is stable by the operator $\cA({\bf K}^*)$, we can introduce
\begin{equation}
    \label{eq:A_s_graph}
    \forall\,{\bf K}^* \in\{{\bf K},{\bf K}'\},\;\forall s\in\{0,1,2\},\quad \cA_s({\bf K}^*)=\cA({\bf K}^*)\big|_{L^2_{{\bf K}^*,s}(\cG)},
\end{equation}
and deduce the following decomposition.
\begin{corollary}
\label{cor:decomp_A_graph}
For ${\bf K}^*\in\{{\bf K},{\bf K}'\}$, the operator $\cA({\bf K}^*)$ can be decomposed as follows
\begin{equation}
  \label{eq:decomp_A_graph}  
  \cA({\bf K}^*)=\cA_0({\bf K}^*)\oplus\cA_1({\bf K}^*)\oplus\cA_2({\bf K}^*),
\end{equation}
where $\cA_s({\bf K}^*)$ for $s\in\{0,1,2\}$ is defined in \eqref{eq:A_s_graph}.
\end{corollary}
We can finally relate, for ${\bf K}^*\in\{{\bf K},{\bf K}'\}$,  the eigenvalues and eigenvectors of $\cA_1({\bf K}^*)$ to the ones of $\cA_2({\bf K}^*)$ for ${\bf K}^*\in\{{\bf K},{\bf K}'\}$. To do so, we use the symmetry operator ${\mathcal{S}}$ defined in \eqref{symmetry} for $\mathcal{O}=\cG$ and
  the following result holds using similar arguments than in Proposition \ref{sym_eigenvectors}.
 \begin{lemma} Let ${\bf K}^*\in\{{\bf K},{\bf K}'\}$. Then, $(\lambda^2, \phi)$ is an eigenpair of $\cA_{1}({\bf K}^*)$ if and only if $(\lambda^2, \overline{{\mathcal{S}}\phi})$ is an eigenpair of $\cA_{2}({\bf K}^*)$.
 \end{lemma}
 \begin{remark}\label{CritereDiracGraphe}
 We could have reproduced the analysis made in Proposition~\ref{sym_eigenvectors} and Proposition~\ref{PropositionDiracgrapheEpais} to prove the existence of Dirac points. In that context, denoting by $\phi_1$ one eigenvector of $\mathcal{A}_1({\bf K}^*)$, and by $\tau_0$ (resp. $\tau_1$ and $\tau_2$) the normalized tangent vector to $e_0$ oriented from the vertex $A$ to the vertex $B$ (resp. from $A$ to $B_{1,0}$ and from $A$ to $B_{0,1}$, see Fig~\ref{fig:percell}), it would yield to prove that there exists a complex number  $v_\mathcal{G}\neq 0$ such that 
  $$
  \sum_{i=0}^2 \tau_i \int_{e_i} \phi_1(s) \phi_1'(L-s) ds = v_\mathcal{G} (1, \imath)^t.
  $$ 
  Moreover, we could additionally verify   that $4 \pi | v_\mathcal{G}|$ coincides with $\alpha_n$ defined in~\eqref{eq:Dirac_graph}.
 \end{remark}\vspace{-0.3cm}
  It is worth noticing that we can link the eigenpairs of $\cA_{0}({\bf K})$ and $\cA_{1}({\bf K})$ to the one of 1d-Laplacian operators on the interval $(0, L)$. Similar results could be obtained for $\cA_{2}({\bf K})$ and $\cA_{i}({\bf K}'),\,i\in\{0,1,2\}$ but they are not used in the following so we omit them. 
  \begin{proposition}\label{prop : OperateurReduitGraph}
  \begin{itemize}
 \item If $(\lambda, \phi)$ is an eigenpair of $\cA_{1}({\bf K})$ then $(\lambda, \phi_0 := \phi|_{{e_0}})$ is an eigenpair of the  1d Laplacian operator $\mathcal{A}_{ND}:\mathcal{D}(\mathcal{A}_{ND}) \subset L^2(0,L) \rightarrow L^2(0,L)$ with Neumann boundary condition on $t=0$ and Dirichlet on $t=L$
 \begin{equation}\label{eigenOned1}
  \mathcal{D}(\mathcal{A}_{ND})  = \left\{ u  \in H^2(0,L), \; u'(0)= 0, u(L) =0 \right\}\;,\quad
  \mathcal{A}_{ND} u = - u''.
 \end{equation}
 Reciprocally, let $(\lambda, \phi_0)$ be an eigenpair of $\mathcal{A}_{ND}$ and introduce ${\phi} $ defined by
 \begin{equation}\label{DefinitionWidetildephi1}
{\phi}  = \mathcal{E}_{\bf K}( \hat{\phi})   \quad \text{where}\quad \hat{\phi}|_{e_0} = \hat{\phi}|_{e_1} =\hat{\phi}|_{e_2}=\phi_0,
 \end{equation}
 where $ \mathcal{E}_{\bf K}$ is defined in \eqref{eq:extensionop_1D} and where we have identified functions defined on the edges $e_i$ to 1D functions using the parametrization \eqref{orientedgeneratorEdges}.
 Then $(\lambda, \phi)$ is an eigenpair of $\cA_{1}({\bf K})$. 
% \item If $(\lambda, \phi)$ is an eigenpair of $\cA_{2}({\bf K})$ then $(\lambda, \phi_0 := \phi|_{{e_0}})$ is an eigenpair of the  1d Laplacian operator $\mathcal{A}_{DN}:\mathcal{D}(\mathcal{A}_{DN}) \subset L^2(0,L) \rightarrow L^2(0,L)$ with Dirichlet boundary condition on $t=0$ and Neumann on $t=L$
% \begin{equation}\label{eigenOned2}
%  \mathcal{D}(\mathcal{A}_{DN})  = \left\{ u  \in H^2(0,L), \; u(0)= 0, u'(L) =0 \right\}\,,\quad 
%  \mathcal{A}_{DN} u = - u''.
% \end{equation}
% Reciprocally, let $(\lambda, \phi_0)$ be an eigenpair of $\mathcal{A}_{DN}$ and introduce ${\phi} $ defined by
% \begin{equation}\label{DefinitionWidetildephi2}
%{\phi}  = \mathcal{E}_{\bf K}( \hat{\phi})  \quad \hat{\phi}|_{e_0} = e^{2\imath \pi/3}\,\hat{\phi}|_{e_1} =e^{-2\imath \pi/3}\,\hat{\phi}|_{e_2}=\phi_0
% \end{equation}
% then $(\lambda, \phi)$ is an eigenpair of $\cA_{2}({\bf K})$. 
 \item If $(\lambda, \phi)$ is an eigenpair of $\cA_0({\bf K})$ then, $(\lambda, \phi_0 := \phi|_{e_0})$ is an eigenpair of the 1d Dirichlet Laplacian operator $\mathcal{A}_{D}:\mathcal{D}(\mathcal{A}_{D}) \subset L^2(0,L) \rightarrow L^2(0,L)$:
 \begin{equation}\label{eigenOned0}
    \mathcal{D}(\mathcal{A}_{D})  = \left\{ u  \in H^2(0,L), \; u(0)=  u(L) =0 \right\}\;,\quad 
    \mathcal{A}_{D} u = - u''.
   \end{equation}
 Reciprocally,  let $(\lambda, \phi)$ be an eigenpair of $\mathcal{A}_{D}$ and introduce ${\phi} $ defined by
 \begin{equation}\label{DefinitionWidetildephi0}
{\phi}  = \mathcal{E}_{\bf K}( \hat{\phi})  \quad \text{where}\quad \hat{\phi}|_{e_0} = e^{-2\imath \pi/3}\,\hat{\phi}|_{e_1} =e^{2\imath \pi/3}\,\hat{\phi}|_{e_2}=\phi_0.
 \end{equation}
 Then $(\lambda, \phi)$ is an eigenpair of $\cA_{0}({\bf K})$. 
 \end{itemize} 
  \end{proposition}\vspace{-0.5cm}
% \begin{figure}[htbp]
%     \centering
%     \includegraphics[width=0.7\textwidth, trim=1cm 1cm 1cm 1cm, clip]{./DessinGrapheAnnotation.png}
%     \caption{Representation of the Periodic graph}
%     \label{fig:PeriodicGraphAnnotation}
% \end{figure}
\begin{proof}
  We show only the first result, since the other one can be proven similarly.
 Assume that $(\lambda, \phi)$ is an eigenpair of $\cA_{1}({\bf K})$ and let $\phi_0:=\phi|_{e_0}$. By definition of the operator $\cA$  defined on the graph, we have $\phi_0\in H^2(0,L)$ and $-\phi_0''=\lambda\phi_0$ on $(0,L)$. Using the characterization \eqref{eq:carac_L2sK_graph} of $L^2_{{\bf K},1}(\cG)$, 
    writing the Kirchhoff condition at the vertex $A$ gives
    $\phi_0'(0) =0,$
    while the the continuity condition in $B_{1,0}$ and the ${\bf K}-$ quasi-periodicity leads to
    $$
\phi_0(L) =  e^{2 \imath  \pi {\bf K} \cdot \mathbf{v_1}} \phi_0(L),
    $$ 
    and consequently
    $
    \phi_0(L)=0.
    $ 
   Reciprocally, assume that $(\lambda, \phi_0)$ is an eigenpair for $\mathcal{A}_{ND}$ and  let us show that $(\lambda, {\phi})$ defined by \eqref{DefinitionWidetildephi1} is an eigenpair for $\cA_{1}({\bf K})$. First, it is clear that $\phi\in L^2_{\bf K}(\cG)$ and $-{\phi}'' = \lambda {\phi}$ on all the edges of the graph $\mathcal{G}$. Then, we can verify that ${\phi}$ is continuous at each vertex  of the graph: indeed, it is easily seen that  ${\phi}(B_{m, n})=0$  and  ${\phi}(A_{m, n})=e^{2 \imath  \pi {\bf K} .(m \mathbf{v_1} + n \mathbf{v_2})} \phi_0(1)$ for any $(m,n) \in \Z^2$.  Besides, the Kirchhoff conditions are satisfied on the $A_{m,n}$'s 
   $$
   \sum_{i=0}^3 {\phi}|_{e_i}'(A_{ m,n}) = (e^{2 \imath \pi {\bf K}. (m\mathbf{v_1} + n\mathbf{v_2})} ) \sum_{i=0}^3 \phi_0'(0) = 0,
   $$
   and on the $B_{m,n}$'s
   \begin{multline*}
   \sum_{i=0}^3 {\phi}|_{e_i}'(B_{m, n}) =  - e^{2 \imath \pi {\bf K}\cdot (m\mathbf{v_1} + n\mathbf{v_2}) } (1 + e^{-2 \imath \pi {\bf K} \cdot \mathbf{v_1} } + e^{-2 \imath \pi {\bf K}\cdot \mathbf{v_2} })   \phi'(L) = 0.
   \end{multline*}
   We have then that $\phi\in D(\cA({\bf K}))$. From \eqref{eq:carac_L2sK_graph}, we deduce that $\phi\in L^2_{{\bf K},1}$ which finishes the proof. 
\end{proof}
We deduce from the previous proposition the spectrum of $\mathcal{A}_s(K)$ for $s\in\{0,1,2\}$.
\begin{corollary}\label{lemma:operateurreduit}
The spectrum of $\mathcal{A}_{0}({\bf K})$ consists of the set of simple eigenvalues $\{
  ({n \pi}/{L})^2,\; n \in \N^\ast\}$
while the spectrum of $\mathcal{A}_1({\bf K})$ consists of the simple eigenvalues $\{ \lambda_n^*,\; n \in \N\}$ $\lambda_n^*$ being defined in \eqref{eq:lambda_n_star}.
\end{corollary}

\subsection{Proof of Theorem \ref{th:Diracpoints_delta} and Proposition \ref{PropositionDiracgrapheEpais}}

We show the result for ${\bf K}^*={\bf K}$, the result for ${\bf K}^*={\bf K}'$ can be obtained by using \eqref{eq:prop_vertex2}.

 General results of \cite{PostBook,KuchmentQuantumgraphs1,KuchmentPost} prove the convergence of the eigenvalues of $A_\delta(\kg)$ (respectively $A_{\delta,i}(\kg)$, $i \in \{0,1,2\}$) towards the eigenvalues of  $\cA(\kg)$ (resp. $\cA_{i}(\kg),\,i \in \{0,1,2\}$).  In particular, those results together with the analysis of the previous subsection show that, for any $n\in \mathbb{N}$, there exists $\delta_0>0$ such that for any $\delta< \delta_0$, the operator $A_{\delta,1}({\bf K})$ has a simple eigenvalue $\lambda_{n,\delta}^*$ that tends, as $\delta$ goes to $0$, to $\lambda_n^*$ given in Corollary \ref{lemma:operateurreduit}. Moreover,  there exist two positive constants $C_1$  and $C_2$, depending on $\delta_0$ and $n$ such that 
 $$
 | \lambda_{n,\delta}^* - \lambda_n^* | \leq C_1 \sqrt{\delta},
 $$ 
and 
\begin{equation}\label{DistanceCSpectre}
  \inf_{\lambda \in \sigma(A_{\delta,1}({\bf K})) \setminus \{ \lambda_{n,\delta} \}} | \lambda - \lambda_{n,\delta}^* | > C_2,\quad\text{and}\quad \inf_{\lambda \in \sigma(A_{\delta,0}({\bf K}))} | \lambda - \lambda_{n,\delta}^* | > C_2.
\end{equation}
To show Proposition \ref{PropositionDiracgrapheEpais} and then Theorem \ref{th:Diracpoints_delta}, it suffices now to show that $v_\delta$ defined in \eqref{eq:v_delta} does not vanish for $\delta$ small enough. This is stated in the following proposition.
 \begin{proposition}\label{TheoremeLimitevdelta}
 Let $(\lambda_{\delta},\phi_{\delta})$ be an eigenpair of $A_{\delta,1}({\bf K})$ and let $v_\delta$ be defined by \eqref{eq:v_delta} in  Proposition~\ref{sym_eigenvectors} replacing $\phi_{\delta,1}$ by $\phi_\delta$ and $\phi_{\delta,2}$ by $\overline{\mathcal{S}\phi_\delta}$ (see also Remark \ref{RemAutreDefinitionvdelta}). 
%\begin{equation}
%\lim_{\delta \rightarrow 0}  \left| %\int_{\mathcal{C}_\delta^\sharp} %\varphi_\delta(-{\bf x}) \cdot %\nabla\varphi_\delta({\bf x}) d{\bf x} \right| % =  v_\delta (1, \imath)^t 
%\end{equation}
Then
$$
\lim_{\delta \rightarrow 0} |v_\delta| = \frac{\pi}{4L} (2n+1).
$$ 
 \end{proposition}
For the proof of this proposition, let $\lambda_{\delta}$ be an eigenvalue of $A_{\delta,1}({\bf K})$ and let $\lambda$ be the limit of $\lambda_\delta$ when $\delta$ goes to $0$. We know that $\lambda$ is a simple eigenvalue of $\cA_{1}({\bf K})$. The proof of Proposition~\ref{TheoremeLimitevdelta} results from three main steps. 
\begin{enumerate}
    \item[{\bf Step 1.}] We first construct, in Lemma~\ref{LemmeQuasiMode}, a quasi-mode  $\varphi_\delta$, namely an explicit approximation of an eigenvector of $A_{\delta,1}({\bf K})$ associated with $\lambda_{\delta}$. More precisely, we construct $\varphi_\delta\in H^1_{\bf K}(\Omega_\delta)\cap L^2_1(\Omega_\delta)$ such that there exist $\delta_0>0$ and a constant $C>0$ such that for all $\delta<\delta_0$ and for all $\varphi\in H^1_{\bf K}(\Omega_\delta)\cap L^2_1(\Omega_\delta)$
    \begin{equation}\label{Quasi-modes}
      \int_{\mathcal{C}_\delta^\sharp} \big(\nabla {\varphi}_\delta({\bf x}) \cdot \overline{\nabla \varphi({\bf x})} d{\bf x} - \lambda_{\delta} \,{\varphi}_\delta({\bf x})\overline{\nabla \varphi({\bf x})}\Big) d{\bf x} \leq {C} \sqrt{\delta} \left\| {\varphi}_\delta \right\|_{H^1(\mathcal{C}_\delta^\sharp)}  \left\| \varphi \right\|_{H^1(\mathcal{C}_\delta^\sharp)}.
      \end{equation} 
      We deduce, in Lemma \ref{LemmeQuasiMode}, that $\varphi_\delta$ is an approximation of a certain eigenvector $\phi_\delta$ (see ) which is used for the computation of $v_\delta$.
    \item[{\bf Step 2.}] We evaluate in Lemma~\ref{LemmeLimiteQuasiMode} the quantity
    \begin{equation}\label{quantity_quasimode}
       \int_{\mathcal{C}_\delta^\sharp} \varphi_\delta(S{\bf x}) \cdot \nabla\varphi_\delta({\bf x}) d{\bf x}. 
    \end{equation}
    \item[{\bf Step 3.}] We deduce Proposition~\ref{TheoremeLimitevdelta} from the previous two results.
\end{enumerate} 
 \paragraph{{\bf Step 1.} Construction of the quasi-mode $\varphi_\delta$} 
 
 We know that $\lambda$, the limit of $\lambda_\delta$ when $\delta$ goes to $0$, is a simple eigenvalue of $\cA_{1}({\bf K})$ and, by Proposition \ref{prop : OperateurReduitGraph}, also a simple eigenvalue of $\cA_{ND}$. There exists an $n\in\N$ such that $\lambda=\lambda_n^\star$ and an associated eigenvector is given by 
 $$
 \phi_0(t) = \frac{1}{\sqrt{L}}cos(\sqrt{\lambda} t).
 $$ 
 We can deduce from \eqref{DefinitionWidetildephi1} an eigenvector $\phi$ of $\cA_{1}({\bf K})$. It is natural to construct the quasi-mode $\varphi_\delta$ from $\phi$. To do so, we have to 'extend' $\phi$ on the periodicity cell $\mathcal{C}_\delta^\sharp$. For that, we  decompose $\mathcal{C}_\delta^\sharp$ into four junction regions (denoted $J^A_\delta, J^B_{i,\delta}, i\in \{0,1,2\}$) and three shrunken edges ($\tilde{e}_{i,\delta}, i\in \{0,1,2\}$) (see Figure~ \ref{fig:periodicitycell}) defined by
 \begin{figure}[h]
	\begin{center}
 \includegraphics[width=0.8\textwidth, trim=0cm 0cm 0cm 0cm, clip]
 {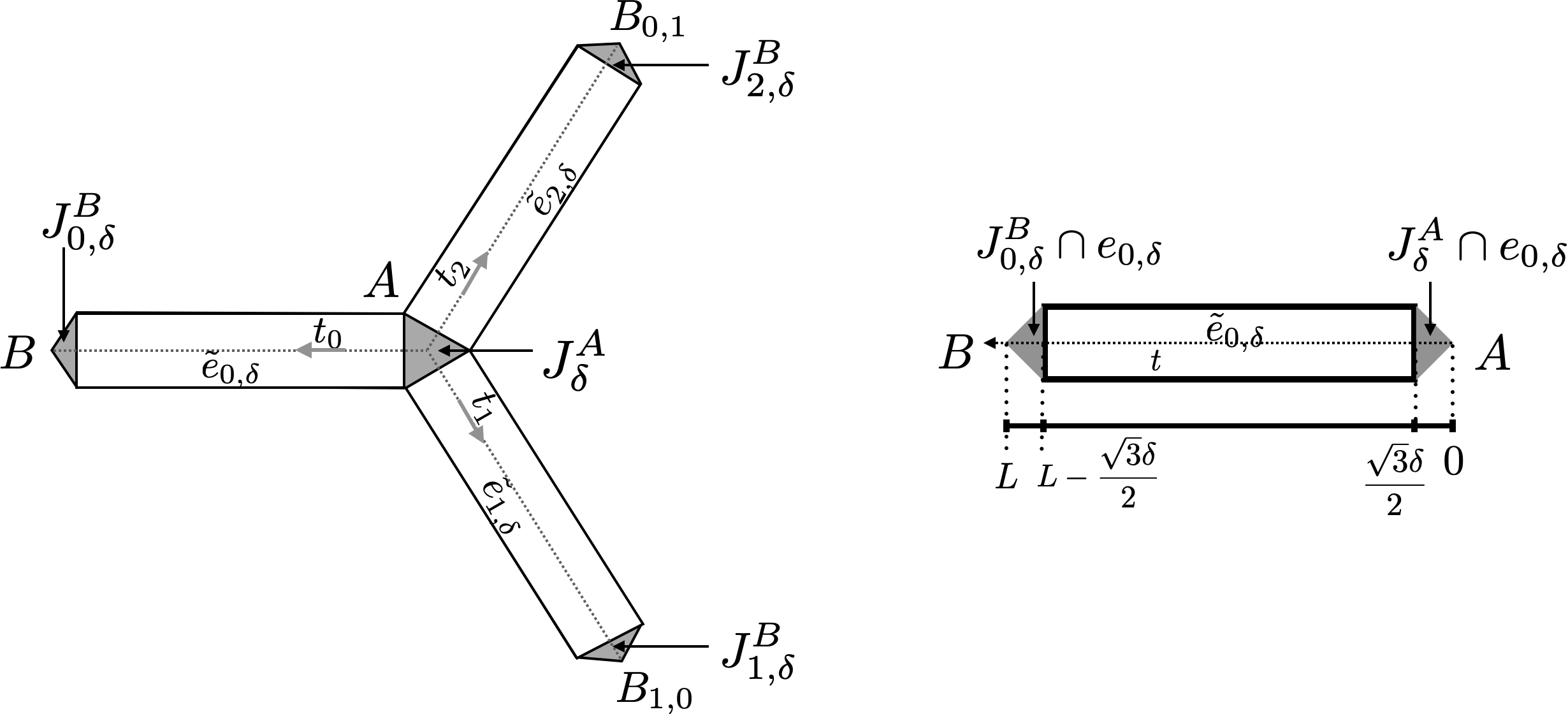}
	 \caption{Decomposition of the periodicity cell ${\cal C}_\delta^\sharp$ into three shrunken edges $\tilde{e}_{i,\delta}, i\in \{0,1,2\}$ and four junction regions $J^A_\delta, J^B_{i,\delta}, i\in \{0,1,2\}$.}
	 \label{fig:periodicitycell}
 	\end{center}\vspace{-0.8cm}
\end{figure}
\begin{equation*}
  \begin{array}{l}
\dsp \tilde{e}_{i,\delta} = {e}_{i,\delta} \cap \left\{ {\sqrt{3}\delta}/{2}  < t_i({\bf x}) < L -{\sqrt{3}\delta}/{2} \right\} \quad i\in \{0,1,2\},\\\dsp 
J^A_\delta = \bigcup_{i\in {0,1,2}} {e}_{i,\delta} \cap \left\{ t_i({\bf x}) < {\sqrt{3}\delta}/{2} \right\},\\
J^B_{i,\delta} = {e}_{i,\delta} \cap \left\{ t_i({\bf x}) > L - {\sqrt{3}\delta}/{2} \right\}, \; i\in \{0,1,2\},
  \end{array}
\end{equation*}
where $t_i$ is a local 'longitudinal' coordinate on each edge ${e}_{i,\delta},$  defined by:
 $
t_i({\bf x}) :=  \| P_i({\bf x}) - A \|_2, \; \xg \in  {e}_{i,\delta},  \; i \in \left\{0,1, 2\right\},
$
where $ P_0(x)$ (resp. $P_1$, $P_2$)  denotes the orthogonal projection of ${\bf x}$ on $[AB]$ (resp. $[AB_{1,0}]$, $[AB_{0,1}]$). 
  Let us also introduce the linear function $g$ that stretches $[0, L]$ into $[{\sqrt{3}\delta}/{2}, L-{\sqrt{3}\delta}/{2}]$ which is given by
 $
 g(s) = L  (s   - { \sqrt{3}\delta}/{2})(L -\sqrt{3} \delta )^{-1}. $ 
 We can now define the function $\hat{\varphi}_\delta\in H^1({\mathcal{C}}_\delta^\sharp)$ as 
\begin{equation}\label{DefinitionPhideltaHat}
\hat{\varphi}_\delta({\bf x}) = 
\begin{cases}
\phi_0(g(t_i({\bf x})), & {\bf x} \in \tilde{e}_{i,\delta} \quad i \in \{0,1,2\},  \\
1/\sqrt{L},  & {\bf x} \in J_\delta^A, \\
0, & {\bf x} \in J^B_{i,\delta} \quad i \in \{0,1,2\}.
\end{cases} 
\end{equation}
Using the same computation than in \cite[Section 5]{Article_DFJV_Part1}, we can show that  that there exist $\delta_0>0$ and $C>0$, such that, for any $\delta <\delta_0 $, for any $\varphi \in H^1(\mathcal{C}_\delta^\sharp)$
 \begin{equation}\label{EstimationElizaveta}
 \int_{\mathcal{C}_\delta^\sharp} \nabla \hat{\varphi}_\delta({\bf x}) \cdot \overline{\nabla \varphi({\bf x})}d{\bf x} - \lambda_n \int_{\mathcal{C}_\delta^\sharp} \hat{\varphi}_\delta({\bf x})\overline{\nabla \varphi({\bf x})} d{\bf x} \leq C \sqrt{\delta} \left\|\hat{\varphi}_\delta \right\|_{H^1(\mathcal{C}_\delta^\sharp)} \left\| \varphi \right\|_{H^1(\mathcal{C}_\delta^\sharp)}.
 \end{equation}
A direct calculation (see Lemma~\ref{LemmeEstimationNorme} in Appendix~\ref{AppendixTechnique}) shows that 
\begin{equation}\label{Normphi}
 \left\| \hat{\varphi}_\delta\right\|_{L^2(\mathcal{C}_\delta^\sharp)} = \frac{\sqrt{3 \delta}}{\sqrt{2}}   + O(\delta^{3/2}) \quad  \left\| \hat{\varphi}_\delta\right\|_{H^1(\mathcal{C}_\delta^\sharp)} = \frac{\sqrt{3 \delta}}{\sqrt{2}} \sqrt{ 1 + \lambda_n}  + O(\delta^{3/2}),
\end{equation}
so that 
\begin{equation}\label{DefinitionPhiDelta}
{\varphi}_\delta = {\hat{\varphi}_\delta}/{\left\| \hat{\varphi}_\delta\right\|_{L^2(\mathcal{C}_\delta^\sharp)} } ,
\end{equation}
 satisfies \eqref{Quasi-modes}.
It results from the previous equality that $\ {\varphi}_\delta$ is actually an approximation of an eigenvector associated with $\lambda_{n,\delta}$, as stated by the following Lemma proved in Appendix~\ref{AppendixTechnique} . 
\begin{lemma}\label{LemmeQuasiMode} There exists a normalized eigenmode $\phi_\delta$ of $A_{\delta,1}({\bf K})$ associated with the eigenvalue $\lambda_{\delta}$ such that
\begin{equation}\label{DefinitionReste}
\left\|  {\varphi}_\delta- {\phi}_\delta  \right\|_{H^1(\mathcal{C}_\delta^\sharp)}  \leq C \sqrt{\delta} \quad \mbox{ and } \quad \| \phi_\delta \|_{L^2(\mathcal{C}_\delta^\sharp)} = 1. 
\end{equation}
\end{lemma}
\paragraph{{\bf Step 2. Evaluation of the quantity given in \eqref{quantity_quasimode}}}
\begin{lemma}\label{LemmeLimiteQuasiMode}
   We have
\begin{equation*}
    \int_{\mathcal{C}_\delta^\sharp} \varphi_\delta(S{\bf{x}}) \cdot {\nabla\varphi_\delta}({\bf{x}}) d{\bf{x} }  = I_n
     \frac{1 - \imath \sqrt{3}}{2} \, \left( 1 , \imath \right)^T + { O(\sqrt{\delta})},\quad\text{where}\; I_n=\frac{\sqrt{\lambda_n}\sin(\sqrt{\lambda_n}L)}{2},
\end{equation*}  
 $S$ being the symmetry transformation defined in \eqref{DefinitionSym}.
\end{lemma}\vspace{-0.5cm}
\begin{proof}
First, let us remark that $I_n$ is nothing else but
\[
  I_n=\int_0^L \phi_0(L-t)\phi_0'(t)\,dt.\]
Since ${\varphi}_\delta$ is constant in each junction region, and since ${\varphi}_\delta=\hat{\varphi}_\delta/\|\hat{\varphi}_\delta \|_{L^2({\mathcal{C}}_\delta)}$ where $\hat{\varphi}_\delta$ is given by \eqref{DefinitionPhideltaHat} , we have
\begin{equation}
  \label{eq:etap1}
  \int_{\mathcal{C}_\delta^\sharp} \varphi_\delta (S{\bf x}) \cdot \nabla{\varphi}_\delta ({\bf x}) d{\bf x} = \frac{1}{\|\hat{\varphi}_\delta \|^2_{L^2({\mathcal{C}}_\delta)}}\;\sum_{i=0}^2 \int_{\tilde{e}_{i,\delta}} \hat{\varphi}_\delta (S{\bf x}) \cdot \nabla{\hat\varphi}_\delta ({\bf x})d{\bf x}.
\end{equation}
Since $\hat{\varphi}_\delta$ in each $\tilde{e}_{i,\delta}$ depends only on the longitudinal variable, we have
\begin{equation*}
  \nabla{\hat\varphi}_\delta ({\bf x}) = \tau_i\, \phi_0'(g(t_i({\bf x})))\, g'(t_i({\bf x})),\quad{\bf x}\in \tilde{e}_{i,\delta},\;i\in\{0,1,2\},
\end{equation*}
with
\begin{equation*}
  \tau_0=(-1, 0)^T,\;\tau_1=(1/2,- \sqrt{3}/{2})^T,\; \tau_2=(1/2 ,\sqrt{3}/{2})^T.
\end{equation*}
Since $S \tilde{e}_{0,\delta}=\tilde{e}_{0,\delta}+\vg_1-\vg_2$, $S\tilde{e}_{1,\delta}=\tilde{e}_{1,\delta}-\vg_2$ and $S\tilde{e}_{2,\delta}=\tilde{e}_{2,\delta}\vg_1-2\vg_2$ , we have
\begin{equation*}
  \hat{\varphi}_\delta (S\cdot)|_{\tilde{e}_{0,\delta}} = \hat{\varphi}_\delta |_{\tilde{e}_{0,\delta}} \,e^{2\imath\pi/3}\;,\hat{\varphi}_\delta (S\cdot)|_{\tilde{e}_{1,\delta}} = \hat{\varphi}_\delta |_{\tilde{e}_{1,\delta}} \,e^{-2\imath\pi/3}\;\text{and}\; \hat{\varphi}_\delta(S\cdot)|_{\tilde{e}_{2,\delta}} = \hat{\varphi}_\delta |_{\tilde{e}_{2,\delta}}.
\end{equation*}
Gathering the last three equalities, we obtain
\begin{equation}\label{EstimationIntegralePhi0}
  \int_{\mathcal{C}_\delta^\sharp} \hat{\varphi}_\delta (S{\bf x}) \cdot \nabla{\hat\varphi}_\delta ({\bf x})d{\bf x}
  = (\tau_0 e^{2\imath\pi/3}+\tau_1 e^{-2\imath\pi/3}+\tau_2)\delta I_n,
\end{equation}
where we have used that (note that $g(L - t) = L - g(t)$)
\[
  \int_{\frac{ \sqrt{3}\delta}{2}}^{L -\frac{\sqrt{3}\delta}{2} }   u( L-g( t)) u'(g(t)) g'(t) dt = I_n.
  \]
% Because $g(L - t) = L - g(t)$, we obtain by change of variable
A direct calculation shows that 
\begin{equation}\label{CalculSomme}
    \tau_0 e^{\frac{2 i \pi}{3} }+ \tau_1 e^{-\frac{2 i \pi}{3} }  + \tau_2 =\frac{3}{4} (1 - \imath \sqrt{3})  (1,\imath)^T.
\end{equation}
Collecting \eqref{eq:etap1}, \eqref{Normphi}, \eqref{EstimationIntegralePhi0} and \eqref{CalculSomme} ends the proof.
\end{proof}.\vspace{-1cm}
\paragraph{{\bf Step 3. Proof of Proposition ~\ref{TheoremeLimitevdelta}}}
 It remains to evaluate  $v_\delta$. We have
\begin{equation}\label{Decoupage}
    \int_{\mathcal{C}_\delta^\sharp} \phi_\delta(S{\bf x}) \cdot \nabla\phi_\delta({\bf x}) d{\bf x}   =  \int_{\mathcal{C}_\delta^\sharp} {\varphi}_\delta (S{\bf x}) \cdot \nabla{\varphi}_\delta ({\bf x}) d{\bf x} + \mathcal{R}^\delta,
\end{equation}
where, introducing $e_\delta = \phi_\delta - {\varphi}_\delta$, 
$$
\mathcal{R}^\delta = \int_{\mathcal{C}_\delta^\sharp}{\varphi}_\delta (S{\bf x}) \cdot \nabla e_\delta ({\bf x}) d{\bf x} + \int_{\mathcal{C}_\delta^\sharp}e_\delta (S{\bf x}) \cdot \nabla {\varphi}_\delta ({\bf x}) d{\bf x} + \int_{\mathcal{C}_\delta^\sharp} e_\delta (S{\bf x}) \cdot \nabla e_\delta ({\bf x}) d{\bf x}. 
$$
We directly deduce from Lemma~\ref{LemmeQuasiMode} that
$
    \left| \mathcal{R}^\delta\right|  \leq D\sqrt{\delta}.
$
and Lemma~\ref{LemmeLimiteQuasiMode} together with~\eqref{Normphi} give that
$$
 \int_{\mathcal{C}_\delta^\sharp} \varphi_\delta(-{\bf x}) \cdot \nabla\varphi_\delta({\bf x}) d{\bf x}  = 
     \frac{1 - \imath \sqrt{3}}2 I_n  (1,\imath)^T + O(\sqrt{\delta}).
$$
As a result, by the definition of $v_\delta$ (Proposition~\ref{sym_eigenvectors} and Remark~\ref{RemAutreDefinitionvdelta}), we obtain
$$
\lim_{\delta \rightarrow 0} |v_\delta| = \frac{1}{2}  \left( \frac{\pi}{2L} + n \frac{\pi}{L} \right).
$$ 	
\subsection{Numerical illustrations}\label{sec:numericalResultSpectreEssentiel}

We end this section by  numerical illustrations of Theorem~\ref{th:Diracpoints_delta}. In the following, the discretization is made using a  $P_1$ finite element method. The computation of the essential spectrum is standard since it only requires to solve a linear eigenvalue problem in a periodicity cell (with $\mathbf{k}$ quasi-periodic conditions on the lateral boundaries). \\

 In the case $\delta = 0.05$, the first three dispersion surfaces are represented on Figure~\ref{fig:dispersionSurf1}-(left).  We note that the third dispersion surface, located around the value $\lambda= (\frac{\pi}{L})^2$ is almost flat. This is due to the fact that $\lambda =  (\frac{\pi}{L})^2 \in \Sigma_D$ belongs to the essential spectrum of $\mathcal{A}(\mathbf{k})$ for any $\mathbf{k} \in \R^2$ (see Proposition~\ref{CharacSpectrumGraphHexa}). Additionally, as expected, conical points appear at the intersection between the first two dispersion surfaces (black circles on Figure~\ref{fig:dispersionSurf1}). 

\begin{figure}[htbp]
	\begin{center}\vspace{-0.3cm}
 \includegraphics[width=0.45\textwidth, trim=0cm 0cm 0cm 0cm, clip]{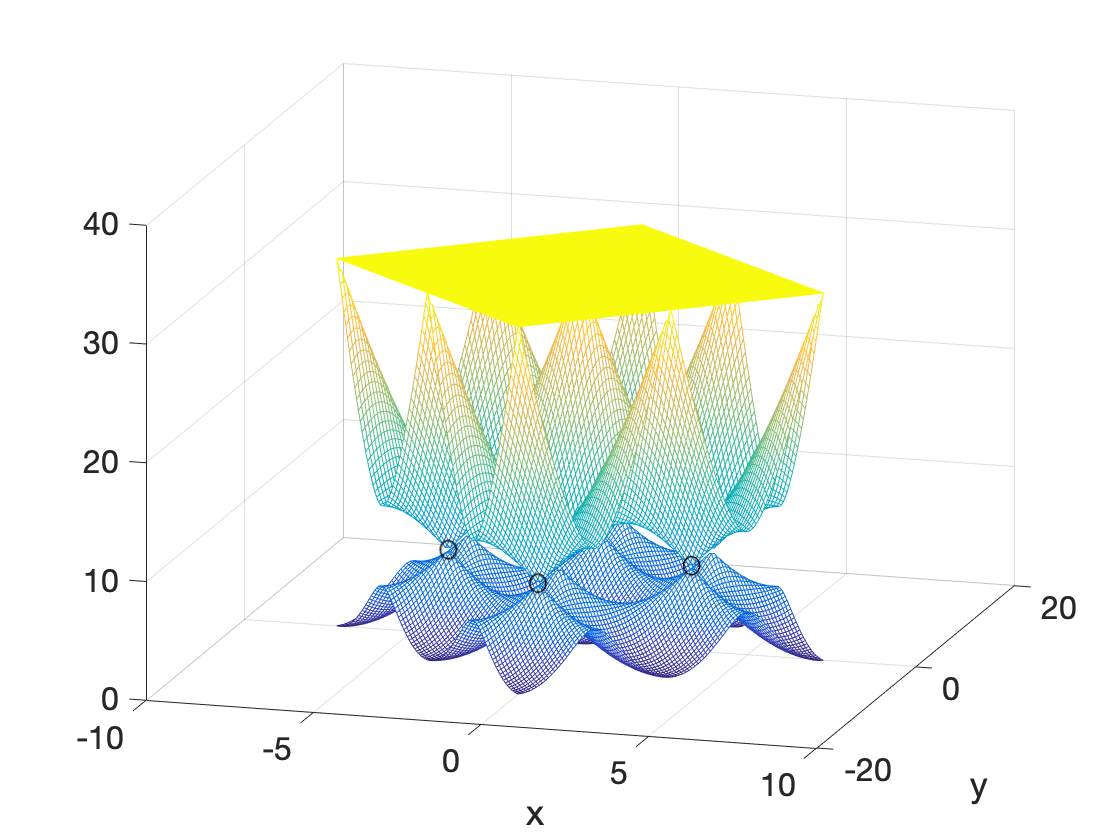}\hspace{1cm}
  \includegraphics[width=0.45\textwidth, trim=0cm 0cm 0cm 0cm, clip]{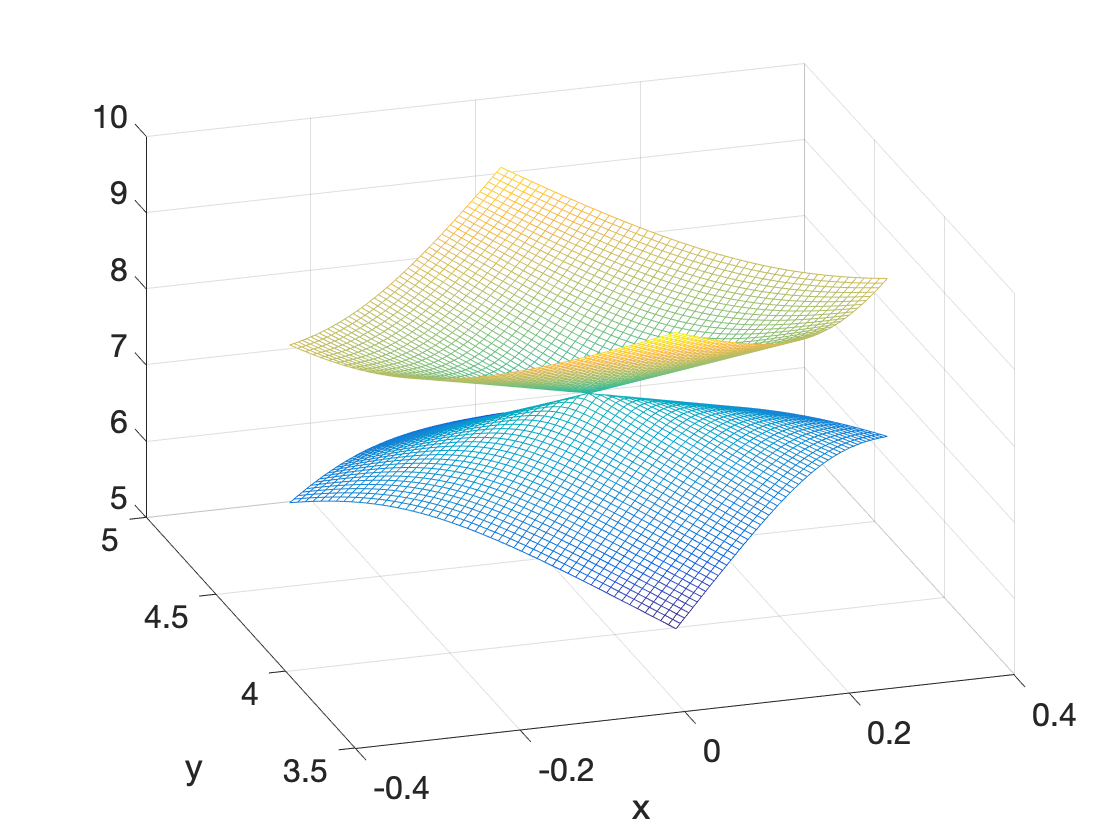}
  \caption{The first three dispersion surfaces for $\delta =0.05/\sqrt{3}$ (left); zoom  on the first two dispersion surfaces around $\mathbf{K} = \frac{2 \pi}{3} (\mathbf{v_2}^\ast - \mathbf{v_1}^\ast)$ (right)}
	 \label{fig:dispersionSurf1}
 	\end{center}\vspace{-0.5cm}
\end{figure}

\noindent Figure~\ref{fig:dispersionSurf1}-(right) is a zoom  around the quasi-momentum  $\mathbf{K} = \frac{2 \pi}{3} (\mathbf{v_2}^\ast - \mathbf{v_1}^\ast)$. Because $\mathbf{K}$ is a vertex of  the Brillouin zone,  a conical point is expected between the first two dispersion surfaces by Theorem \ref{th:Diracpoints_delta}. However, numerically, we observe a small gap between the first two eigenvalues: this is due to the fact that our mesh  does not respect the honeycomb symmetry.  Nevertheless, the size of the gap  goes to zero as $h$ goes to $0$ (as $h^2$ more precisely), as displayed in the Figure~\ref{fig:FigureGap1}-(left).  We point out that using a discretization respecting the honeycomb symmetry would preserve the conical point (in practice mesh generators do not produce meshes respecting this symmetry, though). Finally, we  evaluate (by a  naive first order finite difference) the coefficient $\alpha^\ast$ defined in~\eqref{eq:Dirac points}.  The results are displayed on Figure~\ref{fig:FigureGap1}-(right). As predicted by the theory (see Proposition~\ref{CharacSpectrumGraphHexa}),  $\alpha^\ast$ tends to  $\alpha_1 = {\pi^2}/{L}$.

\begin{figure}[htbp]
	\begin{center}\vspace{-0.5cm}
 \includegraphics[width=0.9\textwidth, trim=0cm 2cm 0cm 5cm, clip]{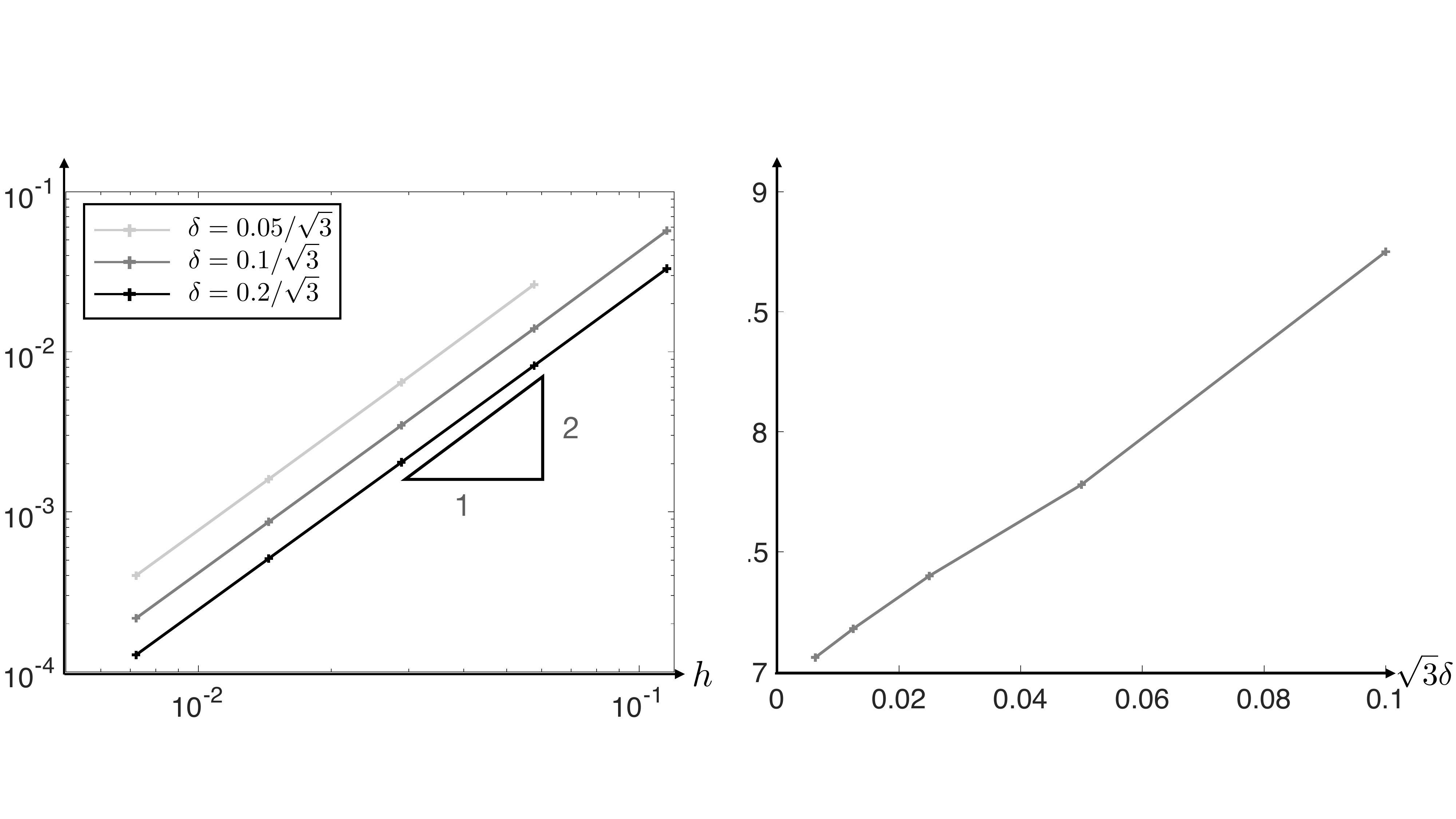}
  \caption{Evolution of the distance $|\lambda_2^{h, \delta}- \lambda_1^{h, \delta}|$ between the first two computed eigenvalues with respect to $h$ for $\mathbf{k} =\mathbf{K}$ for different values of $\delta$ (left). Evolution of $\alpha^\ast$ with respect to $\delta$ (right).}
	 \label{fig:FigureGap1}
 	\end{center}\vspace{-0.5cm}
\end{figure}

The numerical computations suggest that the conical point persists even for $\delta$ large. In Figure~\ref{fig:FigureDiracOfDelta} (left), we have estimated the frequency associated with the conical point for different values of $\delta$ :  for $\delta=0.2/\sqrt{3}$, and by Figure ~\ref{fig:FigureGap1} (left), the Dirac point seems to be still present.  Besides, as predicted by the theory, the accuracy of our limit  graph model is of order $\delta$: indeed, we observe on  Figure~\ref{fig:FigureDiracOfDelta}  that  $|\lambda^\delta-\lambda^\ast_0|/\lambda^\ast_0=  O(\delta)$ ($\lambda^\ast_0 = ({\pi}/{2L})^2$ being the frequency of the Dirac point for the limit graph model).

\begin{figure}[htbp]
	\begin{center}\vspace{-0.3cm}
  \includegraphics[width=0.5\textwidth]{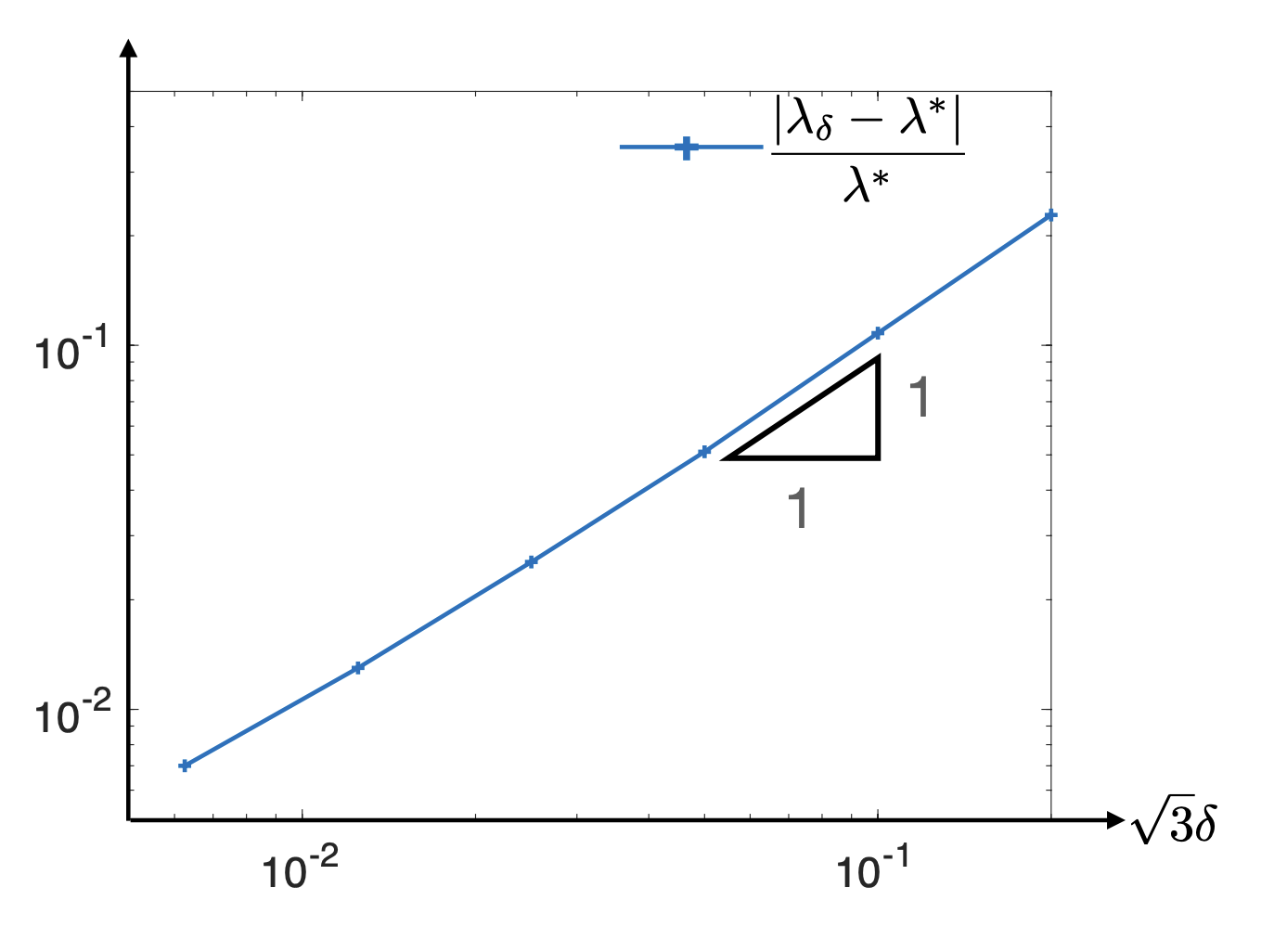}
   \caption{ 
  Log-log plot of  $\delta\mapsto|\lambda^\delta-\lambda^\ast_0|/\lambda^\ast_0$.  }
	 \label{fig:FigureDiracOfDelta}
 	\end{center}\vspace{-0.5cm}
\end{figure}

	\section{Guided modes in the perturbed geometries}\label{sec_guidedmodes}
  Following the theory of \cite{Post:2006, PostBook} and previous works on square lattices \cite{RR,Article_DFJV_Part1}, we know that existence of guided modes or equivalently eigenvalues/bound states for the operators $N_{\delta,\mu}^t(\beta)$ and $D_{\delta,\mu}^t(\beta)$ for any $\beta$ can be deduced for $\delta$ small enough, from the existence of eigenvalues/bound states for a certain limit operator defined in a limit graph that we are going to define. 

\subsection{Definition of the limit graphs and the corresponding limit operators}
\subsubsection{Limit geometry and associated functional spaces}
As $\delta$ tends to $0$, the domains $\Omega_{\delta,\mu}^0$ for $\mu>0$ and $\Omega_\delta^t$ tend respectively to the domains $\mathcal{G}_0$ defined in \eqref{eq:defGraph0} and $\mathcal{G}_t$ defined in \eqref{eq:defGrapht}. For $t\in\{0,L,2L\}$, let $\mathcal{E}_t$ be the set of the edges which are included in $\mathcal{G}_t$ and $\mathcal{V}_t$ be the set of the vertices of $\mathcal{G}$ that are included in $\mathcal{G}_t$: 
\begin{equation}
  \forall t\in\{0,L,2L\},\;\mathcal{E}_t:=\{\, e\in \mathcal{E}, e\subset \mathcal{G}_t\,\}, \quad \mathcal{V}_t:=\{\, M\in \mathcal{V}, M\in \mathcal{G}_t\,\}.
\end{equation}
The graph $\mathcal{G}_t$ contains, for $t\in[0,L]$ the union of the edges in $\mathcal{E}_{L}$, and for $t\in[L,2L]$ the union of the edges in $\mathcal{E}_{2L}$ and also a set of truncated edges denoted $\mathcal{E}^T_t$ :
\begin{equation}
    \forall t\in[0,L],\; \mathcal{G}_t = \bigcup \{\overline{e}, \;e\in \mathcal{E}_{L}\cup\mathcal{E}^T_t\},\quad \forall t\in[L,2L],\; \mathcal{G}_t = \bigcup \{\overline{e}, \;e\in \mathcal{E}_{2L}\cup\mathcal{E}^T_t\},
\end{equation}
where 
\begin{multline}\label{eq:param_edgetruncated}
  \mathcal{E}_t^T:=\begin{cases}\{\, e_{j,t}^T-\vg_1+\Z (\vg_1-\vg_2),j\in\{1,2\}\,\},\;\forall t\in [0,L]\\
     e_{0,t}^T+\Z (\vg_1-\vg_2),\;\forall t\in [L,2L]
  \end{cases},\\
  \text{with}\quad \begin{array}{|l}
    e_{0,t}^T:=\{{\bf x}\in\R^2, \text{s.t.}\;{\bf x}=A (1-s/L)+ Bs/L,\; s\in(0,2L-t)\}, \\
    e_{1,t}^T:=\{{\bf x}\in\R^2, \text{s.t.}\;{\bf x}=A (1-s/L)+B_{1,0} s/L,\; s\in(t,L)\},\\
    e_{2,t}^T:=\{{\bf x}\in\R^2, \text{s.t.}\;{\bf x}=A (1-s/L)+B_{0,1} s/L,\; s\in(t,L)\},
  \end{array}
\end{multline}
For any $t$, the domains $\cG_t$ are 1-periodic in the ${\bf e}_y$-direction. We denote $\widehat{\cG}_t:=\cG\cap\{-\Becommente{L}/2\leq y\leq L/2\}$. 
Let us denote the sets $\widehat{\mathcal{E}}_t$, $\widehat{\mathcal{E}}_t^T$ respectively the set of edges and truncated edges included in $\widehat{\cG}_t$. Let us finally introduce
\[
 \partial{\mathcal{G}_t}:=\overline{\mathcal{G}_t}\cap\{{\bf x}=\alpha(t)\}.%=\begin{cases}  \{C_{j,t}+m(\vg_1-\vg_2),\;j\in\{1,2\},\;m\in\Z\}\\
 %\{C_{0,t}+m(\vg_1-\vg_2),\;m\in\Z\}\end{cases},\text{with}\; C_{j,t}\in \widehat{\cG}_t.
\]
In this section, we are interested in guided modes {(along the ${\bf e}_y$-direction)} for a fixed wavenumber $\beta$, {\it i.e.} $\beta-$quasi-periodic functions in the ${\bf e}_y$-direction (see Section~\ref{sec:mathFormulation}):
\begin{equation}
  \label{QP-graph} \forall t\in[0,2L],\;\forall {\bf x}\in\mathcal{G}_t,\;u({\bf x}+{\bf e}_y)=e^{2\imath\pi\beta}u({\bf x}).
\end{equation} 
Let us introduce the functional spaces 
$$
L^2(\widehat{\cG}_t) = \{\; u,\quad u \in L^2(e) \;\;\forall e \in{\mathcal{E}}_t, \quad \|u\|_{L^2(\widehat{\cG}_t)}^2:=\sum_{e \in \widehat{\mathcal{E}}_t} \| u \|_{L^2(e)}^2 < \infty\;\},
$$ 
and for all $\beta$,
$$
H^1_\beta(\widehat{\cG}_t) = \{ u \in \mathcal{C}(\widehat{\cG}_t)\text{ satisfies \eqref{QP-graph}},\quad u \in H^1(e)\;\;\forall e \in{\mathcal{E}}_t, \quad \|u\|_{H^1_\beta(\widehat{\cG}_t)}^2:=\sum_{e \in \widehat{\mathcal{E}}_t} \| u \|_{H^1(e)}^2 <  \infty\},
$$ 
$$H^2_\beta(\widehat{\cG}_t)=\{u\in H^1_\beta(\widehat{\cG}_t),\quad u\in H^2(e)\;\;\forall e\in \mathcal{E}_t,\quad \sum_{e \in \widehat{\mathcal{E}}_t} \| u \|_{H^2(e)}^2 <  \infty\}.$$
where $\mathcal{C}(\widehat{\cG}_t)$ denotes the set of continuous functions defined on $\widehat{\cG}_t$.

\subsubsection{Limit operators}
Let us now define the limit operators associated with the three cases defined in Section~\ref{sub:geometry}-\ref{sec:DefinitionZigzagPerturbeddomain}.
\vspace{-0.5cm}

\paragraph{\bf Case 1 and Case 3.} 
For all $\beta$, the limit operator of the operator $D_{\delta,\mu}^t(\beta)$ (resp. $N_{\delta,\mu}^t(\beta)$) for $t\in[0,2L]$ and $\mu=1$ is $\mathcal{D}_1^t(\beta)$ (resp. $\mathcal{N}_1^t(\beta)$) where these operators are defined as follows 
\begin{equation}\label{eq:A0_graph_dir}
  \begin{array}{|l}
\dsp {D}(\mathcal{D}_1^t(\beta)) = \{ u \in H^2_\beta(\widehat{\cG}_t),\; \sum_{e \in \mathcal{E}_t(M) } [u\big|_{e}]'(M) = 0, \; \forall M \in \mathcal{V}_t,\; u=0\;\text{on}\;\partial\cG_t \},\\[3pt]
\dsp \forall u\in {D}(\mathcal{D}_1^t(\beta)),\quad [\mathcal{D}_1^t(\beta) u]\big|_{e} = - [u\big|_{e}]'' \quad \forall e \in \mathcal{E}_t,  
  \end{array}
\end{equation}
and
\begin{equation}\label{eq:A0_graph_neu}
  \begin{array}{|l}
\dsp {D}(\mathcal{N}_1^t(\beta)) = \{ u \in H^2_\beta(\widehat{\cG}_t),\; \sum_{e \in \mathcal{E}_t(M) } [u\big|_{e}]'(M) = 0, \; \forall M \in \mathcal{V}_t,\; u'=0\;\text{on}\;\partial\cG_t \},\\[3pt]
\dsp \forall u\in {D}(\mathcal{N}_1^t(\beta)),\quad [\mathcal{N}_1^t(\beta) u]\big|_{e} = - [u\big|_{e}]'' \quad \forall e \in \mathcal{E}_t,  
  \end{array}
\end{equation}
where $\mathcal{E}_t(M)$  stands for the set  of edges lying in  $\mathcal{G}_t$ adjacent to the vertex $M$ and $[u\big|_{e}]'$ (resp. $[u\big|_{e}]''$) is the derivative (resp. the second derivative) of  the restriction of $u$ on each edge $e$ with respect to the local variable introduced in the parametrization of the edges given in \eqref{eq:param} and \eqref{eq:param_edgetruncated}. \vspace{-0.5cm}

\paragraph{\bf Case 2.}
{Here, we introduce a weight function $w^\mu:\mathcal{E}_0\;\rightarrow \; \R^+$ that is equal to $\mu$ on the "perturbed edges", namely the edges in $\mathcal{E}_0$ that are not in $\mathcal{E}_L$}:
\begin{equation}
  \label{eq:w_mu}
w_\mu(e)=1,\;\forall e\in\mathcal{E}_L,\quad w_\mu(e)=\mu,\;\forall e\in\mathcal{E}_0\setminus\mathcal{E}_L.
\end{equation}
This weight function mimics in the graph the function $d_\mu$ defined in \eqref{eq:d_mu} { for the thick graph-like domain $\Omega_{\delta,\mu}^{0}$.}
We can then introduce the functional space
$$
L^2(\widehat{\cG}_0;\mu) = \{ u, u\big|_{e} \in L^2(e) \;\forall e \in \mathcal{E}_0, \quad \|u\|_{L^2(\widehat{\cG}_0;\mu)}^2:=\sum_{e \in \widehat{\mathcal{E}}_0} w_\mu(e)\| u \|_{L^2(e)}^2 < + \infty\},
$$ 
% $$
% H^1_{\beta}(\widehat{\cG}_0;\mu) = \{ u \in \mathcal{C}(\cG_0)\text{ satisfies \eqref{QP-graph}}, u \in  L^2(\widehat{\cG}_0;\mu),  u'\in L^2_\mu(\cG_0)\},
% $$ 
% and $$H^2_\mu(\cG_0)=\{u\in H^1_\mu(\cG_0),\; u\in H^2(e)\;\forall e\in \mathcal{E}_0\}$$
and for all $\beta$, the limit operator associated to the operator $D_{\delta,\mu}^t(\beta)$ (resp. $N_{\delta,\mu}^t(\beta)$) for $t=0$ and $\mu>0$ is $\mathcal{D}^0_\mu(\beta)$ (resp. $\mathcal{N}^0_\mu(\beta)$) defined as follows: 
\begin{equation}\label{eq:A0mu_graph_dir}
  \begin{array}{|l}
\dsp {D}(\mathcal{D}^0_\mu(\beta)) = \{ u \in H^2_\beta(\cG_0), \; \sum_{e \in \mathcal{E}_0(M) } w_\mu(e)[u|_{e}]'(M) = 0, \; \forall M \in \mathcal{V}_0,\; u=0\;\text{on}\;\partial\cG_t \},\\[3pt]
\dsp \forall u\in {D}(\mathcal{D}^0_\mu(\beta)),\quad [\mathcal{D}^0_\mu(\beta) u]|_{e} = - [u|_{e}]'' \quad \forall e \in \mathcal{E}_0,
  \end{array}
\end{equation}
and
\begin{equation}\label{eq:A0mu_graph_neu}
  \begin{array}{|l}
\dsp {D}(\mathcal{N}^0_\mu(\beta)) = \{ u \in H^2_\beta(\cG_0), \; \sum_{e \in \mathcal{E}_0(M) } w_\mu(e)[u|_{e}]'(M) = 0, \; \forall M \in \mathcal{V}_0,\; u'=0\;\text{on}\;\partial\cG_t \},\\[3pt]
\dsp \forall u\in {D}(\mathcal{N}^0_\mu(\beta)),\quad [\mathcal{N}^0_\mu(\beta) u]|_{e} = - [u|_{e}]'' \quad \forall e \in \mathcal{E}_0.  
  \end{array}
\end{equation}
Note that the Kirchhoff's conditions have changed for all the vertices belonging to the "perturbed" edges, which correspond to the vertices $\{A_{-m,m-1},\;m\in\Z\}\cup \{B_{-m,m},\;m\in\Z\}$.

% {\bf Case 3.}
% For all $\beta$, the limit operator associated to the operator $A_{\delta,\mu}^t(\beta)$ for $t\in[0,2L]$ and $\mu=1$ is the operator $\cA^t(\beta)$, which is defined as follows 
% \begin{equation}\label{eq:At_graph}
%   \begin{array}{|l}
% \dsp {D}(\cA^t(\beta)) = \{ u \in H^2_\beta(\cG_t), \sum_{e \in \mathcal{E}_t(M) } [u|_{e}]'(M) = 0, \; \forall M \in \mathcal{V}_t,\; [u|_{e}]'(t)=0\;\forall e\in \mathcal{E}^T_t \},\\[3pt]
% \dsp \forall u\in {D}(\cA^t(\beta)),\quad [\cA^t(\beta) u]|_{e} = - [u|_{e}]'' \quad \forall e \in \mathcal{E}_t.  
%   \end{array}
% \end{equation}
%{\color{magenta}
%For the sake of simplicity, for all $\beta$,  for $t\in[0,2L]$ and $\mu>0$, we use the {\color{magenta} generic} notation $\mathcal{A}_\mu^t(\beta)$ for either  $\mathcal{N}_\mu^t(\beta)$ or $\mathcal{D}%_\mu^t(\beta)$),which corresponds to the operator $\mathcal{N}^t$ for $t\in[0,2L]$ and $\mu=1$ and the operator $\cA_\mu$ for $t=0$ and $\mu>0$. 
%}

\subsubsection{Properties of the spectrum of the limit $\mathcal{D}^t_\mu(\beta)$ and $\mathcal{N}^t_\mu(\beta)$ }
General results from \cite{PostBook,RR,Article_DFJV_Part1} show that the theorems \ref{th:ess_spec_beta_delta}, \ref{th:guided_case12_delta} and \ref{th:guided_case3_delta} can be directly deduced from the following results.
The first one deals with the location of the essential spectrum of $\mathcal{D}^t_\mu(\beta)$ and $\mathcal{N}^t_\mu(\beta)$ for any $t\in[0,2L]$, $\mu>0$ and for any $\beta\in[0,1/2)$. 
%\begin{proposition}\label{prop:ess_spec_beta}
%  For all $\beta\in[0,1/2)\setminus\{\frac{1}3\}$, for all $n$, there exists an interval $I^n(\beta)$ containing $\lambda_n^*$, defined in \eqref{eq:lambda_n_star}, which is one of the gaps of $\mathcal{D}^t_\mu(\beta)$ and $\mathcal{N}^t_\mu(\beta)$ for any $t\in[0,2L]$ and $\mu>0$.
%\end{proposition}
\begin{proposition}\label{prop:ess_spec_beta}
  For all $\beta\in[0,1/2)\setminus\{\frac{1}3\}$, for any $t\in[0,2L]$, for any $\mu>0$, for any $n\in \N$, the spectrum of $\mathcal{D}^t_\mu(\beta)$ and  $\mathcal{N}^t_\mu(\beta)$ has a gap $I^n(\beta)$ (independent of $\mu$ and $t$) that contains $\lambda_n^*$, where $\lambda_n^*$ is defined in~\eqref{eq:lambda_n_star}.
\end{proposition}
%We emphasize that the gap $I^n(\beta)$ is an interval . Moreover, it is the same for $\mathcal{D}^t_\mu(\beta)$ and  $\mathcal{N}^t_\mu(\beta)$.
The next results prove the existence of eigenvalues for the operators in the cases 1, 2 and 3.
\begin{proposition}[Existence of guided modes for Case 1 and Case 2.]\label{prop:guided_case12}
  Let $t=0$ and $\mu>0$. For all $\beta\in({1}/3,{1}/2)$, for all $n$, $\lambda_n^*$ is an eigenvalue of the operator $\mathcal{N}_{\mu}^0(\beta)$. { For all $\beta\in[0,{1}/3)$, for all $n$, $\lambda_n^*$ is an eigenvalue of the operator $\mathcal{D}_{\mu}^0(\beta)$.}
\end{proposition}
\begin{proposition}[Existence of guided modes for Case 3.]\label{prop:guided_case3}
  Let $t\in[0,2L)$ and $\mu=1$. For all $\beta\in[0,1/2)\setminus\{{1}/3\}$, let $I^n(\beta)$ be the gap of $\mathcal{D}_{1}^t(\beta)$ (resp. $\mathcal{N}_{1}^t(\beta)$) containing $\lambda_n^*$. Then, for all $n$, for all $\lambda\in I^n(\beta)$,  there exist $(2n+1)$ values of $t$ , $t_1^D(\beta), \ldots, t_{2n+1}^D(\beta)$ (resp.  $t_1^N(\beta), \ldots, t_{2n+1}^N(\beta)$) such that $\lambda$ is an eigenvalue of the operator $\mathcal{D}_{1}^t(\beta)$ (resp. $\mathcal{N}_{1}^t(\beta)$). Moreover,  for $\lambda=\lambda_n^\ast$, $t_1^D(\beta), \ldots, t_{2n+1}^D(\beta)$ (resp.  $t_1^N(\beta), \ldots, t_{2n+1}^N(\beta)$) are independent of $\beta$  on $[0,1/3)$ and $(1/3, 1/2)$. 
\end{proposition}

{The previous proposition is illustrated  in Figures~\ref{DessinBetaPiSurTrois} in the case of the operator $\mathcal{N}_{1}^t(\beta)$ for $\beta = 1/6$  and $\beta = 5/12$. In both cases, we see the existence of $2n +1$ spectral flows  through the gap $I^n(\beta)$.  Note that the blue points, representing eigenvalues independent of $\beta$, are not the same for $\beta=1/6$ ($\beta \in [0,1/3)$) and  for $\beta=5/12$ ($\beta \in (1/3, 1/2)$).
%The blue points materialize values of $t_k^n$ that are independent of $\beta$, leading  to flat eigenvalues when $\beta$ is varying (as demonstrated in Section~\ref{sec:DiscreteSepctrumeA0beta} for $t=0$). For instance, in the case $n=0$, this invariant corresponds to $t=L$ (bearded configuration) for  $\beta \in [0, \frac{1}{3}) $ and to $t=0$ for $\beta \in (\frac{1}{3},\frac{1}{2}[$ (classical zig-zag configuration). We shall see that there are exactly $2n+1$ invariants in the gap $I^n$. 
\begin{figure}[htbp]
         \begin{center}
               \includegraphics[width=0.45\textwidth,trim=17cm 5cm 17cm 4.5cm, clip]{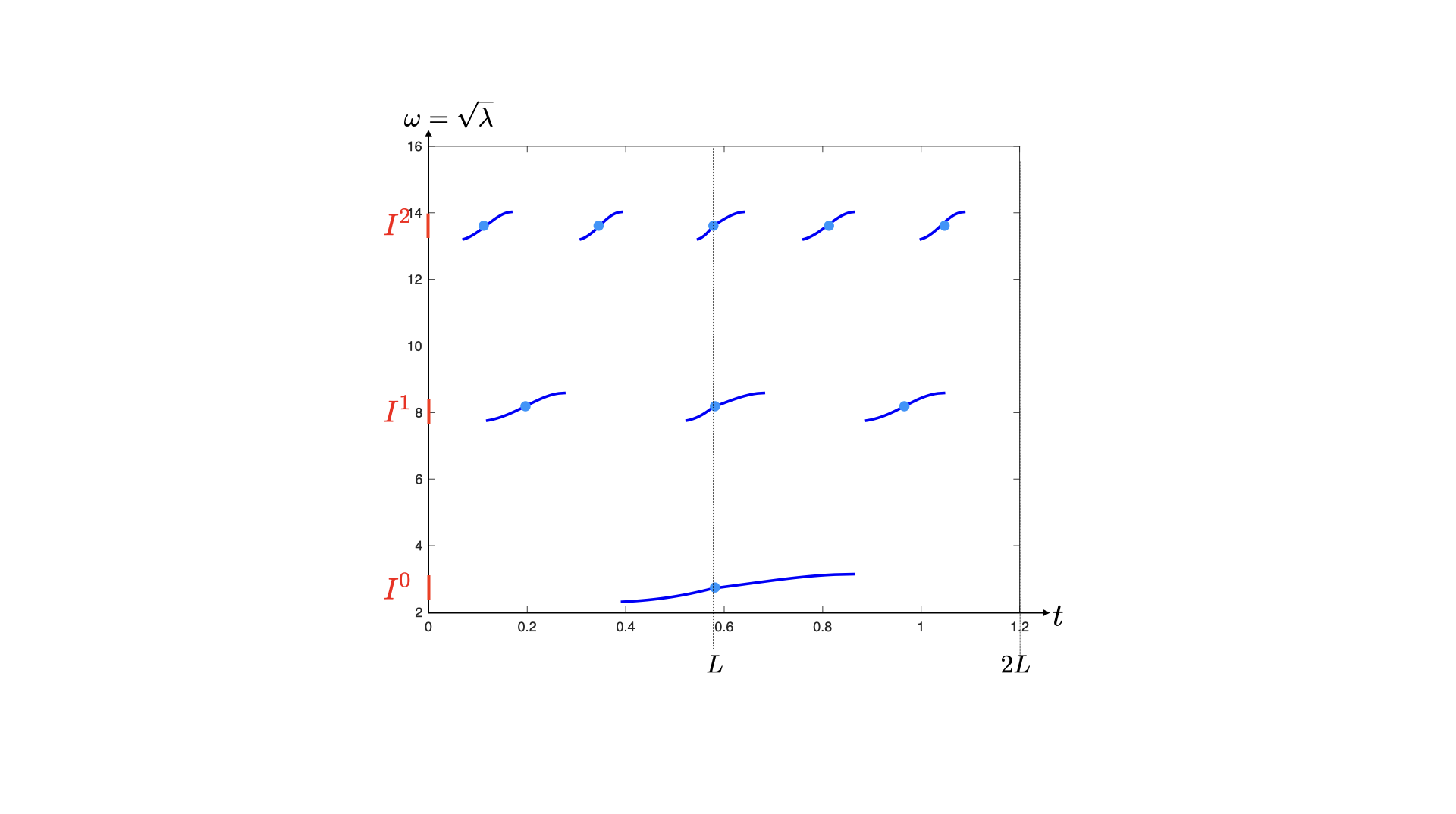}
             	 \includegraphics[width=0.47\textwidth,trim=17cm 5.3cm 17cm 4.5cm, clip]{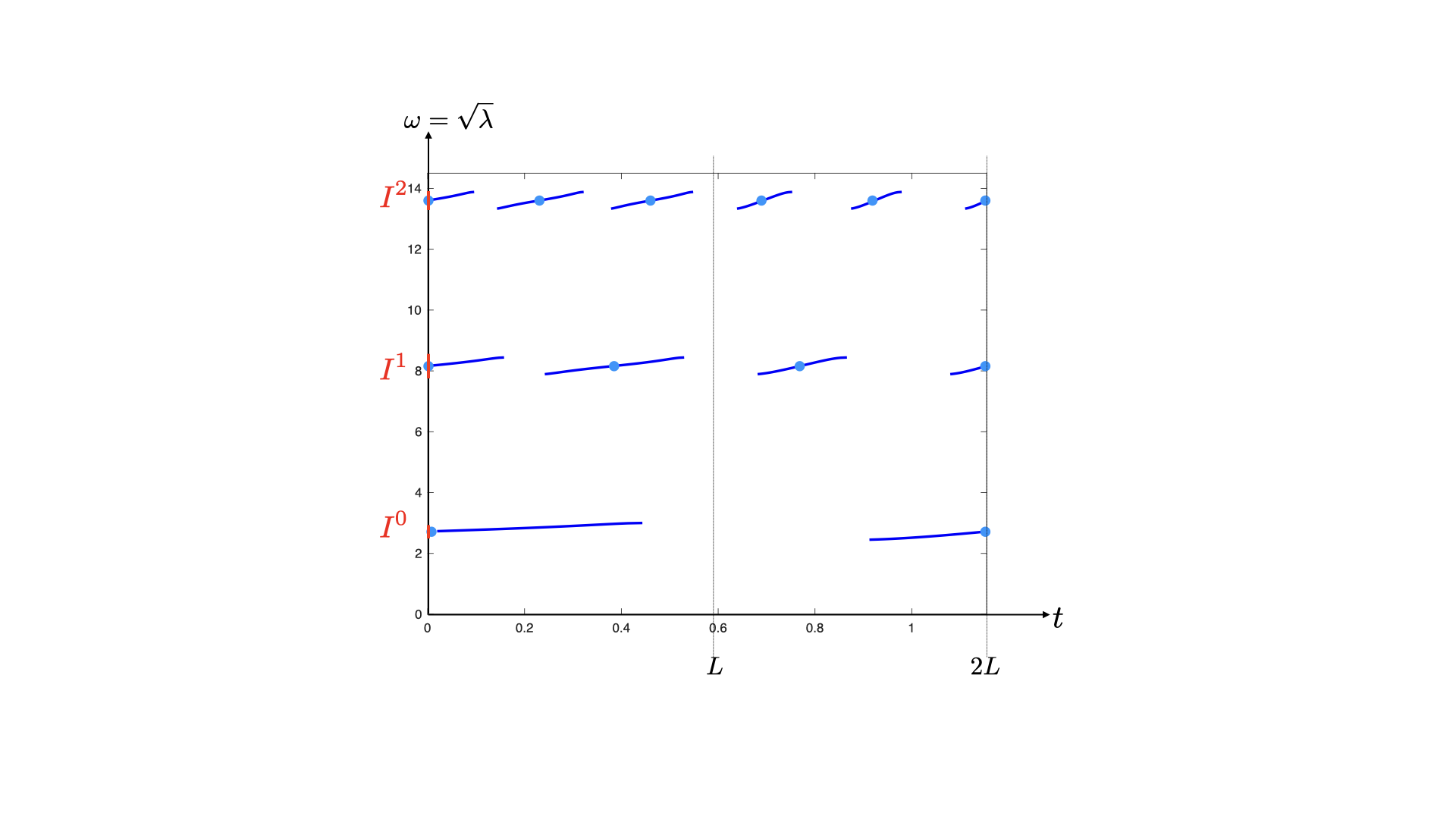}
              
                \caption{ For all $t\in[0,2L)$, representation of the first three eigenvalues of $\mathcal{N}_1^t(\beta)$ for $\beta = 1/6$ (left) and  $\beta = 5/12$ (right). The blue points stand for eigenvalues independent of $\beta$.}
                \label{DessinBetaPiSurTrois}
                \end{center}\vspace{-0.5cm}
 \end{figure}
\begin{remark}
The statement of Theorem~\ref{th:guided_case3_delta} differs from the one of Proposition~\ref{prop:guided_case3}, in Theorem~\ref{th:guided_case3_delta}, the frequency $\lambda=\lambda_n^\ast$ being treated independently. The difficulty comes from the vicinity of  the junctions, namely around $t=0$, $t=L$ and $t=2L$  (see e.g~\cite{henrot2006extremum}[Section 2.3.4]). Indeed, it is not clear wether the dependence of  the eigenvalues of $N_{\delta,1}^t(\beta)$ (or $D_{\delta,1}^t(\beta)$) is continuous with respect to $t$, around $t=0$, $t=L$ and $t=2L$. It turns out that in the limit configuration, if $\lambda \in \sigma( \mathcal{N}_{1}^t(\beta))$  for $ t\in \{0, 2L,L\}$  then $\lambda=\lambda_n^\ast$, which explains why $\lambda=\lambda_n^\ast$ has to be studied separately.  In addition, we shall see in Lemma~\ref{LemmeOmegaStar} that  $ \lambda_n^\ast \in \sigma( \mathcal{N}_{1}^t(\beta))$ for $t=0$ if and only if $\beta\in(0,1/3)$ and $ \lambda_n^\ast \in \sigma( \mathcal{N}_{1}^t(\beta))$ for $t=L$ if and only if $\beta\in(1/3,1/2)$. We then only have to exclude $t=0$ or $t=L$, when applying the asymptotic results for $\lambda=\lambda_n^\ast$. It then leads to the existence of at least  $2n$ points such that $ \lambda^\ast\in  \sigma(N_{\delta,1}^t(\beta)) $, instead of $2n+1$ for $\lambda\neq \lambda_n^\ast$.
\end{remark}\vspace{-0.3cm}
\subsection{Proof of Proposition \ref{prop:ess_spec_beta}}In this subsection $\mathcal{A}^t_\mu(\beta)$ stands for the operators $\mathcal{D}^t_\mu(\beta)$ and $\mathcal{N}^t_\mu(\beta)$. 
	For all $\beta\in[0,1/2]$, the essential spectrum of $\mathcal{A}^t_\mu(\beta)$, denoted $\sigma_{\text{ess}} \left(\mathcal{A}^t_\mu(\beta)\right)$,  is independent of $t$ (since $\mathcal{A}^t_\mu(\beta)$ is a compact perturbation of $\mathcal{A}^0_1(\beta)$, \cite[Chapter 9]{BirmanSolomjakBook}, Prop 1. in~\cite{Article_DFJV_Part1}). Moreover, it can be easily deduced from the essential spectrum of the family of reduced operators $\{\mathcal{A}({\bf k}),\;{\bf k}\in \mathcal{B}\}$ introduced in \eqref{eq:op_k_graph} that
	\begin{equation}\label{eq:firstCharacterizationEssentialSpectrumAtbeta}
	\sigma_{\text{ess}} \left(\mathcal{A}^t_\mu(\beta)\right) = \bigcup_{\{\mathbf{k} = k_1 \mathbf{v}_1^\ast +  k_2\mathbf{v}_2^\ast \in \R^2 \, \text{s.t} \,  k_2 - k_1 = \beta \}}\sigma(\mathcal{A}(\mathbf{k}))=\bigcup_{\{\mathbf{k} \in \R \mathbf{e}_x +\beta \mathbf{v}_2^\ast \}}\sigma(\mathcal{A}(\mathbf{k})). 
	\end{equation}
  where we have used in the last equality that $\mathbf{v}_1^\ast -\mathbf{v}_2^\ast$ is collinear to $\mathbf{e}_x$. Using the proof of Proposition \ref{CharacSpectrumGraphHexa}, we deduce that
  \begin{equation}\label{NpisurL}
    \Sigma_D:=\{ (n\frac{\pi}L)^2,\;n\in\N\}\subset\sigma_{\text{ess}} \left(\mathcal{A}^t_\mu(\beta)\right),
  \end{equation}
    and
    \[
      \forall \lambda\notin\Sigma_D,\;\lambda\in \sigma_{\text{ess}} \left(\mathcal{A}^t_\mu(\beta)\right) \quad\Leftrightarrow\quad \exists k, \; \cos^2 \sqrt{\lambda} L = \frac{1}{9} | 1 +  e^{2\imath\pi k} + e^{2\imath\pi(k+\beta)}|^2.
      \]
      This allows to prove Proposition \ref{prop:ess_spec_beta}. Indeed, since $\cos(\sqrt{\lambda_n^* }L)=0$ for all $n$, for all $\beta\in[0,1/2)$, as explained in \eqref{DispersionRelation0} we have
      \[
        | 1 +  e^{2\imath\pi k} + e^{2\imath\pi(k+\beta)}|^2=0\quad \text{only if}\;\beta=1/3\;\text{for}\;k=1/3.
        \]
         This implies that for $\beta\in[0,1/2)\setminus\{1/3\}$, for all $n$, $\lambda_n^*$ is not in $\sigma_{\text{ess}} \left(\mathcal{A}^t_\mu(\beta)\right)$. Since $\sigma_{\text{ess}} \left(\mathcal{A}^t_\mu(\beta)\right)$ is closed, there exists for all $n$, an interval $I^n(\beta)$ containing $\lambda_n^*$ which is included in a gap of $\mathcal{A}^t_\mu(\beta)$.

	\subsection{Proof of Proposition \ref{prop:guided_case12} - Case 1}\label{sec:DiscreteSepctrumeA0beta}

  We suppose here that $t=0$ and $\mu=1$. {We make the proof for the operator $ \mathcal{N}_1^0(\beta)$ and indicate shortly the modification for  $ \mathcal{D}_1^0(\beta)$ at the end of this section.} 
Assume that $\lambda  \in\sigma_d\left( \mathcal{D}_1^0(\beta)\right)$ and let us denote by $u$ an associated eigenvector. By \eqref{NpisurL}, $\lambda \notin \Sigma_D$. As explained in the proof of Proposition \ref{CharacSpectrumGraphHexa}, it is sufficient to know $u$ at each vertex of the truncated graph $\mathcal{G}_0$ to know it everywhere. Indeed, since $- u'' = \lambda u $ on each edge of the graph, using the parametrization \eqref{eq:param}, we have
\begin{equation}\label{eq:solutionFonda3}
u\big|_{e}(s) = u\big|_{e}(0) \frac{\sin(\sqrt{\lambda} (L-s))}{\sin(\sqrt{\lambda} L)} +  u\big|_{e}(L) \frac{\sin(\sqrt{\lambda} s)}{\sin(\sqrt{\lambda} L)}.
\end{equation}
Moreover, in view of the $\beta$-quasi-periodicity condition \eqref{QP-graph}, we have
\begin{equation}\label{eq:QP_vertex}
  \forall p\in\Z,\quad\begin{array}{ll}\forall n\geq -1,& u(A_{-p,n+p})=e^{2\imath\pi\beta p} u(A_{0,n}),\\[3pt]
    \forall n\geq 0, & u(B_{-p,n+p})=e^{2\imath\pi\beta p} u(B_{0,n}),
  \end{array}\end{equation}
so that it is enough to compute $u$ at the points $A_{0,n}$ for $n\geq -1$ and $B_{0,n}$ for $n\geq 0$. Let us denote
$$
\forall n\geq -1,\; u_n := u(A_{0,n}) \quad \text{and}\quad  \forall n\geq 0,\; v_n := u(B_{0,n}).
$$	
The Kirchhoff conditions at the points $\{B_{0,n},\;n\geq 0\}$ give
\begin{equation}\label{eq:KirchoffCondionB0n}
  \forall n\geq 0,\quad - 3 v_n \cos(\sqrt{\lambda}L) +  (1 + e^{2\pi\imath\beta}) u_{n-1}  + u_{n} =0,
\end{equation}
while the Kirchhoff conditions at the points $\{A_{0,n}),\;n\geq 0\} $ lead to 
\begin{equation}\label{eq:KirchoffCondionA0n}
  \forall n \geq 0,\quad - 3 u_n \cos(\sqrt{\lambda}L) +  (1 + e^{- 2\pi\imath \beta}) v_{n+1}  + v_{n} =0.
\end{equation}
Finally, the Kirchhoff condition at $A_{0,-1}$ give
\begin{equation}\label{eq:KirchoffCondionA0-1}
  \forall n \geq 0,\quad - 3 u_{-1} \cos(\sqrt{\lambda}L) +  (1 + e^{- 2\pi\imath \beta}) v_{0} =0.
\end{equation}
\noindent Suppose first that $\beta = 1/2$, then \eqref{eq:KirchoffCondionB0n}-\eqref{eq:KirchoffCondionA0n} rewrite as
\begin{equation}\label{eq:KirchoffPi}
\left[
\begin{array}{cc}
- 3  \cos \sqrt{\lambda} L & 1 \\
1 & - 3  \cos \sqrt{\lambda} L  
\end{array}
\right]
\left[
\begin{array}{c}
u_{n} \\
v_{n}
\end{array}
\right] = \left[
\begin{array}{c}
0 \\
0
\end{array}
\right] \quad \forall n \in \N.
\end{equation}
and \eqref{eq:KirchoffCondionA0-1} rewrites as
$$
- 3 u_{-1} \cos(\sqrt{\lambda}L)=0. 
$$ 
When $\lambda=\lambda_n^*$,  $\cos(\sqrt{\lambda_n^*}L)=0$ so 
$u_n =v_n=0$ for any $n\in \N$ and $u_{-1}$ can take any value.  Therefore $\lambda_n\in \sigma_d\left( \mathcal{D}_1^0(\pi)\right)$. We remark that in this case, the eigenvector $u$ is compactly supported on the edges of $\mathcal{E}_0\setminus\mathcal{E}_t$.  \\

 Suppose now that $\beta \neq 1/2$.  The Kirchhoff conditions~(\ref{eq:KirchoffCondionB0n},\ref{eq:KirchoffCondionA0n})  lead to the  recurrence equation
\begin{equation}\label{eq:recurrenceEquation}
\left[
\begin{array}{c}
u_{n+1} \\
v_{n+1}
\end{array}
\right]= M(\lambda, \beta) \left[
\begin{array}{c}
u_{n} \\
v_{n}
\end{array}
\right] \quad \forall\, n\geq 0,
\end{equation}
where
\begin{equation}\label{eq:reccurence}
M(\lambda, \beta) 
% =  \left[
% \begin{array}{cc}
%  1 & -3 \cos \sqrt{\lambda} L \\
% 0  & (e^{- \imath \beta} +1)
% \end{array}
% \right]^{-1}  \left[
% \begin{array}{cc}
%  -(e^{\imath \beta} +1) & 0 \\
% 3 \cos \sqrt{\lambda} L & -1
% \end{array}
% \right]  
:= \rho(\beta) M_r(\lambda, \beta),
\end{equation}
with $\rho(\beta) $ a complex number of modulus equal to $1$ given by
\begin{equation}\label{DefintionUbeta}
\rho(\beta) :=\frac{e^{\imath\beta} + 1}{\sqrt{2 + 2 \cos 2\pi\beta}},
\end{equation}
and $ M_r$  a matrix of determinant equal to $1$ given by
\begin{equation}\label{eq:definitionMr}
 \;   M_r(\lambda, \beta) =\left[
\begin{array}{cc}
\displaystyle \frac{9\cos^2 \sqrt{\lambda} L - 2 - 2 \cos 2\pi \beta }{\sqrt{2 + 2 \cos 2\pi \beta}} & -\displaystyle\frac{3 \cos \sqrt{\lambda} L }{\sqrt{2 + 2 \cos 2\pi \beta}}\\
\displaystyle\frac{3 \cos \sqrt{\lambda} L }{\sqrt{2 + 2 \cos 2\pi \beta}} &\displaystyle -  \frac{1}{\sqrt{2 + 2 \cos 2\pi \beta} }
\end{array}
\right].
\end{equation}
When $\lambda = \lambda_n^*$, the matrix $M(\lambda_n^*,\beta)$ becomes diagonal (the equations~\eqref{eq:KirchoffCondionB0n} and~\eqref{eq:KirchoffCondionA0n}  are uncoupled), and we obtain
$$
M(\lambda_n^*,\beta) = \rho(\beta)\left[
\begin{array}{cc}
\displaystyle  -\sqrt{2 + 2 \cos 2\pi \beta} &0 \\
0 &\displaystyle -  \frac{1}{\sqrt{2 + 2 \cos 2\pi \beta} }
\end{array}\right]. 
$$ 
The two eigenvalues of $M(\lambda_n,\beta)$ are then given by
\begin{equation}\label{definitionmu_1_mu_2}
r_A =   - \rho(\beta) \sqrt{2 + 2 \cos 2\pi \beta}  \quad \mbox{ and } r_B =  -  \frac{\rho(\beta)}{\sqrt{2 + 2 \cos 2\pi \beta} },
\end{equation}
and $\{u_n,\;n\in\N\}$ and $\{v_n,\;n\in\N\}$ follow a geometrical progression:
\begin{equation}\label{eq:suite_geom}
\forall n\in\N,\quad u_n =  (r_A)^{n} u_0, \quad\text{and}\quad  v_n =  (r_B)^{n} v_0.
\end{equation}
Since $|\rho(\beta)|=1$, note that 
$$
\begin{array}{l}
\dsp |r_A| < 1 \quad  \Leftrightarrow  \quad |r_B| > 1 \quad \Leftrightarrow  \quad \beta \in ({1}/{3},  {1}/{2} ),\\[3pt]
\dsp |r_A| > 1 \quad  \Leftrightarrow  \quad |r_B| < 1 \quad \Leftrightarrow  \quad \beta \in [0,{1}/{3}).
\end{array}
$$
Since the eigenvector $u$ has to be in $L^2(\mathcal{G}^0)$, the sequences $\{u_n,\;n \in\N \}$ and  $\{v_n,\;n \in\N \}$ cannot be exponentially increasing as $n$ goes to $+\infty$.  We deduce that 
\begin{equation}\label{eq:suite_geom2}
\dsp \beta \in ({1}/{3},  {1}/{2} ) \quad \Rightarrow \quad v_n = 0 \quad \forall n\geq 0,\quad\text{and}\quad
\dsp \beta \in [0,{1}/{3})\;\quad\Rightarrow \quad u_n = 0 \quad \forall n\geq 0.
\end{equation}
Finally, \eqref{eq:KirchoffCondionB0n}  for  $n=0$ (which is not taken into account in~\eqref{eq:recurrenceEquation}), and  \eqref{eq:KirchoffCondionA0-1} gives respectively when $\lambda=\lambda_n^*$
\begin{equation}\label{eq:CI}
u_{-1} =(r_A)^{-1} u_0\quad\text{and} \quad v_0=0. 
\end{equation}
 As a result, 
 \begin{equation}\label{eq:concl_spectred}
\dsp\beta \in ({1}/{3},  {1}/{2} ]\quad \Rightarrow \quad \lambda_n^*\in \sigma_d\left( \mathcal{N}_1^0(\beta)\right) \quad\text{and}\quad
\dsp \beta \in [0,{1}/{3})\quad\Rightarrow \quad \lambda_n^*\notin \sigma_d\left( \mathcal{N}_1^0(\beta)\right) .
\end{equation}

For the operator $\mathcal{D}_1^0(\beta)$, the relation ~\eqref{eq:KirchoffCondionA0-1} has to be replaced with
$
u_{-1} =0.$ 
Therefore, the same analysis leads to
 \begin{equation}\label{eq:concl_spectredD}
%\begin{array}{l}
\dsp\beta \in ({1}/{3},  {1}/{2} ]\quad \Rightarrow \quad \lambda_n^*\notin \sigma_d\left( \mathcal{D}_1^0(\beta)\right) \quad\text{and}\quad
\quad \dsp \beta \in [0,{1}/{3})\quad\Rightarrow \quad \lambda_n^*\in \sigma_d\left( \mathcal{D}_1^0(\beta)\right) .
%\end{array}
\end{equation}\vspace{-0.7cm}
\begin{remark}
		For the operator $\mathcal{N}_1^0(\beta)$, we  recover a well-known result of condensed Matter Physics  (see e.g.~\cite{kohmoto2007zero,Nakada:1996}). Indeed, the recurrence equations~\eqref{eq:KirchoffCondionB0n} and~\eqref{eq:KirchoffCondionA0n} are entirely similar to the  recurrence equation for the SSH model~\cite{SSH1979} or the tight-binding model for the graphene given by 
		$$
		\begin{cases}
		E  b_n + J' a_n + J a_{n+1} = 0  \\
		 E  a_n + {J'} b_{n} + \overline{J} b_{n-1} = 0  
		\end{cases}  \; \text{with }\; E = -3 \cos(\sqrt{\lambda L}), \quad  J = (1 + e^{\imath  2\pi  \beta}),  \quad J' = 1 .
		$$  
		For that model, the presence of 'flat' eigenmodes in the zigzag case is well-known: it is a direct consequence of the chirality of the associated discrete hamiltonian (see~\cite{shapiro2020bulk,Coutant:2022} and references therein). 
		This result is also known as a famous simple example of the Bulk-Edge correspondance as the presence of 'flat' eigenmodes  can be linked to non trivial values for topological invariants.  The literature being extremely vast and beyond the scope of this work, we refer for instance to~\cite{hatsugai1993chern, graf2013bulk, shapiro2020bulk},  proofs in the one-dimensional continuous are available  in~\cite{xiaoZhangChan2014} and \cite{linZhang2022mathematical},  two dimensional results may be found in~\cite{drouot2019bulk}. 
  \end{remark}
  \begin{remark}\label{rem:chirality}
For all $\beta \in ({1}/{3}, 1/2 )$, the eigenvectors of $\mathcal{N}_1^0(\beta)$  associated to $\lambda_n^*$ vanish at the $B-$points for while, for all $\beta \in  [0,\frac{1}{3})$, the eigenvectors of  $\mathcal{D}_1^0(\beta)$ associated to $\lambda_n^*$ vanish at the $A-$points . 
\end{remark}\vspace{-0.3cm}
	\subsection{Proof of Proposition \ref{prop:guided_case12} - Case 2}
	
    Here,  we suppose that $t=0$ and $\mu>0$. Suppose that $\lambda\in\sigma_d(\mathcal{N}^0_\mu(\beta))$ and $u$ is an eigenvector. We follow exactly the same proof as in Section \ref{sec:DiscreteSepctrumeA0beta} and use the same notation. By definition \eqref{eq:A0mu_graph_neu} of $\mathcal{N}^0_\mu(\beta)$, the only difference with Section \ref{sec:DiscreteSepctrumeA0beta} is the Kirchhoff condition at $B_{0,0}$. Namely, the conditions \eqref{eq:KirchoffCondionA0n} and \eqref{eq:KirchoffCondionA0-1} are still satisfied whereas \eqref{eq:KirchoffCondionB0n} is replaced by
  \[
    \begin{array}{l}
    \forall n\geq 1,\quad - 3 v_n \cos(\sqrt{\lambda}L) +  (1 + e^{2\pi\imath\beta}) u_{n-1}  + u_{n} =0,\\[3pt] \quad - (1+2\mu) v_0 \cos(\sqrt{\lambda}L) +  \mu(1 + e^{2\pi\imath\beta}) u_{-1}  + u_{0} =0.
    \end{array}
  \]
	This implies that \eqref{eq:suite_geom} and \eqref{eq:suite_geom2} still hold and \eqref{eq:CI} is similar replacing $r_A$ by $r_A\mu$. We conclude that \eqref{eq:concl_spectred} is still true for any $\mu>0$.  The analysis for $\mathcal{D}^0_\mu(\beta)$ is absolutely the same.
\begin{remark}\label{rem:different_mu}
Note that taking two different values of $\mu$ on the edges $[A_{0,-1} B_{0,0}]$ and $[A_{-1,0} B_{0,0}]$ leads exactly to the same result (see Figure \ref{fig:VecteurPropreZigZagDirichlet}-right).
    \end{remark}

	\subsection{Proof of Proposition \ref{prop:guided_case3}}
    
  We suppose now that $t>0$ and $\mu=1$. 
   The idea of the proof of Proposition \ref{prop:guided_case3} comes from \cite{Gontier:2020} where the author establishes a link between the edge states and the zeros (for Dirichlet boundary conditions at the edge) or the zeros of the derivative (for Neumann boundary conditions) of a particular solution of the differential equation, that we call the characteristic function of the graph. Here are the steps of the proof.
   \begin{enumerate}
    \item We first introduce in Section \ref{sub:step1} the so-called characteristic function of the graph denoted $\phi_{\omega,\beta}$ defined in the whole graph $\mathcal{G}_0$ which  is in $H^2(e)$ for each edge $e\in\mathcal{E}_0$ and in $\mathcal{C}(\widehat{G}_t)$ for all $t>0$. (we will see in particular that it is not continuous at the vertices $\{A_{m-1,-m},\,m\in\Z\}$ of the left boundary of $\widehat{G}_0$), is $\beta$-quasi-periodic (see \eqref{QP-graph}), satisfies the Kirchhoff conditions at each vertices of $\mathcal{G}_t$ for $t>0$ (but not at the vertices $\{A_{m-1,-m},\,m\in\Z\}$ of the left boundary of $\widehat{G}_0$) and satisfies finally on each edge 
    \[
      -[\phi_{\omega,\beta}\big|_{e}]''=\omega^2 [\phi_{\omega,\beta}\big|_{e}],\quad \forall e\in\mathcal{E}_0.
    \]
    \item We show in Proposition \ref{prop:lienavecleszeros} that $\lambda=\omega^2$ is an eigenvalue of $\mathcal{D}_\mu^t(\beta)$  (resp. $\mathcal{N}_\mu^t(\beta)$) if and only if $\phi_{\omega,\beta}$ (resp. the derivative of $\phi_{\omega,\beta}$) vanishes at $s=t$ on the edges of $\mathcal{E}_0\setminus \mathcal{E}_L$ if $0<t<L$ and on the edges of $\mathcal{E}_L\setminus \mathcal{E}_{2L}$ if $L<t<2L$.
    \item We investigate the zeros of $\phi_{\omega,\beta}$ and of its derivative in Proposition \ref{PropFlotSpectral}.
   \end{enumerate}
\subsubsection{Definition of the characteristic function}\label{sub:step1}
Let us begin this section by another characterization  of the essential spectrum of $\mathcal{A}^t(\beta)=\mathcal{D}^t_1(\beta)$ or $\mathcal{N}^t_1(\beta)$. Suppose $\lambda\in\sigma_\text{ess}(\mathcal{A}^t(\beta))\setminus\Sigma_D$ where $\Sigma_D$ is defined in \eqref{NpisurL} then there exists $u$ such that $\mathcal{A}^t(\beta)u=\lambda u$. This function $u$ satisfies \eqref{eq:solutionFonda3} on each edge of the graph, the $\beta$-quasi-periodicity condition \eqref{eq:QP_vertex} and the Kirchhoff conditions (\ref{eq:KirchoffCondionB0n},\ref{eq:KirchoffCondionA0n}) for $n$ large enough (using the same notation than in Section \ref{sec:DiscreteSepctrumeA0beta}). This leads when $\beta\neq 1/2$ and for $n$ large enough to the same recurrence equation \eqref{eq:recurrenceEquation} and the expressions \eqref{eq:suite_geom} 
where $r_A$ and $r_B$ are the eigenvalues of $M_r(\lambda,\beta)$ ,defined in \eqref{eq:definitionMr}, and then solutions of 
\begin{equation}\label{eq:definitionEquationRecurrente}
  r^2 - g_\beta(\lambda)r + 1 = 0,
  \end{equation}
  where for all $\beta\in[0,1/2)$
  \begin{equation}\label{DefinitionGbeta}
   g_\beta:\lambda\in\R^+ \mapsto  tr(M_r(\lambda, \beta))  = \frac{9 \cos^2(\sqrt{\lambda} L) -3 - 2 \cos 2\pi\beta}{\sqrt{2 + 2 \cos 2\pi\beta} } ,
    \end{equation}
  We show easily that $\lambda\in\sigma_\text{ess}(\mathcal{A}^t(\beta))\setminus\Sigma_D$ if and only if $|r_A|=|r_B|=1$ .
   In other words, for all $\beta\in[0,1/2)$ we have
$$
\lambda\in\sigma_{\text{ess}}(\mathcal{A}^t(\beta)) \quad\Leftrightarrow\quad \lambda\in \Sigma_D\;\;\text{or}\;\; |g_\beta(\lambda)|\leq 2.
$$
If $\beta=1/2$, we have
$$
\sigma_{\text{ess}}(\mathcal{A}^t(\beta)) = \Sigma_D \cup \Big\{  \Big( \frac{1}{L} \big( \arccos(\pm \frac{1}{3} )+ 2 k \pi \big) \Big)^2, \quad k\in \N \Big\}.
$$ 
 It turns out that the previous characterization provides us with an explicit definition of the gap $I_n(\beta)$ of $\mathcal{A}^t(\beta)$.  For the sake of simplicity, for all $\beta\in[0,1/2)$, we introduce $\tilde{g}_\beta:\omega\mapsto g_\beta(\omega^2)$ where $g$ is given in \eqref{DefinitionGbeta}. This function is periodic of period $\pi/L$, is strictly decreasing on $[0, {\pi}/{2L}]$ and it satisfies 
 \begin{equation}\label{eq:sym_gbeta}
  \tilde{g}_\beta(\omega)=\tilde{g}_\beta(\frac{\pi}L-\omega),\quad \omega\in [0,\frac{\pi}{2L}].
 \end{equation}
 We deduce in particular that $\tilde{g}_\beta$ is strictly increasing on $[ {\pi}/{2L}, {\pi}/{L}]$. Moreover, for any $\beta \in (0,1/2)\setminus\{1/3\} $ not equal to $1/3$, one can prove that 
$\tilde{g}_\beta(\omega) < -2 $ on the set  
$$\tilde{I}^0(\beta) = ]\frac{a(\beta)}L,\frac{\pi}L-\frac{a(\beta)}L[\quad\text{with}\;
a(\beta) = \arccos\frac{(-2 \sqrt{2 + 2 \cos 2\pi\beta} + 3 + 2 \cos 2\pi\beta )^{1/2}} {3}.
$$ 
In the sequel, we denote
\begin{equation}\label{DefintionGn}
  \tilde{I}^n(\beta) := \tilde{I}^0(\beta) +\frac{ n\pi}{L} \qquad n \in \mathbb{N}.
\end{equation}
We remark that $\omega_n^* \in \tilde{I}^n(\beta)$, $\omega_n^*$ being defined in \eqref{eq:lambda_n_star}. Consequently,  the gap $I^n(\beta)$ (in the spectrum of ${\cal A}^t(\beta)$) containing $\lambda_n^*$ is indeed given by  
$$ I^n (\beta)= \left] \left(\frac{a(\beta)}L + \frac{n\pi}{L}\right)^2, \left(\frac{\pi}L-\frac{a(\beta)}L + \frac{n\pi}{L}\right)^2\right[. $$ 
Since the value $\omega_n^*$ plays an important role in the sequel,
we separate $\tilde{I}^n(\beta)$ (resp. $I^n(\beta)$)  into two parts as follow:
$$\tilde{I}^0_-(\beta) = ]\frac{a(\beta)}L , \frac{\pi}{2L}[,  \quad \tilde{I}^0_+(\beta)= ] \frac{\pi}{2L} , \frac{\pi}{L}-\frac{a(\beta)}L [,\quad \tilde{I}^n_\pm(\beta) = \tilde{I}^0_\pm(\beta) +  \frac{n\pi}{L} \; n\in \mathbb{N},
$$
and
$$
I^n_\pm(\beta) =  \{\lambda \geq 0, \text{such that } \sqrt{\lambda} \in \tilde{I}^n_\pm(\beta)\}, \quad n\in \N.
$$
For $\beta  \notin \left\{ {1}/{3}, 1/2 \right\}$, if $\omega \in \tilde{I}^n(\beta)$, $ \tilde{g}_\beta(\omega)<-2$, and the recurrence equation~\eqref{eq:definitionEquationRecurrente} (written for  $\omega = \sqrt{\lambda}$)
has a real root $r(\omega,\beta)$ of modulus strictly smaller than one given by
\begin{equation}\label{definitionr}
r(\omega, \beta) = \frac{g_\beta(\omega)  + \sqrt{g_\beta(\omega)^2-4} }{2} .
\end{equation}
For $\beta  \notin \left\{ {1}/{3}, 1/2 \right\}$, $ r(\omega, \beta) $ is negative when $\omega\in\tilde{I}^n(\beta)$ and is equal to $-1$ for $\omega=a(\beta)/L+n\pi/L$ and $\pi/L-a(\beta)/L+ n\pi/L$. By \eqref{eq:sym_gbeta}, we deduce that
\begin{equation}\label{eq:sym_rbeta}
  r(\omega,\beta)=r(\frac{\pi}L+n\frac{\pi}L-\omega,\beta),\quad \omega\in \tilde{I}^n_-(\beta),
 \end{equation}
and by a direct differentiation of \eqref{definitionr}, we show that $\omega\mapsto r(\omega, \beta)$ is strictly increasing on $\tilde{I}^n_-(\beta)$ and then  strictly decreasing on $\tilde{I}^n_+(\beta)$. Its maximum value is attained at $\omega_n^*$ and is equal to
\begin{equation}\label{eq:r_max}
  \forall \omega\in\tilde{I}^n(\beta),\quad r(\omega,\beta)\leq r(\omega_n^*,\beta)\equiv\begin{cases}  \dsp - \frac{1}{\sqrt{2 + 2 \cos \beta }}&\text{if } \beta \in[0,1/3),\\
  - \sqrt{2 + 2 \cos \beta }&\text{if } \beta \in(1/3,1/2).
\end{cases}
\end{equation}
 \begin{figure}[htbp]
         \begin{center}\vspace{-0.5cm}
               \includegraphics[width=0.8\textwidth,trim=0cm 0cm 0cm 0cm, clip]{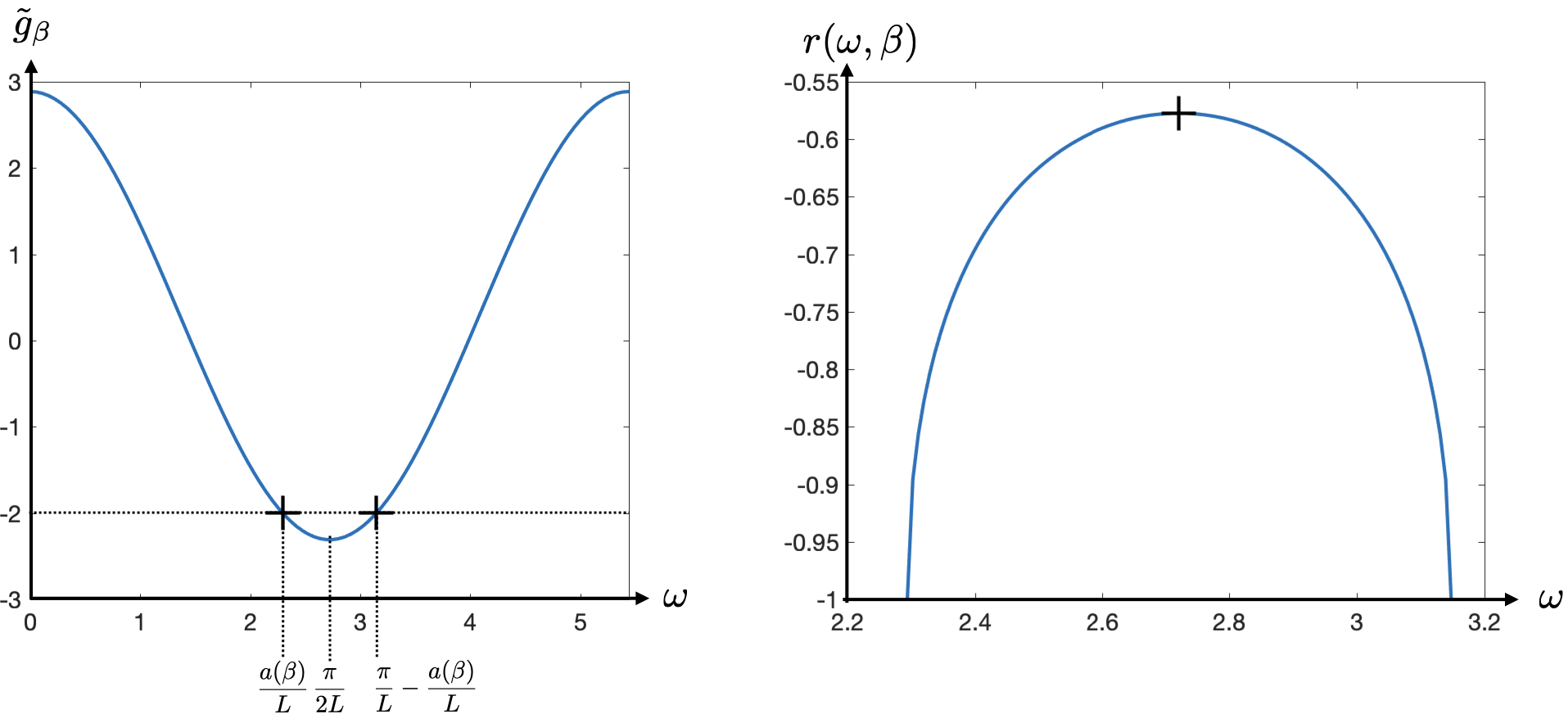}
                \caption{ Representation of $\tilde{g}_\beta$ (left) and $r(\omega, \beta) $  (right)  for $\beta = 1/3$ when $\omega \in \tilde{I}_0(\beta)$.}
                \label{Dessin2}
                \end{center}\vspace{-0.5cm}
 \end{figure}
\begin{remark}
 The gaps $I^n(\beta)$ are not the only gaps in the spectrum of ${\cal A}^t(\beta)$ but they are important in our study because they contain $\lambda_n^\ast$ defined in \eqref{eq:lambda_n_star}.
   \end{remark} 
   For $\omega \in \tilde{I}_n(\beta)$. The matrix $M(\lambda =\omega^2, \beta)$  defined in~\eqref{eq:reccurence}-\eqref{eq:definitionMr} has a unique eigenvalue $\mu(\omega, \beta)$ of modulus smaller than $1$, which is given by 
   \begin{equation}\label{eq:mubetaomega}
   \mu( \omega,\beta)  = \rho(\beta) r(\omega, \beta),
   \end{equation}
   $\rho(\beta)$ being the complex number of modulus one defined in~\eqref{DefintionUbeta}. 
   Let $\mathbf{e}(\omega,\beta) $ be the associated unit eigenvector defined (see~\eqref{eq:definitionMr}) as follows: $ \forall \beta\in[0,1/2)\setminus\{1/3\},\;\forall \omega\in \tilde{I}^n_-(\beta)$
   \begin{multline}\label{eq:definitioner1}
     \mathbf{e}(\omega,\beta) =\mathbf{e}(\frac{\pi}L+n\frac{\pi}L-\omega,\beta)=\frac{1}{   n(\omega,\beta) } \left[ \begin{matrix} 1 + r(\omega, \beta) \sqrt{2 + 2\cos(\beta)} \\ 3 \cos(\omega L)   \end{matrix} \right],\\
     \text{with}\quad n(\omega,\beta) = \left( 9 \cos^2(\omega L) + \left(1+ r(\omega, \beta) \sqrt{2 + 2\cos(\beta)}\right)^2\right)^{1/2}.
      \end{multline}
   By \eqref{eq:r_max}, we have that  $n(\omega,\beta)\neq 0$ for all $\beta$ and $\omega\neq\sqrt{\lambda_n^*}$ and $n(\sqrt{\lambda_n^*},\beta)=0$ only if $\beta\in[0,1/3[$. However, we show in the following lemma that for any  $\beta\notin\{1/3,1/2\}$, $\omega\mapsto\mathbf{e}(\omega,\beta)$ can be extended as a continuous and differentiable function on $\tilde{I}^n(\beta)$.
   \begin{lemma}\label{lemma:extensionContinuity}
   Let $\beta\notin\{1/3,1/2\}$. The function $\omega \rightarrow \mathbf{e}(\omega,\beta)$ can be continuously extended  at the point $\sqrt{\lambda_n^*}$ as follows:
   $$
\mathbf{e}(\omega_n,\beta) := \lim_{\omega \rightarrow \sqrt{\lambda_n^*}}   \mathbf{e}(\omega,\beta)  = \dsp \left[ \begin{matrix} 0 \\  1 \end{matrix} \right] \;\text{if }   \beta \in[0,{1}/{3})\quad\text{and}\quad
  \left[ \begin{matrix} 1 \\ 0 \end{matrix} \right]\;\text{if }  \beta \in({1}/{3},1/2).
   $$ 
   Moreover, the extended function $\omega \mapsto \mathbf{e}(\omega,\beta)$ is continuously differentiable on $\tilde{I}^n(\beta)$.
   \end{lemma}
   \begin{proof}
    Since $n(\sqrt{\lambda_n^*},\beta)$ does not vanish if $\beta\in[1/3,1/2[$, it is easy to obtain this result for $\beta\in[1/3,1/2[$. Let us now suppose that $\beta\in[0,1/3[$. By definition \eqref{eq:definitioner1}, it suffices to consider $\omega\in \tilde{I}^n_-(\beta)$.
  By a direct computation and 
% $$
% g^2(\omega)-4  =\frac{(1 + 2 \cos \beta)^2}{2 + 2 \cos \beta}  \left( 1 - \frac{18 (3 +2 \cos \beta) }{(1 + 2 \cos \beta)^2} \cos^2(\omega L)+  \frac{81 }{(1 + 2 \cos \beta)^2} \cos^4(\omega L )\right)
% $$ 
 since $1+ 2 \cos \beta >0$, we have
\begin{equation}\label{Asymptog2moins4}
\sqrt{g^2(\omega)-4 } =  \frac{(1 + 2 \cos \beta) }{\sqrt{2 + 2\cos \beta}}  - \frac{9 (3 +2 \cos \beta)  }{(1 + 2 \cos \beta)\sqrt{2 + 2\cos \beta}} \cos^2(\omega L ) + O( \cos^4(\omega L ) ).
\end{equation}
Therefore, in view of the definition~\eqref{definitionr} of $r(\omega, \beta) $, we have 
\begin{equation}\label{Asymptoromega}
r(\omega, \beta) =  - \frac{1}{ \sqrt{2 + 2\cos \beta} } -    \frac{9 (3 +2 \cos \beta) \cos^2(\omega L ) }{\sqrt{2 + 2\cos \beta}(1 + 2 \cos \beta)}  + O( \cos^4(\omega L ) ),
\end{equation}
and since $\omega\in \tilde{I}^n_-(\beta)$, we have
% $$
% 1 + r(\omega, \beta)\sqrt{2 + 2\cos \beta}  = -    \frac{9 (3 +2 \cos \beta)  \cos^2(\omega L ) }{(1 + 2 \cos \beta)}  + O( \cos^4(\omega L ) ).
% $$
% As a result
$$
\left( 9 \cos^2(\omega L) + \left(1+ r(\omega, \beta) \sqrt{2 + 2\cos(\beta)}\right)^2 \right)^{1/2}=  3 \cos(\omega L)  +  O( \cos^2(\omega L ) ),
$$ 
which leads to the desired result. For the continuity of the derivative, it suffices to show that 
$$
\lim_{h\rightarrow 0} \frac{\mathbf{e}(\sqrt{\lambda_n^*} -h,\beta) - \mathbf{e}(\sqrt{\lambda_n^*},\beta)}{h}\quad\text{exists.}
$$ 
The previous computations show that
$$
\mathbf{e}(\sqrt{\lambda_n^*} -h,\beta)\sim  \left[ \begin{matrix} \mathcal{O}(h) \\ 1 +\mathcal{O}(h^2) \end{matrix} \right],\quad\text{as}\; h\rightarrow 0,
$$
which allows to conclude.
\end{proof}
% \begin{figure}[htbp]
%          \begin{center}
%                \includegraphics[width=0.8\textwidth,trim=5cm 10cm 20cm 4cm, clip]{./GrapheG0Tilde.png}

%                 \caption{ Representation of truncated graph $\tilde{\mathcal{G}}_0$}.
%                 \label{DessinG0tilde}
%                 \end{center}
%  \end{figure}
Let us now define the characteristic function $\phi_{\omega,\beta}$ defined on the graph $\mathcal{G}_0$ for all $\beta\in(0,1/2)\setminus\{1/3\}$ and $\omega\in \tilde{I}^n(\beta)$ for all n. We define the function $\phi_{\omega, \beta}$  as follows: $\phi_{\omega, \beta}$ is in $H^2(e)$ for each edge $e\in\mathcal{E}_0$ and in $\mathcal{C}(\widehat{\cG}_t)$ for all $t>0$ and 
\[
  -[\phi_{\omega,\beta}\big|_{e}]''=\omega^2 [\phi_{\omega,\beta}\big|_{e}],\quad \forall e\in\mathcal{E}_0.
\]
with
\begin{equation}\label{eq:phib_n}
  \left[
 \begin{array}{c}
 \phi_{\omega, \beta}(A_{0,n})\\
 \phi_{\omega, \beta}(B_{0,n})
 \end{array}
 \right] := \mu(\beta, \omega)^{n} \mathbf{e}(\omega,\beta),\quad\forall n\in\N,
 \end{equation}
 $\phi_{\omega, \beta} $ is $\beta$-quasi-periodic (see \eqref{QP-graph} for the definition). 
By definition of $\mu(\omega,\beta)$ (see \eqref{eq:mubetaomega}) and ${\bf e}(\omega,\beta)$ (see \eqref{eq:definitioner1}), $\phi_{\omega, \beta} $ satisfies the Kirchhoff conditions at all the vertices of $\mathcal{G}_{2L}$. We impose finally that $\phi_{\omega,\beta}$ satisfies the Kirchhoff condition at $B_{0,0}$ and that
\[
  \phi_{\omega,\beta}\big|_{e_1-\vg_2}=\phi_{\omega,\beta}\big|_{e_2-\vg_1}.\]
This implies in particular that
 \begin{equation}\label{eq:phib_B00}
  \lim_{s\rightarrow0}\phi_{\omega,\beta}\big|_{e_1-\vg_2}(s) = \lim_{s\rightarrow0}\phi_{\omega,\beta}\big|_{e_2-\vg_1}(s) = \frac{1}{2} \left( 3 \cos(\omega L) \phi_{\omega, \beta}(B_{0,0}) - \phi_{\omega, \beta}(A_{0,0}) \right). 
 \end{equation}
 Let us make some remarks on this characteristic function.
 By \eqref{eq:phib_n} and since $|\mu(\beta,\omega)|<1$, we can show easily that $\phi_{\omega,\beta}\in H^1(\triangle,\widehat{\Omega}^t_\mu)$.
 Since $\phi_{\omega,\beta}$ is $\beta$-quasi-periodic and satisfies the previous relation, we easily notice that it is not continuous at the vertices $\{A_{m, -(m+1)},\,m\in\Z\}$. 

 For our purpose (which is the existence of edge states), it is sufficient to study $\phi_{\omega, \beta}$ on the edges $e_1-\vg_2$ and $e_0$. We then consider the continuous function
 \[
  \tilde{\phi}_{\omega,\beta}(s)=\begin{cases} \phi_{\omega,\beta}\big|_{e_1-\vg_2}(s)&s\in[0,L],\\
    \phi_{\omega,\beta}\big|_{e_0}(2L-s)&s\in[L,2L].
  \end{cases}
    \]
 To simplify, we abusively use the notation  $\phi_{\omega, \beta}$ for $\tilde{\phi}_{\omega,\beta}$. We can then give an explicit expression of $\phi_{\omega, \beta}$ deduced from~\eqref{eq:solutionFonda3}-\eqref{eq:phib_n}. Denoting by
 \begin{equation}\label{eq:values_phi}
  \begin{array}{l}
 u_0(\omega,\beta) = \phi_{\omega, \beta}(2L)\; (=\phi_{\omega,\beta}(A_{0,0})), \quad  v_0(\omega,\beta) =  \phi_{\omega, \beta}(L) \;(= \phi_{\omega, \beta}(B_{0,0})), 
 \\[3pt]\dsp v_0'(\omega,\beta)= \phi_{\omega, \beta}'(L^+)\equiv \frac{\omega}{\sin (\omega L)} \left(- v_0(\omega,\beta)  \cos(\omega L) + u_0(\omega,\beta) \right).
  \end{array}
 \end{equation}
 we have
 \begin{equation}\label{eq:definitionExplicitephiprime1}
 \phi_{\omega, \beta}(s)  =     v_0(\omega,\beta) \cos(\omega (L-s)) - \frac{v_0'(\omega,\beta)}{2\omega} \sin(\omega(L-s)), \;s\in[0,L],
 \end{equation}
 and
  \begin{equation}\label{eq:definitionExplicitephiprime2}
  \phi_{\omega, \beta}(s) =  \frac{1}{\sin(\omega L)} \left(  v_0(\omega,\beta)    \sin(\omega (2L-s ) + u_0(\omega,\beta)    \sin(\omega (s-L) \right),\; s \in [L, 2L].
  \end{equation} 
 We point out that, in general, $ \phi_{\omega, \beta}'$ is not continuous at $s =L$.
\subsubsection{Link between the discrete spectrum of the operators $\mathcal{D}_\mu^t(\beta)$ and $\mathcal{N}_\mu^t(\beta)$ and the characteristic function}
It is easy to see that if $ \phi_{\omega, \beta}(t)=0$ (resp. $ \phi_{\omega, \beta}'(t)=0$), then $\omega^2 \in \sigma_d(\mathcal{D}^t_\mu(\beta))$ (resp. $\omega^2 \in \sigma_d(\mathcal{N}^t_\mu(\beta))$), the associated eigenvector being  $ \phi_{\omega, \beta}\big|_{\mathcal{G}_t}$. It turns out that the converse statement is also true as stated by the following proposition.
\begin{proposition} \label{prop:lienavecleszeros}
  Let $t \in [0, 2L)$, $\beta \in (0, 1/2)\setminus\{ 1/3\}$ and $\omega\in \tilde{I}^n(\beta)$ for all $n$. Then
  $$
 \omega^2 \in \sigma_d({\cal D}^t_\mu(\beta))   \quad \Leftrightarrow   \quad \phi_{\omega, \beta}(t)=0\qquad\text{
and}\qquad
 \omega^2 \in \sigma_d({\cal N}^t_\mu(\beta))   \quad \Leftrightarrow   \quad \phi'_{\omega, \beta}(t)=0.
$$ 
\end{proposition} \vspace{-0.3cm}
\begin{proof}
  Let us show the first statement, the proof of the second one being similar.  We first assume  that $t \in (L, 2L)$. Let $ \omega^2 \in \sigma_d({\cal D}^t_\mu(\beta))$ and $u$ an associated eigenvector. As in Section~\ref{sec:DiscreteSepctrumeA0beta},  let 
$$\forall\, n\geq 0,\;u_n =u(A_{0,n}) \quad  \text{ and }\quad \forall \,n\geq 1,\; v_n =u(B_{0, n}).$$
Then, writing the Kirchhoff conditions at the nodes $B_{0,n}$ and $A_{0,n}$ (for $n\geq 1$) leads~\eqref{eq:KirchoffCondionB0n}-\eqref{eq:KirchoffCondionA0n}  for any $n\geq 1$.  In other words, the recurrence equation~\eqref{eq:recurrenceEquation} is valid for any $n\geq 2$. The eigenvector $u$ is in $ L^2(\mathcal{G}_0)$,  it cannot be exponentially increasing. As a result,  there exists a complex number $\alpha$ such that, for any $n\geq 1$,
$$
\left[
\begin{array}{l}
u_n \\
v_n 
\end{array}
\right] = \alpha \mu( \omega,\beta)^{n} \mathbf{e}(\omega,\beta). $$ 
It remains to defined $u$ on the edge $e_2$ and on the truncated one $e_{0,t}^T$, where $e_{0,t}^T$ is defined in \eqref{eq:param_edgetruncated}. The Kirchhoff condition~\eqref{eq:KirchoffCondionB0n} at the node $B_{0,1}$  gives 
$$
u_{0} = \frac{-u_1 + 3 v_1 \cos(\omega L)}{1 + e^{2\imath \pi\beta}},
$$ 
which defines $u$ on  $e_2$. The expression \eqref{eq:solutionFonda3} of $u$ on $e_2$ and $e_1$ (by quasi-periodicity $u(B_{1,0})=e^{2\imath\pi\beta} u(B_{0,1}))$ and the Kirchhoff condition at the point $A_{0,0}$ provides the value  $u_0'$ of the derivative of $u|_{e_{0,t}^T}$ at $s=0$ ( corresponding to the node $A_{0,0}$):
\begin{equation*}
u_0' = \frac{\omega}{\sin(\omega L)} ( - 2 u_0 \cos(\omega L) + v_1 e^{2\imath\pi \beta}) .
\end{equation*}
Then, by the Cauchy Lipschitz theorem, $u$ is defined uniquely on  $e_{0,t}^T$ from the data $u_0$ and $u_0'$. By definition of $\phi_{\omega,\beta}$, we shall see that $u$ and $\phi_{\omega, \beta}$ coincide, up to the multiplicative constant $\alpha$, on $e_{0,t}^T$. Consequently, $u|_{e_{0,t}^T}(t)=0$ implies that $\phi_{\omega, \beta}(t) =0$.

Now, we have to treat the case $t \in (0, L)$. Let us denote  $v_0'= u|_{e_0}'(L)$. The Kirchhoff condition at $B_{0,0}$ rewrites $u|_{e_{1,t}^T-{\bf v}_2}'(L) + u|_{e_{2,t}^T-{\bf v}_1}'(L) + v_0' =0$. There exists then $\theta\in\C$ such that $u|_{e_{1,t}^T-{\bf v}_2}'(L)=\theta v_0'$ and $u|_{e_{2,t}^T-{\bf v}_1}'(L)=-(1+\theta) v_0'$, where $e_{j,t}^T$ is defined in \eqref{eq:param_edgetruncated}. We obtain then that
\begin{equation}\label{eq:quOblicEdge}
u\big|_{e_{1,t}^T-{\bf v}_2}(s) = v_{0} \cos(\omega(L-s)) +  \theta \frac{v_0'}{\omega} \sin(\omega(L-s)), \quad s\in (t,L),
\end{equation}
and
\begin{equation}\label{eq:quOblicEdge2}
u\big|_{e_{2,t}^T-{\bf v}_1}(s) = v_{0} \cos(\omega(L-s)) -  (1 +\theta) \frac{v_0'}{\omega} \sin(\omega(L-s)), \quad s\in (t,L).
\end{equation}
The Dirichlet boundary condition at $s=t$ on $e_{1,t}^T-{\bf v}_2$ and $e_{2,t}^T-{\bf v}_1$ yields
\begin{equation}\label{PreuveCaracterisationphibetaomega1}
  v_{0} \cos(\omega(L-t)) +  \theta \frac{v_0'}{\omega} \sin(\omega(L-t))=v_{0} \cos(\omega(L-t)) -  (1 +\theta) \frac{v_0'}{\omega} \sin(\omega(L-t))=0.
\end{equation}
We end up with 
\[
  (1 +2\theta) \frac{v_0'}{\omega} \sin(\omega(L-t))=0.
  \]
  If $v_0'=0$ then $u\big|_{e_{1,t}^T-{\bf v}_2}=u\big|_{e_{2,t}^T-{\bf v}_1}$ and $u$ coincides  (up to a multiplicative constant) with $\phi_{\omega, \beta}$. Note that $v_0\neq 0$ in that case, otherwise $u=0$. If $\theta = -1/2$ then  $u\big|_{e_{1,t}^T-{\bf v}_2}=u\big|_{e_{2,t}^T-{\bf v}_1}$ and $u$ coincides (up to a multiplicative constant) with $\phi_{\omega, \beta}$. Finally, if  $\sin(\omega(L-t)) =0$, then, the equalities~\eqref{PreuveCaracterisationphibetaomega1} can be both rewritten as 
$
v_{0} \,\cos(\omega(L-t))=0.
$ 
Since $ \cos(\omega(L-t))$ does not vanish, this implies that $v_0=0$. In that case,  $u\big|_{e_{1,t}^T-{\bf v}_2}$, $u\big|_{e_{2,t}^T-{\bf v}_1}$ and $\phi_{\omega, \beta} $ are proportional to $ v_{0}' \sin(\omega(L-s))$ and their zeros coincide.
\end{proof}
% \begin{rem} \commente{J enleverai bien cette remarque...}{\color{magenta} d accord}We point out that we recover the results of Section~\ref{sec:DiscreteSepctrumeA0beta}  as we make tend $t$ toward $0$ or $L$. Indeed, taking
% $$
% \left[ \begin{array}{l} u_0 \\
%v_0 \end{array} \right] = \mathbf{e}(\omega,\beta)   
% $$  
%we can see that (cf.~\eqref{eq:quOblicEdge} for $\theta =\frac{1}{2}$), 
% $$
%  \phi'_{\omega, \beta}(0) = \omega v_0 \sin(\omega L) + \frac{v_0^p}{2} \cos(\omega L).
% $$ 
% with $v_{0}^p$ defined by~\eqref{eq:defv0p}. In the case $\omega = \omega_n$, we obtain $v_0=0$, which is possible only if $\beta \in ]\frac{2\pi}{3},\pi[$. Be careful however, in that case, in order to guarantee the $\beta$ quasi-periodicity condition, we  must take $\theta = \frac{e^{\imath \beta}}{1+e^{\imath \beta} }$ in the definition~\eqref{eq:quOblicEdge}-\eqref{eq:quOblicEdge2} of $u$ on the oblique edges $[A_{0,-1} B_{0,0} ]$ and $[A_{-1,0} B_{0,0} ]$. 
%
%\noindent Similarly, we can also verify that
%$$
% \phi'_{\omega, \beta}(L) = \frac{ \omega}{\sin(\omega L) } \left( u_0 \cos(\omega L) - v_0  \right) ,
%$$ 
%which also leads to $v_0=0$, only possible if $\beta \in ]\frac{2\pi}{3},\pi[$. 
% \end{rem} 
  \subsubsection{ Investigation of the zeros of the characteristic function and of its derivative}
  
  We finally prove Proposition \ref{prop:guided_case3}. From Proposition \ref{prop:lienavecleszeros}, it suffices to show the following.
   \begin{proposition}\label{PropFlotSpectral}For any $\beta \in (0,1/2)\setminus\{1/3\}$,  for any $\omega \in \tilde{I}_n(\beta)$, there exist exactly $2n+1$ values $t^D_{n,q} \in [0,2L)$  such that $  (\phi_{\omega,\beta})(t^D_{n,q}) = 0$ and $2n+1$ values $t^N_{n,q} \in[0,2L)$ such that $  (\phi_{\omega,\beta})'(t^N_{n,q}) = 0$.    
\end{proposition} 
 The remainder of this section is dedicated to the proof of Proposition \ref{PropFlotSpectral} following the three steps: 
\begin{enumerate}
 \item We first prove that the number of values $t$ such that $\phi_{\omega,\beta}(t)=0$ (resp. the number of values $t$ such that $\phi_{\omega,\beta}'(t)=0$)  is constant when $\omega\in \tilde{I}_n^-(\beta)$ and $\omega\in \tilde{I}_n^+(\beta)$, see Lemma~\ref{LemmeConstant};
 \item we investigate the particular case $\omega = \omega_n^\ast$ where explicit computations can be done, see Lemma~\ref{LemmeOmegaStar}. The number of zeros (resp. the number of the zeros of the derivative) is equal to $2n+1$.
 \item we prove the strict monotonicity of the curves $\omega \mapsto t^I_{n,q}(\omega)$, see Proposition~\ref{PropFlotSpectral_monotone}. This implies that the number of values $t$ for which $\omega^2$ is an eigenvalue of $\mathcal{D}_1^t(\beta)$ (resp. of $\mathcal{N}_1^t(\beta)$) is the same whenever $\omega\in  \tilde{I}_n^-(\beta)$ or $\omega\in  \tilde{I}_n^+(\beta)$.
\end{enumerate}
%A formal differentiable argument allows finally to prove Proposition \ref{PropFlotSpectral_monotone} (see Lemma~\ref{LemmeStrictCroissance}).
\begin{lemma}\label{LemmeConstant} Let $\beta \in (0,1/2)\setminus\{1/3\}.$ The number $N_n^D(\omega,\beta)$ (resp. $N_n^N(\omega,\beta)$) of values $t$ such that $\phi_{\omega,\beta}(t)=0$ (resp. such that $\phi_{\omega,\beta}'(t)=0$)  is constant when $\omega\in \tilde{I}_n^+(\beta)$ and when $\omega\in \tilde{I}_n^-(\beta)$.
   \end{lemma}
\begin{proof}
Let $n\in \mathbb{N}$ and suppose that $\omega \in \tilde{I}_n^+(\beta)$, the proof being similar for $\omega \in \tilde{I}_n^-(\beta)$. 

(1) Let us first show the result for $N_n^D(\omega,\beta)$. If $t\in(0,L)\cup(L,2L)$ is a zero of $\phi_{\omega, \beta}$, using the expression \eqref{eq:values_phi} of $\phi_{\omega,\beta}$ and the implicite function theorem, we know that there exists a neighborhood $V$ of $\omega$ such that $t=t(\omega)$ and $\phi_{\omega',\beta}(t(\omega'))=0$ for $\omega'\in V$. Note that we cannot use the implicit function theorem if $t=0,L$ or $2L$, the function $\phi_{\omega, \beta}$ being not differentiable at these points. This means that the number of zeros of $\phi_{\omega,\beta}$ could change for $\omega \in \tilde{I}_n^+(\beta)$ if there exists $\omega\in \tilde{I}_n^+(\beta)$ such that $\phi_{\omega, \beta}$ vanishes at $s = 0$, $s =L$ or $s=2L$.  We study now those three cases and show that they are not possible. 

Suppose $ \phi_{\omega, \beta}(0)=0$, then using \eqref{eq:values_phi} and \eqref{eq:definitionExplicitephiprime1}, we end up with 
$
3 v_0(\omega,\beta) \cos(\omega L)  - u_0(\omega,\beta)=0.
$ 
Since $\omega\neq \omega_n^\ast$, we have (see~\eqref{eq:phib_n} and \eqref{eq:definitioner1}) $v_0(\omega,\beta) = 3 \cos(\omega L)$  and $u_0(\omega,\beta) = 1 + r(\omega, \beta) \sqrt{2 +2 \cos 2\pi\beta}$. We obtain
$
 r(\omega, \beta) =(9 \cos^2(\omega L)-1)/{\sqrt{2 +2 \cos 2\pi\beta}},
$ 
% $$
% 3 \cos(\omega L) ( 2 - 3 \cos^2(\omega L)) + (1 + r(\omega, \beta) \sqrt{2 +2 \cos \beta}) \cos(\omega L)=0.
% $$ 
Introducing this value in \eqref{eq:definitionEquationRecurrente} leads to $\cos^2(\omega L)(1+\cos(2\pi\beta))=0$ which is in contradiction with $\omega\neq \omega_n^\ast$ and $\beta\neq 1/2$.

If $ \phi_{\omega, \beta}(L)=0$, then $v_0(\omega,\beta)=0$ . However if $\omega\neq \omega_n^\ast$, we have by ~\eqref{eq:phib_n} and ~\eqref{eq:definitioner1}) that $v_0(\omega,\beta) = 3 \cos(\omega L)$  which cannot vanish if $\omega\neq \omega_n^\ast$.

Finally, if $ \phi_{\omega, \beta}(2L)=0$, then $u_0(\omega,\beta)=0$. However if $\omega\neq \omega_n^\ast$, we have by ~\eqref{eq:phib_n} and \eqref{eq:definitioner1}) that $u_0(\omega,\beta) = 1 + r(\omega, \beta) \sqrt{2 +2 \cos 2\pi\beta}$  which by \eqref{eq:r_max}  could happen only if $\omega=\omega_n^\ast$. 

(2) Let us now show the result for $N_n^N(\omega,\beta)$. Similarly, we have to show that $\phi_{\omega,\beta}'$ cannot vanish at $s=0$, $L$ or $2L$.

Suppose first that $ \phi_{\omega, \beta}'(0)=0$, then using \eqref{eq:values_phi} and \eqref{eq:definitionExplicitephiprime1}, we end up with 
 $
  v_0(\omega,\beta) ( 2 - 3 \cos^2(\omega L)) +  u_0(\omega,\beta)  \cos(\omega L)=0.
 $
 Since $\omega\neq \omega_n^\ast$, we obtain by \eqref{eq:phib_n} and \eqref{eq:definitioner1} 
 $
r(\omega, \beta)  = (9 \cos^2(\omega L)-7)/{\sqrt{2 +2 \cos 2\pi\beta}}.
 $ 
 Introducing this value in \eqref{eq:definitionEquationRecurrente} leads to $15+18\cos^2(\omega L)+\cos(2\pi\beta)(9\cos^2(\omega L)+1)=0$, which is impossible since the left hand side is strictly positive.

 If $ \phi_{\omega, \beta}'(L)=0$ , we have  $v_0'(\omega,\beta) =0$, which implies by  \eqref{eq:phib_n} and \eqref{eq:definitioner1} 
  $
  r(\omega, \beta) =(3  \cos^2(\omega L)-1)/{\sqrt{2+2\cos2\pi\beta}}.
  $ 
  Since, for $\omega \in \tilde{I}^n(\beta)$, $r(\omega, \beta) <0$ (see the discussion of the variation $r(\omega, \beta)$, right after its definition~\eqref{definitionr}),  the previous equality can only hold if  $ \cos^2(\omega L) < {1}/{3}$. This is in contradiction with the result obtained when introducing the value of $r(\omega, \beta)$  in \eqref{eq:definitionEquationRecurrente} (since $\omega\neq \omega_n^\ast$)
  $
   \cos^2(\omega L) = (2+ \cos2\pi\beta)/{3}  \geq {1}/{3}.
  $ 

  Finally, if $ \phi_{\omega, \beta}'(2L)=0$, we obtain using \eqref{eq:definitionExplicitephiprime2}, \eqref{eq:phib_n} and \eqref{eq:definitioner1} 
%  $
%   \cos(\omega L) (- 2 + r(\omega, \beta)  \sqrt{2 +2 \cos 2\pi\beta})=0;
%  $
 and since $\omega \neq \omega_n^\ast$
that
 $
 r(\omega, \beta) = {2}/{ \sqrt{2 +2 \cos 2\pi\beta}},
 $ 
 which is impossible since $r(\omega, \beta)<0$.
\end{proof}
\begin{lemma}\label{LemmeOmegaStar}  Let $n\in\N$, $\omega = \omega_n^\ast $ and $\beta\in (0,1/2)\setminus\{1/3\}$. Then 
    \begin{itemize}
    \item $(\phi_{\omega,\beta})$ has exactly $2n +1$ zeros given by 
    \begin{equation}\label{eq:Deftnq1_D}
 \dsp 0 \leq q \leq 2n,\quad  t_{n, q}^D  = \begin{cases}
  \dsp \frac{2q L}{2n+1}  & \text{if}\; \beta<1/3,\\
  \dsp \frac{(2q+1) L}{2n+1}  & \text{if}\; \beta>1/3.
 \end{cases}
  \end{equation} 
    \item $(\phi_{\omega,\beta})'$ has exactly $2n +1$ zeros given by 
      \begin{equation}\label{eq:Deftnq1_N}
   \dsp 0 \leq q \leq 2n,\quad  t_{n, q}^N  = \begin{cases}
    \dsp \frac{(2q+1) L}{2n+1}  & \text{if}\; \beta<1/3,\\
    \dsp \frac{2q L}{2n+1}  & \text{if}\; \beta>1/3.
   \end{cases}
    \end{equation} 
    \end{itemize} 
\end{lemma} 
Note that we recover the result of Proposition \ref{prop:guided_case12} for $t=0$.
%{\color{magenta}
%\begin{rem} The previous lemma shows the existence of new sets of eigenvalues that are independent of $\beta$. Those eigenvalues are materialized by blue points on Figures~\ref{DessinBetaPiSurTrois} and~\ref{DessinBetaCinqPiSurSix}. % In both cases, we see the existence of $2n +1$ spectral flows  through the gap $I^n(\beta)$. 
%The blue points materialize values of $t_k^n$ that are independent of $\beta$, leading  to flat eigenvalues when $\beta$ is varying (as demonstrated in Section~\ref{sec:DiscreteSepctrumeA0beta} for $t=0$). For instance, in the case $n=0$, this invariant corresponds to $t=L$ (bearded configuration) for  $\beta \in [0, \frac{1}{3}) $ and to $t=0$ for $\beta \in (\frac{1}{3},\frac{1}{2}[$ (classical zig-zag configuration). We shall see that there are exactly $2n+1$ invariants in the gap $I^n$. 
%\end{rem}
%}
\begin{proof}
Let $\beta <1/3$ then by \eqref{eq:values_phi} and Lemma \ref{lemma:extensionContinuity}, we obtain $u_0(\omega_n^\ast ,\beta) = 0$, $v_0(\omega_n^\ast,\beta ) = 1$ and $v_0'(\omega_n^\ast,\beta)=0$. Expressions 
\eqref{eq:definitionExplicitephiprime1} and \eqref{eq:definitionExplicitephiprime2} simplify as
\[
  \phi_{\omega_n^\ast , \beta}(s)= (-1)^n\sin(\omega_n^\ast s)\quad\text{if}\;s\in[0,2L].
  \]
For $\beta >1/3$, by \eqref{eq:values_phi} and Lemma \ref{lemma:extensionContinuity}, we have $u_0(\omega_n^\ast,\beta) = 1$, $v_0(\omega_n^\ast,\beta) = 0$ and $v_0'(\omega_n^\ast,\beta)=(-1)^n\omega_n^\ast$. Then Expressions 
\eqref{eq:definitionExplicitephiprime1} and \eqref{eq:definitionExplicitephiprime2} simplify as
$$\phi_{\omega,\beta}=\begin{cases} -1/2\cos(\omega_n^\ast s)&\text{if}\;s\in[0,L],\\
  -\cos(\omega_n^\ast s)&\text{if}\;s\in[L,2L].
\end{cases}$$
We deduce that
\[
  \phi_{\omega_n^\ast , \beta}(t)=0\;\Leftrightarrow \;t=t_{n,q}^D, \;0 \leq q \leq 2n,\quad \text{and}\quad \phi_{\omega_n^\ast , \beta}'(t)=0\;\Leftrightarrow\; t=t_{n,q}^N,\;0 \leq q \leq 2n.
  \]
\end{proof}
We end the proof of Proposition~\ref{prop:guided_case3} by proving the monotony of the curves $\omega \mapsto t^I_{n,q}(\omega)$ for $I\in\{N,D\}$, or equivalently the monotonicity of the eigenvalues around each of the values $t^I_{n,q}$.
\begin{proposition}\label{PropFlotSpectral_monotone}For any $\beta \in (0,1/2)\setminus\{1/3\})$,  for any $\omega \in \tilde{I}_n(\beta)$, for $I\in\{N,D\}$, the curves $\omega \mapsto t^I_{n,q}(\omega)$ are monotonic (decreasing for the Dirichlet case and increasing for the Neumann case).   
\end{proposition}
\begin{proof} Let us first prove that the function $\omega \mapsto t^I_{n,q}(\omega)$ for $I\in\{D,N\}$ are continuously differentiable. We prove it for $I=D$, the same arguments could be used for $I=N$ replacing $\phi_{\omega,\beta}$ by $\phi'_{\omega,\beta}$ in the following arguments. We remind that  $t^D_{n,q}(\omega)$ satisfies 
  $
  \phi_{\omega, \beta}(t)=0$. For any $\beta \in (0,1/2)\setminus\{1/3\}$, let us now apply the implicit function theorem to the function 
  $
  f : ( \omega,t)\mapsto \phi_{\omega, \beta}(t).$ 
  We remark that, for any couple $(t,\omega) \in [0, L]\times \tilde{I}_n^\pm$  (or $[L, 2L] \times \tilde{I}_n^\pm)$   such that $f(\omega, t)=0$, the Cauchy-Lipschitz theorem guarantees that $\partial_t f (\omega,t)\neq 0$. Then, the implicit function theorem ensures that the curves
  $\omega \mapsto t^D_{n,q}(\omega)$ are continuously differentiable on $I_n^\pm$.  
  
  Since the eigenvalue $t^I_{n,q}(\omega)$ is simple, it is easy to define an eigenvector $u^I_{n,q}(\omega)\in H^1(\Delta,\widehat{\cal G}_t)$ with $t=t^I_{n,q}(\omega)$ continuous and differentiable with respect to $\omega$.
  
  We can now prove the monotonicity. The proof is an adaptation of in~\cite[Proposition 3.20]{Gontier:2020}.  We shall make the computation for $t \in (L, 2L)$ but the same holds for $t \in (0, L)$. Let $\widehat{\mathcal{E}}_{2L}$ be the edges of $\mathcal{E}_{2L}$ included in $\widehat{\cal G}_{t^*}$ for $t^*=t^I_{n,q}(\omega)$. Note that this set is independent of $t^*$ as soon as $t^* \in (L, 2L)$. In the sequel, we could use $t^*$ instead of $t^I_{n,q}(\omega)$ to simplify the notation.
  Let $v \in H^1(\widehat{\mathcal{G}}_{t^*})$ and if $I=D$, we suppose that $v=0$ at $x=t^*$. One has
  \begin{multline*}
  \sum_{e \in \widehat{\mathcal{E}}_{2L} } \int_e \left(\partial_su^I_{n,q}(\omega)  \; \partial_s \overline{v}-\omega^2 u^I_{n,q}(\omega)  \; \overline{v}\right) + \int_{e_{0,t^*}^T} \left(\partial_s u^I_{n,q}(\omega) \; \partial_s \overline{v}  - \omega^2 u^I_{n,q}(\omega) \; \overline{v} \right)=0.
  \end{multline*}
  Differentiating this equality with respect to $\omega$ gives
  \begin{multline*}
    \sum_{e \in \widehat{\mathcal{E}}_{2L} } \int_e \left(\partial_s ( \partial_\omega u^I_{n,q}(\omega) ) \; \partial_s \overline{v}-  \omega^2\partial_\omega u^I_{n,q}(\omega) \; \overline{v} \right) + \int_{e_{0,t^*}^T} \left(\partial_s (\partial_\omega u^I_{n,q}(\omega)) \; \partial_s \overline{v} -  \omega^2  \partial_\omega u^I_{n,q}(\omega)\; \overline{v}\right)  \\=
    2\omega \Big( \sum_{e \in \widehat{\mathcal{E}}_{2L}} \int_e  u^I_{n,q}(\omega) \; \overline{v} +  \int_{e_{0,t^*}^T} u^I_{n,q}(\omega)\;  \overline{v} \Big) 
    + \partial_\omega t^I_{n,q}(\omega)\,\Big(\partial_s  u^I_{n,q}(\omega) (t^*) \, \partial_s \overline{v} (t^*)  - \omega^2\,u^I_{n,q}(\omega) (t^*) \overline{v}(t^*)\Big),
  \end{multline*} 
  where we have used the parametrization \eqref{eq:param_edgetruncated} of the edge $e_{0,t}^T$ for $t=t^I_{n,q}(\omega)$. The right hand side of the previous equation should vanish when $v=u^I_{n,q}(\omega)$ which yields
  $$
    2\omega\; \|  u^I_{n,q}(\omega)  \|_{L^2(\widehat{\mathcal{G}}_{t^*})} =  \partial_\omega t^I_{n,q}(\omega)\; \left( \omega^2|  u^I_{n,q}(\omega) (t^*)|^2- |\partial_s  u^I_{n,q}(\omega) (t^*)|^2 \right),\quad t^*=t^I_{n,q}(\omega)
  $$ 
  Since $\omega>0$, if $I=D$ then $u^I_{n,q}(\omega) (t^*) = 0$ so that $ \partial_\omega t^I_{n,q}(\omega)<0$ and if $ I=N$ then $\partial_s  u^I_{n,q}(\omega) (t^*) = 0$ so that $\partial_\omega t^I_{n,q}(\omega)>0$. \end{proof}

\subsection{Numerical illustrations}\label{sec:numericalResultSpectreDiscret}

As for the essential spectrum, we compute an approximation of the discrete spectrum (eigenvalues) using a standard $P_1$ finite element method. However, the computation is less easy  because one has to solve an eigenvalue problem set on an unbounded domain. To address this difficulty, we have used a method based on the construction of Dirichlet-to-Neumann (DtN) operators in periodic waveguides (see \cite{ Fliss:2006,Fliss:2009}): { this requires the solution of cell problems (discretized here again using  the standard $P_1$ finite element methods) and the solution of a stationary Ricatti equation.} The construction of these DtN operators  enables us to reduce the numerical computation to a small neighborhood of the perturbation independently from the confinement of the mode (which is linked to the distance between the eigenvalue and the essential spectrum of the operator). However the reduction of the problem leads to a non linear eigenvalue problem (since the DtN operators depend on the eigenvalue) of a fixed point nature. { It is solved using a Newton-type procedure, each iteration needing a finite element computation, }see \cite{SoniaNonLineraire} for more details.
\subsubsection{The zigzag case and $t=0$}
\noindent We first present results associated with the zigzag case ($t=0$). First, for $\beta = {5 }/{12}$ (in $({1}/{3},{1}/{2}]$)  and $\mu=1$,  for different values of $\delta$, we check the existence of a simple eigenvalue $ \lambda_\delta^\beta$ of $N_{\delta,1}^0(\beta)$ in the vicinity of  $\lambda_0^\ast$ ( $= ({\pi}/{2L})^2$). Numerical results are reported in Figure~\ref{fig:LambdaOfDeltaDiscret}, illustrating that our graph model is a first order asymptotic model. 
\begin{figure}[htbp]
	\begin{center}
 \includegraphics[width=0.8\textwidth]
 {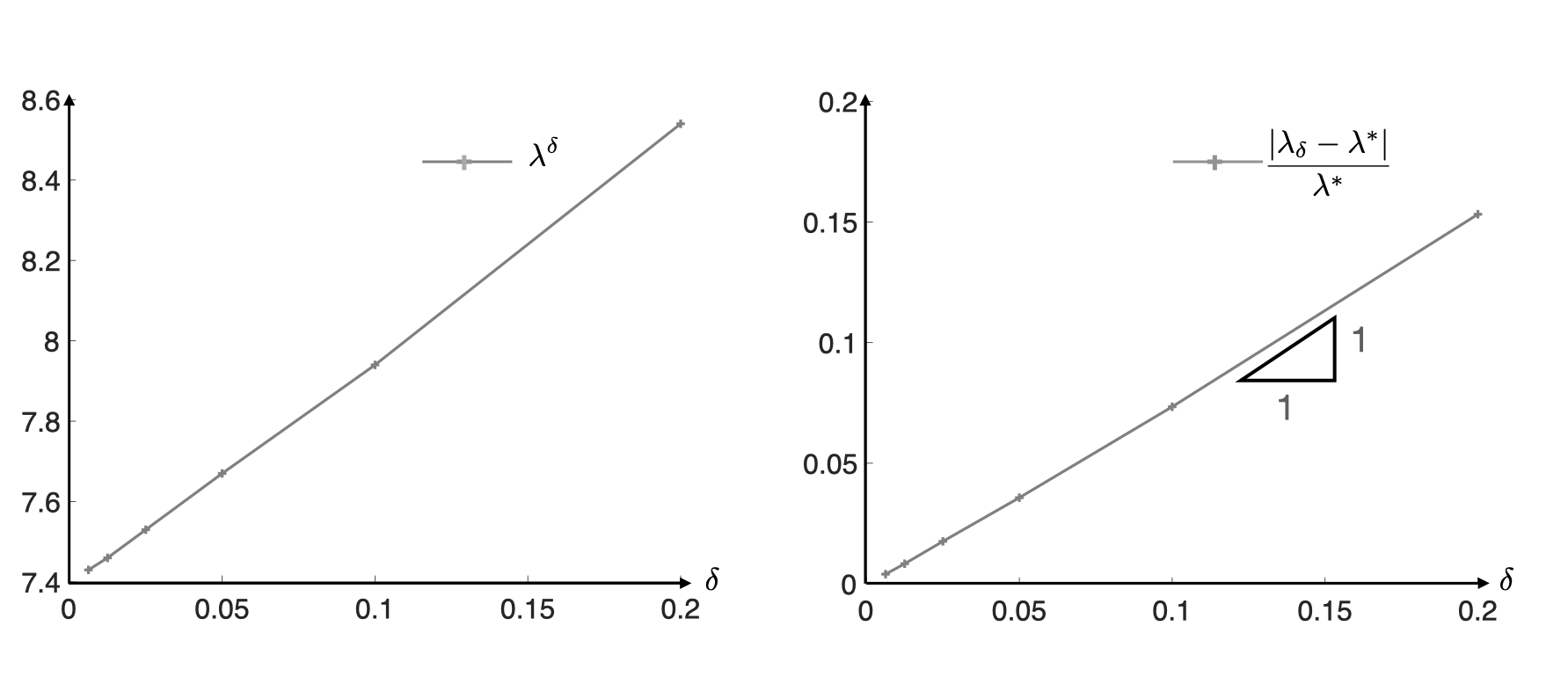}
  \caption{For $\beta = {5}/{12}$: representation of $\delta\mapsto\lambda_\delta$  (left) and the relative error $\delta\mapsto|\lambda_\delta-\lambda^\ast_0|/|\lambda^\ast_0|$ (right).}
	 \label{fig:LambdaOfDeltaDiscret}
 	\end{center}\vspace{-0.5cm}
\end{figure}
 Then, Figure~\ref{fig:VpPlateZigZag}-(left) shows the evolution of the spectrum of $N_{\delta,1}^0(\beta)$ with respect to $\beta$ for $\beta \in  ({1}/{3},{1}/{2} )$ .  The striped blue part represents the essential spectrum (which depends on $\beta$), while the eigenvalue $\lambda_\delta^\beta$ (discrete spectrum) is represented in magenta. As expected, the curve $\beta \mapsto \lambda_\delta^\beta$ is almost flat. Note that with our algorithm, we did not find any eigenvalue for $\beta > {1}/{3}$.
\begin{figure}[htbp]
	\begin{center}\vspace{-0.5cm}
    \includegraphics[width=0.9\textwidth, trim=0cm 0cm 0cm 0cm, clip] {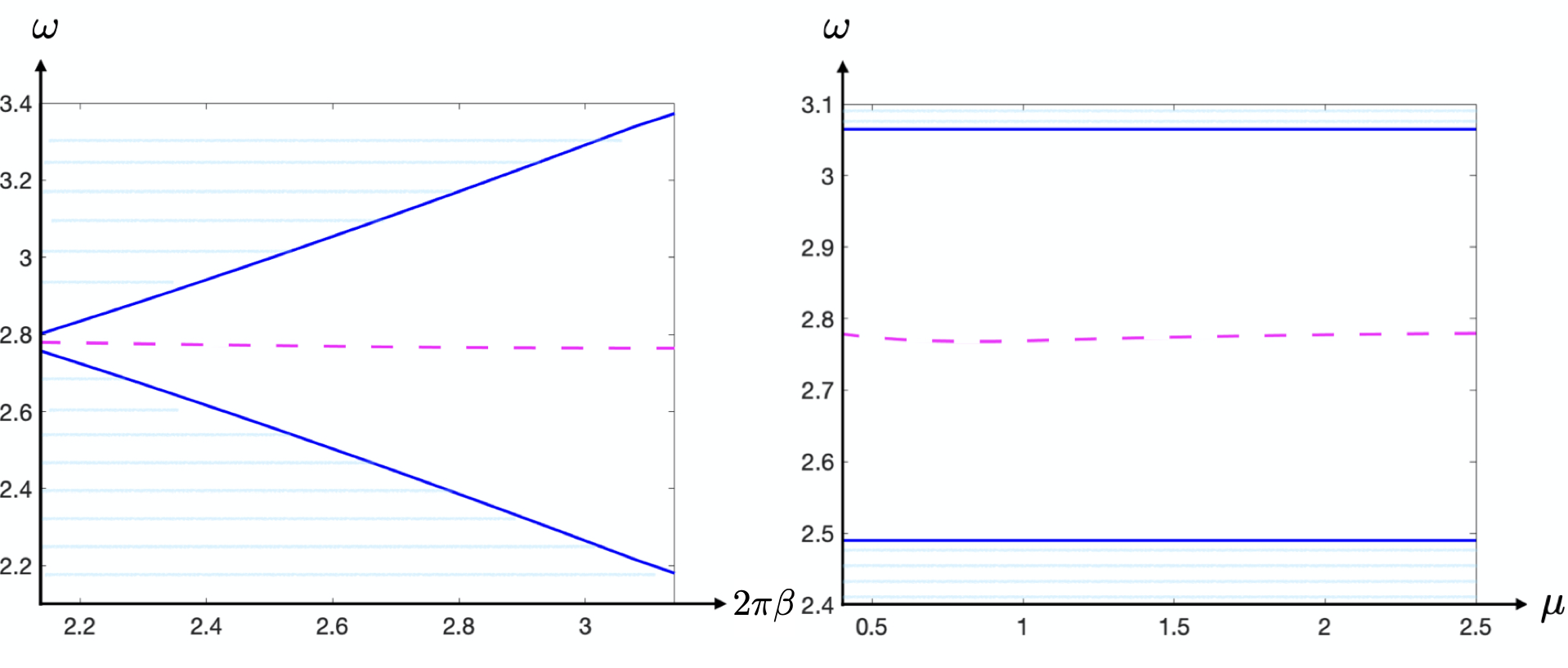}
  \caption{Spectrum of $N_{\delta,1}^0(\beta)$  for $\delta =0.05$:  $\mu=1$, $\beta \in (\frac{1}{3},\frac{1}{2})$ (left) ;  $\beta=\frac{5}{12}$, $\mu \in [0.5, 2.5]$(right).}
	 \label{fig:VpPlateZigZag}
 	\end{center}\vspace{-0.5cm}
\end{figure}

 In a third step,  we modify the width of the perturbed edges from $\delta$ to $\mu \delta$, making varying $\mu$ from $0.5$ to $2.5$. For $\beta = {5}/{12}$ and $\delta =0.05$,  the spectrum of $N_{\delta,\mu}^0(\beta)$  is represented on Figure~\ref{fig:VpPlateZigZag}-(right). Here again,  the striped blue part represents the essential spectrum (and  here is independent of $\mu$). The discrete spectrum, represented with the dotted magenta line, varies very slowly with respect to $\mu$. 

 In the case $\delta=0.1$, the absolute value of a corresponding eigenvector are represented on Figure~\ref{fig:VecteurPropreZigZag}  for $\mu=1$ (left panel) and for  $\mu=2$ (middle panel). We remark that choosing two different values of $\mu$ in the perturbed zig-zag edges also creates a guided mode (see Remark \ref{rem:different_mu}): an example of eigenvector (absolute value) is represented on Figure~\ref{fig:VecteurPropreZigZag}-(right) wherein we choose $\mu = 2$ in the upper perturbed oblique edge and $\mu=1/2$ in the lower perturbed one.
\begin{figure}[htbp]
	\begin{center}
 \includegraphics[width=1\textwidth, trim=0cm 0cm 0cm 0cm, clip]
 {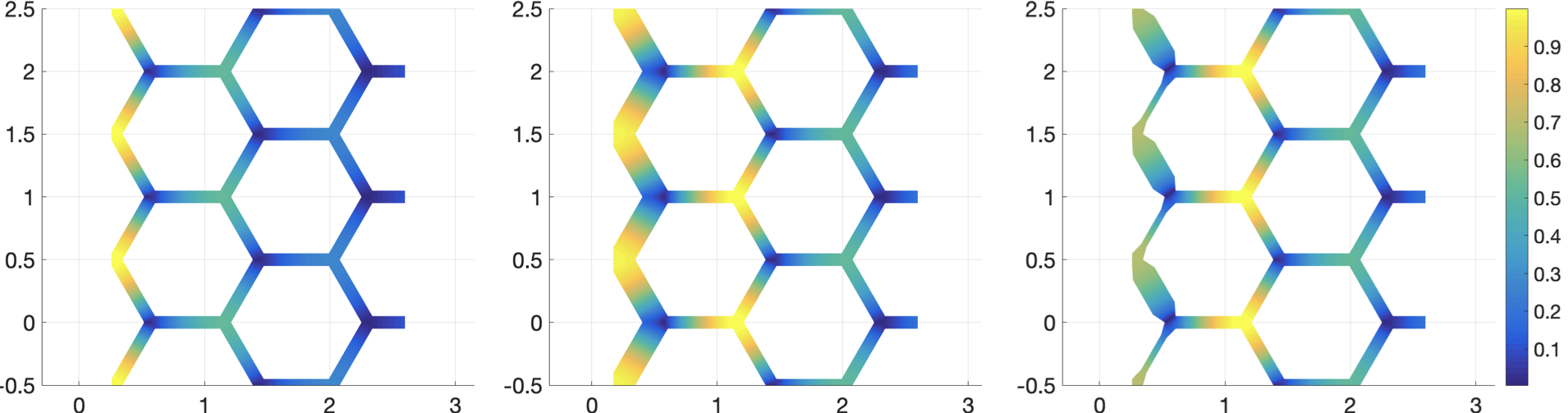}

  \caption{Examples of eigenvectors (absolute value) of $N_{\delta,\mu}^0(\beta)$ for $\beta ={5 }/{12}$ and $\delta =0.1$: $\mu=1$  (left),  $\mu=2$ (middle),  $\mu$ variable (right)}
	 \label{fig:VecteurPropreZigZag}
 	\end{center}\vspace{-0.5cm}
\end{figure}

 In the three pictures of Figure~\ref{fig:VecteurPropreZigZag},  we notice that the absolute value of the eigenvector is  small in the  junctions of type $B$ (namely junctions that shrinks to a $B_{n, m}$ point as $\delta$ goes to $0$). This was predicted by the graph model since the limit eigenvector vanishes on the $B_{n, m}$ (see Remark~\ref{rem:chirality}).   For the operator  $D_{\delta,\mu}^0(\beta)$, similar pictures are represented on Figure~\ref{fig:VecteurPropreZigZagDirichlet} in the case $\beta = \frac{1 }{6}$.  By contrast here, the absolute value is very small in the vicinity of the  junctions on type $A$.\\

\begin{figure}[htbp]
	\begin{center}
  \includegraphics[width=1\textwidth, trim=0cm 0cm 0cm 0cm, clip]
 {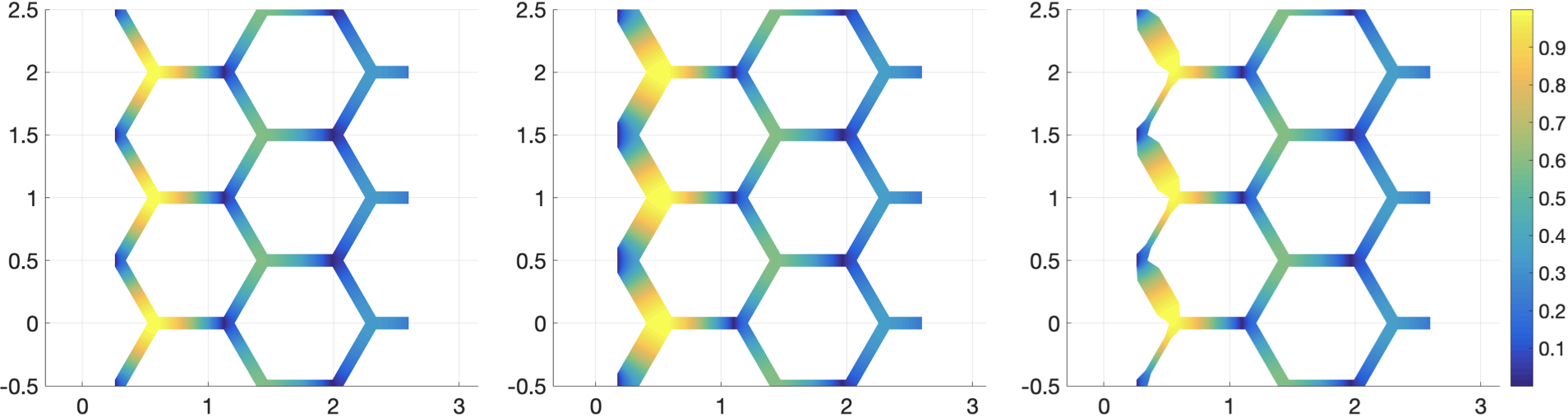}

  \caption{Examples of eigenvectors (absolute value) of $D_{\delta,\mu}^0(\beta)$ for $\beta ={1 }/{6}$ and $\delta =0.1$: $\mu=1$ (left),  $\mu=2$ (middle), $\mu$ variable (right) .}
	 \label{fig:VecteurPropreZigZagDirichlet}
 	\end{center}\vspace{-0.5cm}
\end{figure}

 Finally, we check that those eigenfrequencies indeed correspond to symmetric modes  for $\beta\in ({1}/{3}, {1}/{2})$ (respectively antisymmetric modes for $\beta\in [0, {1}/{3})$)   associated with the Laplacian-Neumann operator posed on an infinite domain obtained by making a mirror symmetry of the domain $\Omega_0^{\delta,\mu}$ (the domain is then infinite in the two directions $x\rightarrow \pm\infty$).   For $\delta =0.1$, the corresponding eigenvectors (real part) are represented in  Figure~\ref{fig:VecteurPropreZigZagGuideComplet} ($\beta ={5 }/{12}$) and Figure~\ref{fig:VecteurPropreZigZagGuideCompletAS} ($\beta ={1 }/{6}$) for two different values of $\mu$.  In all the cases, the associated eigenfrequency is between $7.3$ and $8$. Naturally, the modes are even for $\beta ={5 }/{12}$ while they are odd for  $\beta ={1 }/{6}$.

\begin{figure}[htbp]
	\begin{center}
 \includegraphics[width=1\textwidth, trim=0cm 0cm 0cm 0cm, clip]
 {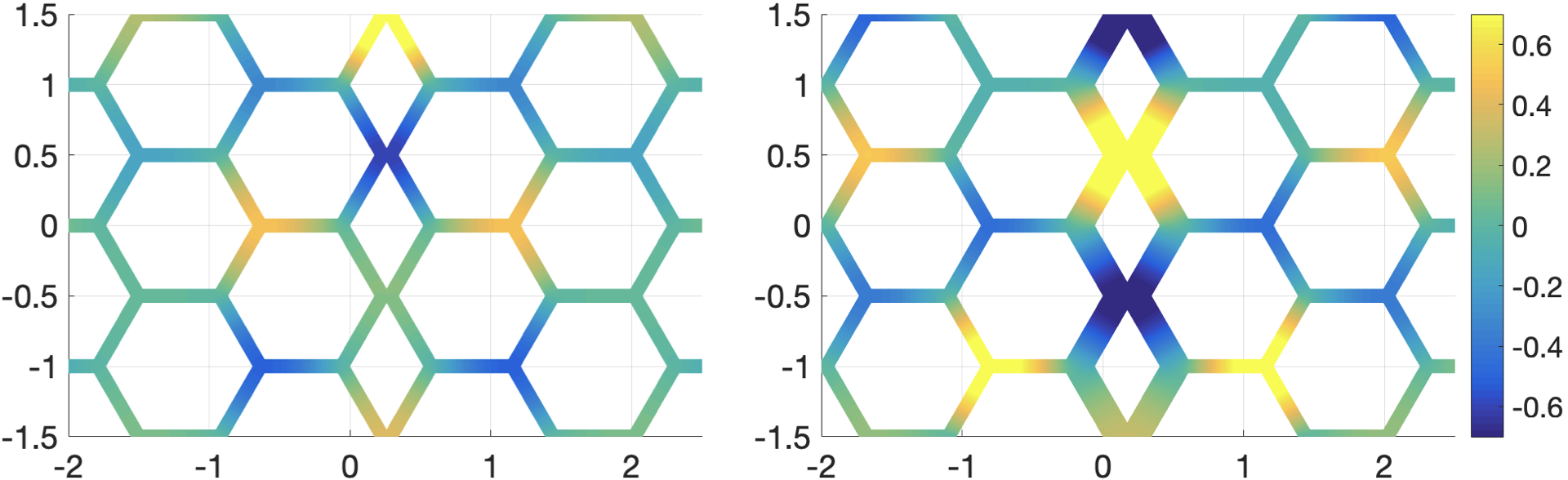}
  \caption{Real part of one eigenvector for $\beta ={5}/{12}$ and $\delta =0.1$:  $\mu=1$ (left);  $\mu=2$ (right).}
	 \label{fig:VecteurPropreZigZagGuideComplet}
 	\end{center}\vspace{-0.5cm}
\end{figure}

\begin{figure}[htbp]
	\begin{center}
 \includegraphics[width=1\textwidth, trim=0cm 0cm 0cm 0cm, clip] {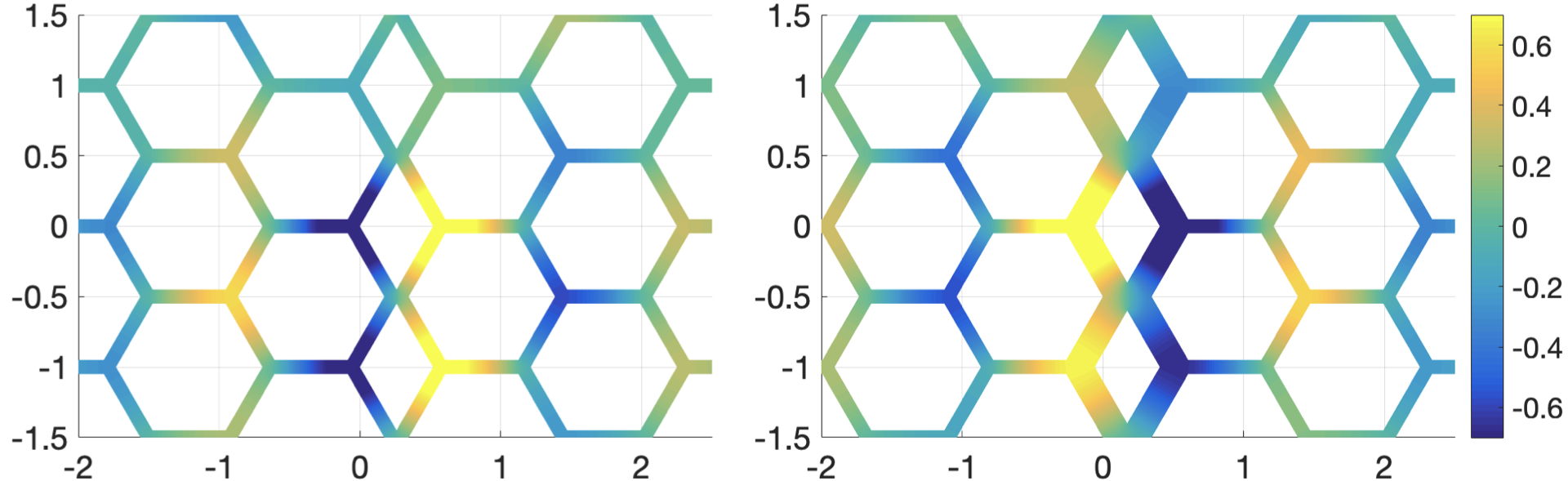}
  \caption{Real part of one eigenvector for $\beta ={1}/{6}$ and $\delta =0.1$:  $\mu=1$ (left);  $\mu=2$ (right).}
	 \label{fig:VecteurPropreZigZagGuideCompletAS}
 	\end{center}\vspace{-0.5cm}
\end{figure}

\subsubsection{The case $t\neq 0$ }

For $t\neq 0$, we focus on the operator  $N_{\delta,1}^t(\beta)$ (with $\mu=1$).  First, we reproduce the experiments of Figure~\ref{DessinBetaPiSurTrois} for a 'thick' graph domain of thickness $\delta = 0.05$.  
Figure~\ref{fig:FlotSpectraux} presents the evolution of the discrete spectrum of $N_{\delta,1}^t(\beta)$ with respect to $t$ in the  gaps $I^0_\delta(\beta)$ and $I^1_\delta(\beta)$ for $\beta = \frac{1}{6}$ (left) and $\beta = {5}/{12}$ (right).   As predicted, we observe the existence of a spectral flow, i.e. functions $t\mapsto\sqrt{\lambda(t)}$ where $\lambda(t)$ is an eigenvalue of $N_{\delta,1}^t(\beta)$, through the gap $I^0_\delta(\beta)$  and three spectral flows for $I^1_\delta(\beta)$. However, in the case   $\beta ={5}/{12}$,  it seems that the curves $t \mapsto \lambda_\delta(t)$ are not always monotone: some brutal change appears when the cut is made in the vicinity of the junction (i.e. $t \approx L$)).

 By contrast to the limit case (Figure~\ref{DessinBetaPiSurTrois}), there is no more eigenvalues strictly independent of $\beta$ (blue points in those figures). However, Figure~\ref{fig:FlatEigenvalues} confirms that the eigenvalues represented with blue stars in Figure~\ref{fig:FlotSpectraux} ($t=L$, $t=L/3$ and $t=5L/3$ for $\beta ={1}/{6}$ and $t=0$, $t=2L/3$ and $t=4L/3$ for $\beta ={5}/{12}$) also vary very slowly with $\beta$ (see the dotted magenta line). The absolute value of two associated eigenvectors  is represented on~Figure~\ref{fig:VPHighFrequency}.

\begin{figure}[htbp]
	\begin{center}\vspace{-0.5cm}
 \includegraphics[width=0.7\textwidth]{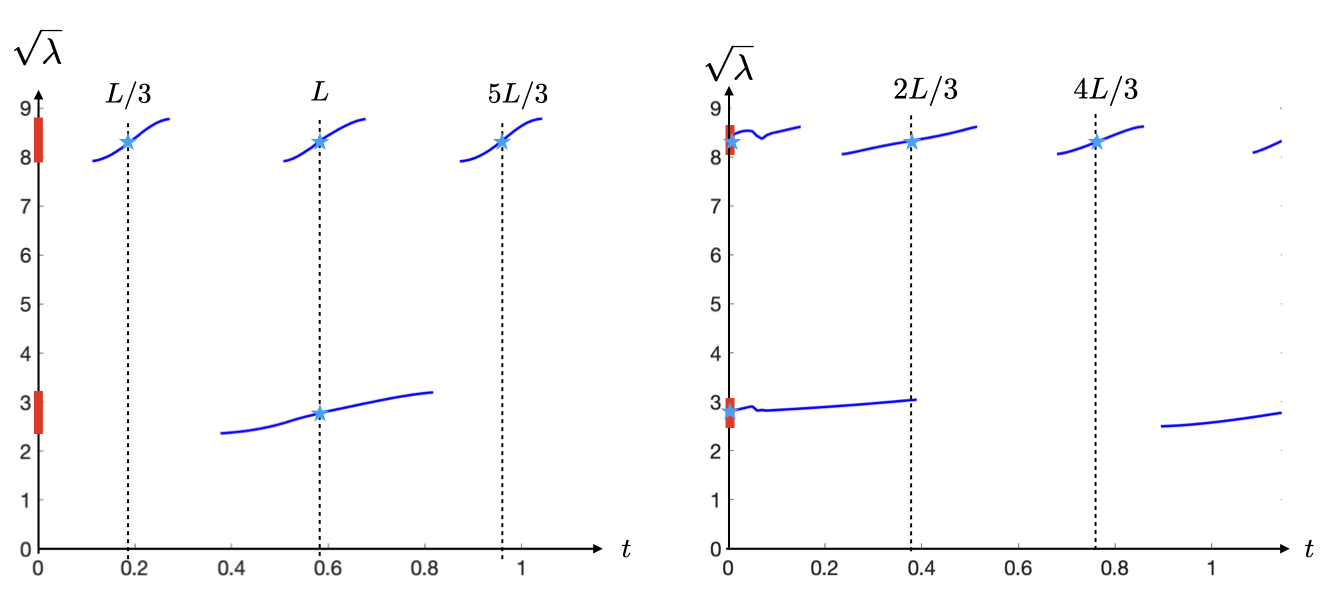}
  \caption{Evolution of the discrete spectrum of  $N_{\delta,1}^t(\beta)$ with respect to $t$ in the first two gaps for  $\delta =0.05$, $\beta =\frac{1}{6}$ (left) and $\beta =\frac{5}{12}$ (right).}
	 \label{fig:FlotSpectraux}
 	\end{center}\vspace{-1cm}
\end{figure}

\begin{figure}[htbp]
	\begin{center}
    \includegraphics[width=0.8\textwidth]{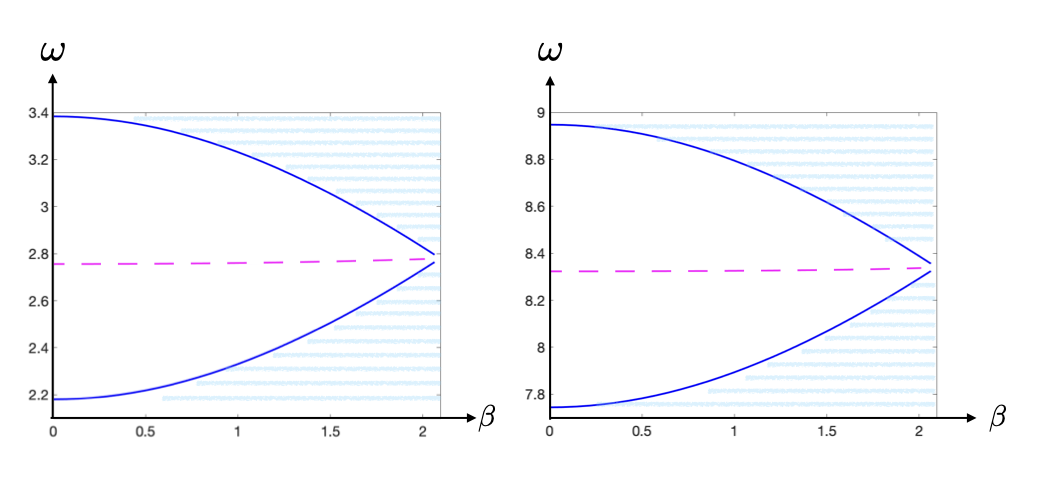}\vspace{-0.5cm}
  \caption{Evolution of the discrete spectrum of  $N_{\delta,1}^t(\beta)$   with respect to $\beta$ ($\beta \in [0, {1}/{3})$) for $t=L$ (left, gap $I^0_\delta(\beta)$) and $t =L/3$ (right, gap $I^1_\delta(\beta)$) }
	 \label{fig:FlatEigenvalues}
 	\end{center}\vspace{-1.3cm}
\end{figure}

\begin{figure}[htbp]
	\begin{center}
	 \includegraphics[width=0.4\textwidth, trim=0cm 0cm 0cm 0cm, clip]
 {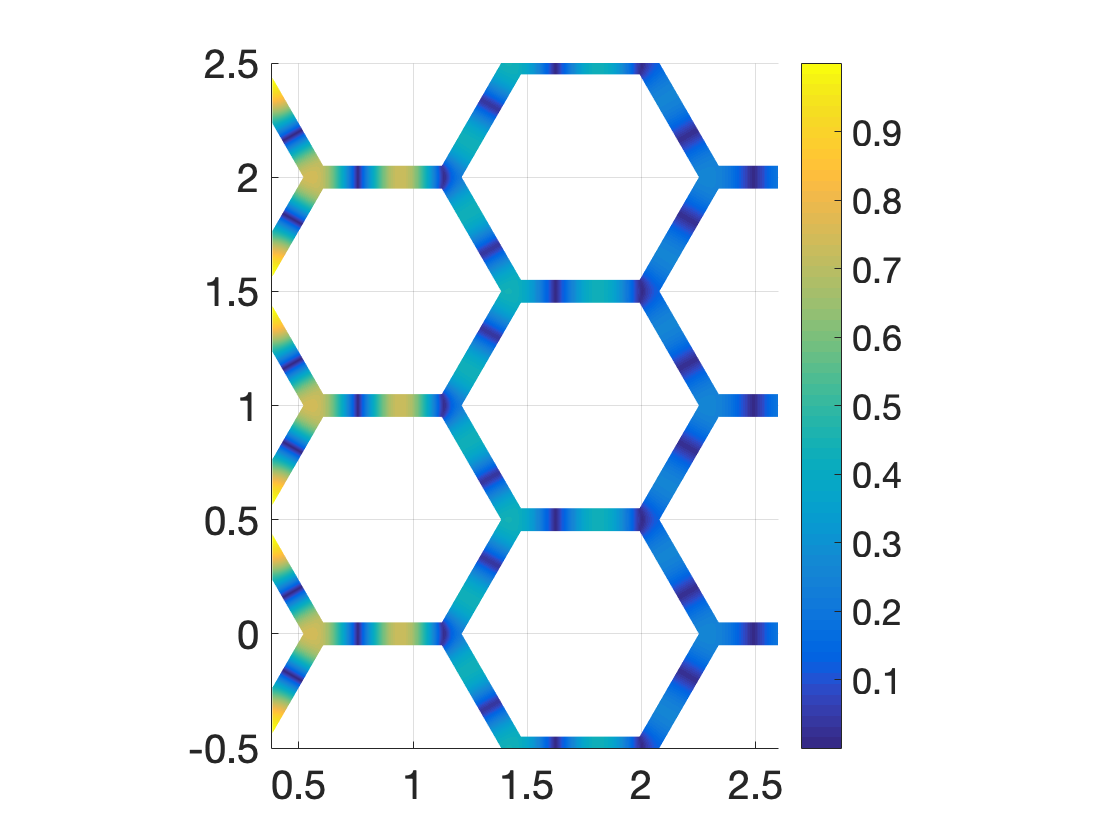}
\includegraphics[width=0.4\textwidth, trim=0cm 0cm 0cm 0cm, clip]
 {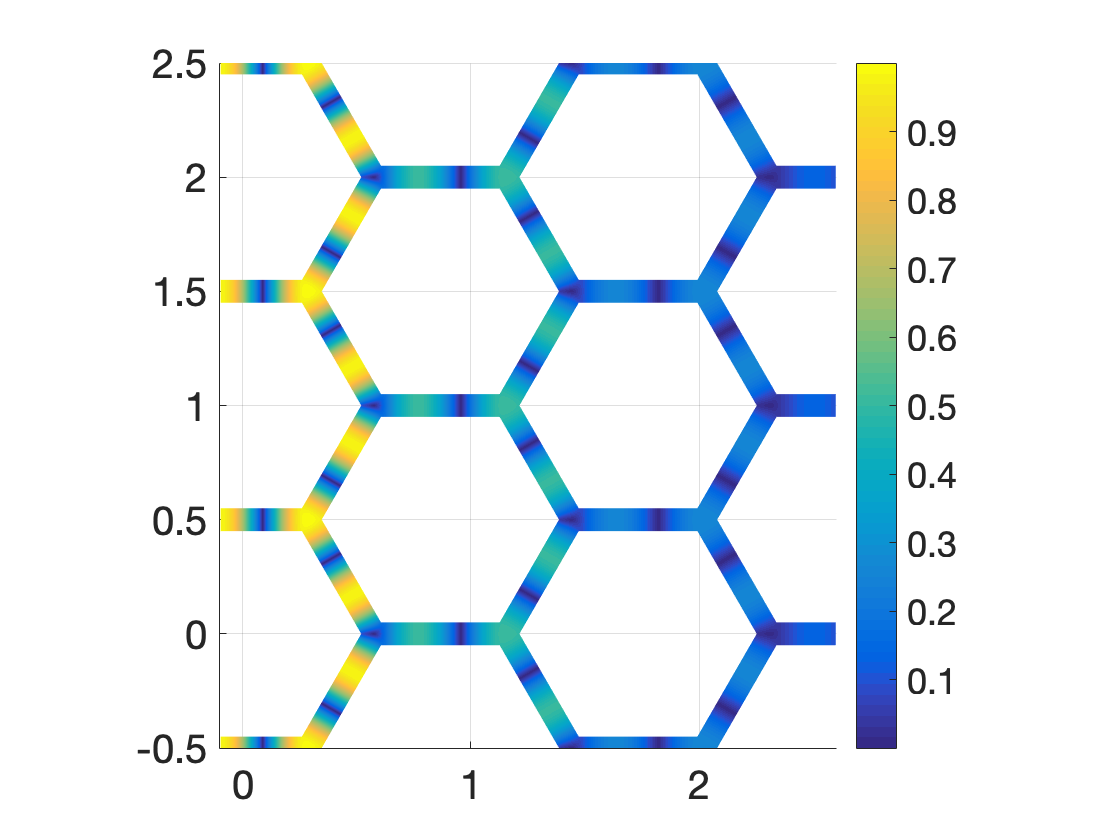}
  \caption{Examples of eigenvectors for $\delta =0.1$:  $t={L}/{3}$ and $\beta ={1}/{6}$ (left)  $t={4L}/{3}$ and $\beta={5}/{12}$ (right).}
	 \label{fig:VPHighFrequency}
 	\end{center}
\end{figure}

	\appendix
    \section{Proof of Proposition \ref{CharacSpectrumGraphHexa}}
    \label{annex:proof_propquantum}
    Let us fix $\kg\in\mathcal{B}$. Suppose that $u\in {D}(\cA(\kg))$ and $\cA(\kg) u=\lambda u$. If $\sqrt{\lambda}\notin \N\pi/L$, by definition of $\cA(\kg)$ and using the parametrization \eqref{eq:param}, we deduce that
        $$\begin{array}{|ll}
        \dsp u\big|_{e_0}(t)=u(A)\frac{\sin(\sqrt{\lambda} (L-t))}{\sin\sqrt{\lambda} L}+u(B)\frac{\sin(\sqrt{\lambda} t)}{\sin \sqrt{\lambda} L},&\\ [5pt]
        \dsp u|_{e_i}(t)=u(A)\frac{\sin(\sqrt{\lambda}(L-t))}{\sin\sqrt{\lambda} L}+u(B)e^{ 2 \imath \pi \kg\cdot\vg_i}\,\frac{\sin(\sqrt{\lambda} t)}{\sin\sqrt{\lambda} L}&\text{for }i\in\{1,2\},
        \end{array}
        $$
        where we have used that $u\in H^1_{\kg}(\cG^\sharp)$. Since $u\in {D}(\cA(\kg))$, the Kirchhoff condition at $A$ and $B$ are satisfied and yield respectively
        $$
        3u(A)\,\cos\sqrt{\lambda} L+(1+e^{2\imath\pi \kg\cdot\vg_1}+e^{2\imath\pi \kg\cdot\vg_2})\,u(B)=0,
        $$
        and
        $$
        (1+e^{-2\imath\pi \kg\cdot\vg_1}+e^{-2\imath\pi \kg\cdot\vg_2})\,u(A)+3u(B)\cos\sqrt{\lambda} L\,=0.
        $$
        We conclude that if $\sqrt{\lambda}\notin \N\pi/L$, there exists a non trivial function $u\in {D}(\cA(\kg))$ such that $\cA(\kg) u=\lambda u$ if and only if 
        \begin{equation}\label{DispersionRelation}
        \cos^2 \sqrt{\lambda} L = \frac{1}{9} \left| 1 +  e^{2\imath\pi \vg_1\cdot\kg} + e^{2\imath\pi\vg_2\cdot\kg}\right|^2
        \end{equation}
        If $\sqrt{\lambda}\in\N\pi/L$, the function $u$ defined by $u|_{e_i}(t)=\alpha_i\,\sin(\sqrt{\lambda} t)$ is in $D(\cA(\kg))$ for all $\kg$ if $\sum_{i=1}^3 \alpha_i=0$ and it satisfies $\cA(\kg) u=\lambda u$. This ends the proof of the first part of the proposition.
      
        Let us now introduce the function
        $f: \,\thetag = (\theta_1, \theta_2)\in\R^2 \mapsto  {1}/{9} | 1 + e^{i \theta_1} + e^{i \theta_2}|^2 \in\R^+$.
        By noting that
        $$
        f(\thetag)= \frac{1}{9} \left( 1 + 8 \cos(\frac{\theta_1 - \theta_2}{2}) \cos(\frac{\theta_1}{2}) \cos(\frac{\theta_2}{2})  \right)\;\;\forall \thetag\in\R^2,
        $$
        one can show that $f(\R^2) \subset [0,1]$. Moreover, we show that
        \begin{equation}\label{DispersionRelation0}
         f(\thetag) = 0\quad  \Leftrightarrow \quad \thetag =   \thetag_{(p,q)}^\pm \quad \mbox{where } \; \thetag_{(p,q)}^\pm = \pm ( \frac{2\pi}{3} + 2 p \pi, - \frac{2\pi}{3} + 2 q \pi) \quad \forall (p, q)\in \Z.
        \end{equation}
        and
        $$
         f(\thetag) = 1 \quad\Leftrightarrow \quad\thetag = \thetag_{(p,q)} \quad \mbox{where } \; \thetag_{(p,q)} = (2 p\pi, 2 q\pi)\; \quad \forall (p,q) \in\Z^2.
        $$
        By using the first part of the proposition, we deduce that $\lambda_0=0$ and
        
            \begin{equation}\label{eq:lambda_1_2}
              \begin{array}{l}
         \sqrt{\lambda_1(\kg)}L=\arccos(\sqrt{f(\thetag(\kg))})\in(0,\pi/2),\\[3pt]
         \sqrt{\lambda_2(\kg)}L=\pi -\arccos(\sqrt{f(\thetag(\kg))}) \in(\pi/2,\pi),\\[3pt]
          \sqrt{\lambda_3(\kg)} L = \pi,
            \end{array}
        \end{equation}
        where $\thetag(\kg)=2\pi(\kg\cdot\vg_1,\kg\cdot\vg_2)$. By periodicity with respect to $\sqrt{\lambda} L$ of the dispersion relation \eqref{eq:disp_rel}, we show easily that
        \begin{equation}\label{eq:per_lambda}
        \forall i\in\{0,1,2\},\quad \sqrt{\lambda_{3n+i}}=\sqrt{\lambda_i}+n\pi/L.
        \end{equation}
        We deduce that it suffices to prove \eqref{eq:Dirac_graph} for $n=0$ to deduce it for all  \Becommente{$n$}. Note that
         $$
          \sqrt{\lambda_{1}(\kg)}L = \sqrt{\lambda_{2}}(\kg)L =  \pi/2  \quad
         \Leftrightarrow\quad (\kg\cdot\vg_1,  \kg\cdot\vg_2) =\pm (\frac{1}{3},- \frac{1}{3})\quad 
         \Leftrightarrow  \dsp \quad \kg\in\{{\bf K}, \,{\bf K}'\} .
         $$ 
         As a result, the dispersion surfaces $\lambda_{1}$ and $\lambda_{2}$ intersect at the vertices of the Brillouin zone. Since $\nabla_{\kg}[f(\thetag(\kg))]= {{2 \pi}}[\vg_1 \;\vg_2]\nabla_\theta f (\thetag(\kg))$ and the $\thetag_{(p,q)}^\pm$ are minima of $f$, we have $\nabla_{\kg}[f(\thetag(\kg))]=0$ for $\kg={\bf K}$ and $\kg={\bf K}'$ (and all the other vertices of the Brillouin zone). Moreover, since $ \nabla_{\kg}\nabla_{\kg}^t[f(\thetag(\kg))]={(2 \pi)^2}[\vg_1\; \vg_2]\nabla_\theta\nabla_\theta^t f (\thetag(\kg))[\vg_1\;\vg_2]^t$, we have
        $$
        [\nabla_\theta\nabla_\theta^t f] (\thetag)= \frac{1}{9}\left[  \begin{matrix}2 & -1  \\ 
          -1 & 2 \end{matrix} \right]\;\text{for}\;\thetag=\thetag_{(p,q)}^\pm\quad\Rightarrow\quad
        \nabla_{\kg}\nabla_{\kg}^t[f(\thetag(\kg))]=\frac{{{(2 \pi)^2}}}{6}\left[ \begin{matrix}1 & 0  \\ 
          0 & 1 \end{matrix} \right]\;\text{for} \;\kg\in\{{\bf K},{\bf K}'\},$$
        which yields
        $$
        \forall \kg\in\mathcal{B},\quad f(\thetag(\kg))= {\frac{ \pi^2}{3}}\|\kg-{\bf K}^*\|^2+ \mathcal{O}(\| \kg-{\bf K}^* \|^3)\;\;\text{for} \;{\bf K}^*\in\{{\bf K},{\bf K}'\},
        $$
         and finally, by \eqref{eq:lambda_1_2}
          \[
             \forall \kg\in\mathcal{B},\quad\begin{array}{|l}
                \dsp \sqrt{\lambda_{1}(\kg)} L=   \frac{\pi}{2 } - {\ \frac{\pi}{\sqrt{3}}} \|\kg-{\bf K}^*\| + \mathcal{O}(\| \kg-{\bf K}^* \|^2),\\[5pt]
                \dsp \sqrt{\lambda_{2}(\kg)} L =   \frac{\pi}{2 } + { \frac{\pi}{\sqrt{3}}} \|\kg-{\bf K}^*\| + \mathcal{O}(\| \kg-{\bf K}^* \|^2).
                \end{array}
         \]

  \section{Proof of Proposition \ref{PropositionDiracgrapheEpais}}\label{annex:proof_propDirac}
    If $\phi_{\delta,1}$ is an eigenvector of  $A_{\delta,1}({\bf K}^*)$ associated with the eigenvalue $\lambda_\delta$ then by Proposition \ref{sym_eigenvectors}, 
    $\phi_{\delta,2}:=\overline{\mathcal{S}\phi_{\delta,1} }$ is  an eigenvector of $A_{\delta,2}({\bf K}^*)$ associated with the same eigenvalue $\lambda_\delta$.
  
  {\bf Step 1: Preliminary notation}
  ~\\It is easy to see that for any ${\bf k}$,
  $
  A^\delta({\bf k})=e^{-2\imath\pi {\bf k}\cdot{\bf x}} \hat{A}^\delta({\bf k}) e^{2\imath\pi {\bf k}\cdot{\bf x}},
  $
  where \\
  $ \hat{A}^\delta({\bf k})=-(\nabla+2\imath\pi{\bf k})^2, \quad D(\hat{A}^\delta({\bf k}))= \left\{  v \in  H^1_{\mathbf{0}}(\Omega_\delta,\Delta), (\nabla+2\imath\pi{\bf k}) v \cdot n = 0 \; \mbox{on} \; \partial \Omega_\delta  \right\}.
  $

  It is easy to see that the eigenvalues of $A^\delta({\bf k})$ coincide with the ones of $\hat{A}^\delta({\bf k})$ and the associated eigenvectors (resp. $\phi^\delta$ and $\hat{\phi}^\delta$) are related by $\hat{\phi}^\delta=e^{-2\imath\pi {\bf k}\cdot {\bf x}}{\phi}^\delta $. Because of the geometry and the boundary conditions, the domain of the operators $\hat{A}^\delta({\bf k})$ still depends on ${\bf k}$. It will be convenient in the sequel to use bilinear forms instead of operators. Let us then introduce
  \begin{equation}\label{eq:form_bil}
  \forall u,v \in H^1_{\mathbf{0}}(\Omega_\delta),\quad [a^\delta({\bf k})](u,v):=\int_{\mathcal{C}_\delta^\sharp}(\nabla+2\imath\pi{\bf k}) u\cdot\overline{(\nabla+2\imath\pi{\bf k}) v}
  \end{equation}
  and it is easy to see that
  \[
   \forall u\in  D(\hat{A}^\delta({\bf k})),\;\forall v \in H^1_{\mathbf{0}}(\Omega_\delta),\quad \int_{\mathcal{C}_\delta^\sharp}\hat{A}^\delta({\bf k})u\,\overline{v} = [a^\delta({\bf k})](u,v).
  \]
  \Becommente{Moreover,}  $\lambda^\delta({\bf k})$ is an eigenvalue of $\hat{A}^\delta({\bf k})$ if and only if there exists $\hat{\phi}^\delta({\bf k})\in H^1_{\mathbf{0}}(\Omega_\delta)$ such that
  \begin{equation}\label{eq:form_bil_eig}
  \forall v \in H^1_{\mathbf{0}}(\Omega_\delta),\quad[a^\delta({\bf k})](\hat{\phi}^\delta({\bf k}),v)=\lambda^\delta({\bf k})\,\int_{\mathcal{C}_\delta^\sharp}\hat{\phi}^\delta({\bf k})\,\overline{v}.
  \end{equation}
  Let us consider now ${\bf k}=\boldsymbol{\xi}+{\bf K}^*$ with ${\bf K}^*\in\{{\bf K},{\bf K}'\}$ and $\boldsymbol{\xi}$ small enough (this will become more precise later on) and the remainder \begin{equation}\label{eq:defremainder} R(\boldsymbol{\xi}):=\lambda^\delta(\boldsymbol{\xi}+{\bf K}^*)-\lambda^\delta({\bf K}^*).
  \end{equation} 
  Since, ${\bf k}\mapsto \lambda^\delta({\bf k})$ is Lipschitz continuous we know that $R$ tends to $0$ when $\boldsymbol{\xi}$ tends to ${\bf 0}$. We want to study the precise behaviour of $R$ for small $\boldsymbol{\xi}$. Let $\hat{\phi}_{\delta}(\boldsymbol{\xi}+{\bf K}^*)$ be an eigenvector of $\hat{A}^\delta(\boldsymbol{\xi}+{\bf K}^*)$ associated to the eigenvalue $\lambda_\delta(\boldsymbol{\xi}+{\bf K}^*)$. We can decompose it as
  \begin{equation}\label{eq:decomp_eigenvector}
  \begin{array}{c}
  \hat{\phi}_{\delta}(\boldsymbol{\xi}+{\bf K}^*) = \alpha_1(\boldsymbol{\xi}) \hat{\phi}_{\delta,1} + \alpha_2(\boldsymbol{\xi})\hat{\phi}_{\delta,2}+\hat{\phi}^R(\boldsymbol{\xi}),\\[3pt]
  \text{where for }i\in\{1,2\}, \;\alpha_i(\boldsymbol{\xi})=(\hat{\phi}_{\delta}(\boldsymbol{\xi}+{\bf K}^*),\hat{\phi}_{\delta,i})_{L^2({\cal C}_\delta^\sharp)}\text{ and  }(\hat{\phi}^R(\boldsymbol{\xi}),\hat{\phi}_{\delta,i})_{L^2({\cal C}_\delta^\sharp)}=0.
  \end{array}
  \end{equation}
   By \eqref{eq:form_bil}, we can decompose $a^\delta(\boldsymbol{\xi}+{\bf K}^*)$ as follows
  \[
  a^\delta(\boldsymbol{\xi}+{\bf K}^*)=a^\delta({\bf K}^*)+2\imath\pi\boldsymbol{\xi}\cdot {\bf b}^\delta({\bf K}^*)+(2\pi)^2\|\boldsymbol{\xi}\|^2\,i_{L^2}
  \]
  where for all ${\bf k}$
  \begin{equation}\label{eq:form_bil_bc}\begin{array}{ll}
  \forall u,v \in H^1_{\mathbf{0}}(\Omega_\delta),&
  [{\bf b}^\delta({\bf k})](u,v):=\int_{\mathcal{C}_\delta^\sharp} \big[u\,\overline{(\nabla+2\imath\pi{\bf k}) v}-(\nabla+2\imath\pi{\bf k})u\,\overline{ v}\big]\\[3pt]
  &i_{L^2}(u,v)=\int_{\mathcal{C}_\delta^\sharp}u\overline{v}.
  \end{array}
  \end{equation}
  Plugging this decomposition in \eqref{eq:form_bil_eig}, with ${\bf k}=\boldsymbol{\xi}+{\bf K}^*$, we obtain that $\forall v \in H^1_{\mathbf{0}}(\Omega_\delta)$
  \begin{multline*}
   [a^\delta({\bf K}^*)-\lambda({\bf K}^*) i_{L^2}](\hat{\phi}^\delta({\bf k}),v)=-2\imath\pi\boldsymbol{\xi}\cdot [{\bf b}^\delta({\bf K}^*)](\hat{\phi}^\delta({\bf k}),v) +(R(\boldsymbol{\xi})-(2\pi)^2\|\boldsymbol{\xi}\|^2 )\,(\hat{\phi}^\delta({\bf k}),v)_{L^2}.
  \end{multline*}
  This can be rewritten in an operator form. By introducing, thanks to Riesz theorem
  \[
  \begin{array}{l}
  \mathbb{A}^\delta_{{\bf K}^*}: H^1_{\mathbf{0}}(\Omega_\delta)\rightarrow \Becommente{H^1_{\mathbf{0}}(\Omega_\delta)},\quad \forall u,v \in H^1_{\mathbf{0}}(\Omega_\delta),\; ( \mathbb{A}^\delta_{{\bf K}^*}u,v)_{H^1}=[a^\delta({\bf K}^*)](u,v)\\[3pt]
  \mathbb{I}_{L^2}: H^1_{\mathbf{0}}(\Omega_\delta)\rightarrow \Becommente{H^1_{\mathbf{0}}(\Omega_\delta)},\quad \forall u,v \in H^1_{\mathbf{0}}(\Omega_\delta),\; (\mathbb{I}_{L^2}u,v)_{H^1}=(u,v)_{L^2}\\[3pt]
  \mathbb{B}^\delta_{{\bf K}^*}: H^1_{\mathbf{0}}(\Omega_\delta)\rightarrow [H^1_{\mathbf{0}}(\Omega_\delta)]^2,\quad \forall u,v \in H^1_{\mathbf{0}}(\Omega_\delta),\; \langle \mathbb{B}^\delta_{{\bf K}^*}u,v\rangle=[{\bf b}^\delta({\bf K}^*)](u,v)
  \end{array}
  \]
  with $(\cdot,\cdot)_{H^1}$ denoting the scalar product of $\Becommente{H^1(\Omega_\delta)}$, we have 
  \begin{multline*}
   [\mathbb{A}^\delta_{{\bf K}^*}-\lambda({\bf K}^*) \mathbb{I}_{L^2}]\,\hat{\phi}^\delta({\bf k})=-2\imath\pi\boldsymbol{\xi}\cdot {\mathbb{B}}^\delta_{{\bf K}^*}\,\hat{\phi}^\delta({\bf k}) +(R(\boldsymbol{\xi})-(2\pi)^2\|\boldsymbol{\xi}\|^2 )\,\mathbb{I}_{L^2}\,\hat{\phi}^\delta({\bf k}).
  \end{multline*}
  We can now plug \eqref{eq:decomp_eigenvector} into the previous formula to finally obtain
  \begin{multline}\label{eq:fred1}
   [\mathbb{A}^\delta_{{\bf K}^*}-\lambda({\bf K}^*) \mathbb{I}_{L^2}]\,\hat{\phi}^R(\boldsymbol{\xi})=\big[-2\imath\pi\boldsymbol{\xi}\cdot {\mathbb{B}}^\delta_{{\bf K}^*}\, +[R(\boldsymbol{\xi})-(2\pi)^2\|\boldsymbol{\xi}\|^2 ]\,\mathbb{I}_{L^2}\big]\,\hat{\phi}^R(\boldsymbol{\xi}) \\
   +\big[-2\imath\pi\boldsymbol{\xi}\cdot {\mathbb{B}}^\delta_{{\bf K}^*}\, +(R(\boldsymbol{\xi})-(2\pi)^2\|\boldsymbol{\xi}\|^2 )\,\mathbb{I}_{L^2}\big]\,\sum_{j=1,2}\alpha_j(\boldsymbol{\xi}) \hat{\phi}_{\delta,j}.
  \end{multline}
  {\bf Step 2: Schur complement reduction}
  ~\\ By Rellich theorem, $\mathbb{I}_{L^2}$ is a compact operator. This implies that $[\mathbb{A}^\delta_{{\bf K}^*}-\lambda({\bf K}^*) \mathbb{I}_{L^2}]$ is a Fredholm operator of index $0$. Moreover, $\lambda({\bf K}^*)$ is an eigenvalue of multiplicity 1 of $A_{\delta,1}({\bf K}^*)$  and $A_{\delta,2}({\bf K}^*)$ but not of $A_{\delta,0}({\bf K}^*)$. Then, we can show easily that $\lambda({\bf K}^*)$ is an eigenvalue of multiplicity 2 of $\mathbb{A}^\delta_{{\bf K}^*}$ and its kernel, denoted $\mathcal{N}$, is span$(\hat{\phi}_{\delta,1},\hat{\phi}_{\delta,2})$.  We deduce that (1) equation \eqref{eq:fred1} has a solution if and only if the r.h.s. is orthogonal in $H^1$ to $\mathcal{N}$ and (2) if $\mathbb{P}$ denotes the orthogonal projection on $\mathcal{N}^\perp$ for the $H^1$-scalar product, then $\mathbb{P}[\mathbb{A}^\delta_{{\bf K}^*}-\lambda({\bf K}^*) \mathbb{I}_{L^2}]\mathbb{P}$ is invertible. Since by \eqref{eq:decomp_eigenvector} and by definition of $\mathbb{I}_{L^2}$
  \begin{equation}\label{eq:proj_eigen}
  \forall i\in\{1,2\},\quad (\mathbb{I}_{L^2}\,\hat{\phi}^R(\boldsymbol{\xi}),\hat{\phi}_{\delta,i})_{H^1} = 0,
  \end{equation}
  the statement (1) implies that $\forall i\in\{1,2\},$
  \begin{multline}\label{eq:cond1}
  ( 2\imath\pi\boldsymbol{\xi}\cdot {\mathbb{B}}^\delta_{{\bf K}^*} \hat{\phi}^R(\boldsymbol{\xi}),\hat{\phi}_{\delta,i})_{H^1}\\ = (\,\big[-2\imath\pi\boldsymbol{\xi}\cdot {\mathbb{B}}^\delta_{{\bf K}^*}\, +(R(\boldsymbol{\xi})-4\pi^2\|\boldsymbol{\xi}\|^2 )\,\mathbb{I}_{L^2}\big]\,\sum_{j=1,2}\alpha_j(\boldsymbol{\xi}) \hat{\phi}_{\delta,j} ,\hat{\phi}_{\delta,i} \,)_{H^1}
  \end{multline} 
  Moreover applying $\mathbb{P}$ to \eqref{eq:fred1}, by using that $\mathbb{P}\hat{\phi}^R(\boldsymbol{\xi})=\hat{\phi}^R(\boldsymbol{\xi})$, $\mathbb{P}\mathbb{I}_{L^2}\hat{\phi}^R(\boldsymbol{\xi})=\hat{\phi}^R(\boldsymbol{\xi})$) and $\mathbb{P}^*\mathbb{I}_{L^2}\hat{\phi}_{\delta,j}=0$ for $j=1,2$, which follows from \eqref{eq:proj_eigen}, we obtain
  \begin{equation}\label{eq:fred2}
   \mathcal{A}(\boldsymbol{\xi},R(\boldsymbol{\xi}))\,\hat{\phi}^R(\boldsymbol{\xi}) 
   =\big[-2\imath\pi\mathbb{P}\boldsymbol{\xi}\cdot {\mathbb{B}}^\delta_{{\bf K}^*}\, \big]\,\sum_{j=1,2}\alpha_j(\boldsymbol{\xi}) \hat{\phi}_{\delta,j}.
  \end{equation}
  where 
  \begin{equation}\label{eq:A_xiR}
  \mathcal{A}(\boldsymbol{\xi},R):=\mathbb{P}[\mathbb{A}^\delta_{{\bf K}^*}-\lambda({\bf K}^*) \mathbb{I}_{L^2}]\mathbb{P}\,\big[\mathbb{I}-\mathcal{A}_1(\boldsymbol{\xi},R)\big]
\end{equation}
  and
  \begin{equation}\label{eq:A1_xiR}
  \mathcal{A}_1(\boldsymbol{\xi},R):=\big[\mathbb{P}[\mathbb{A}^\delta_{{\bf K}^*}-\lambda({\bf K}^*) \mathbb{I}_{L^2}]\mathbb{P}\big]^{-1}\,\big[-2\imath\pi\mathbb{P}\boldsymbol{\xi}\cdot {\mathbb{B}}^\delta_{{\bf K}^*}\, +(R-4\pi^2\|\boldsymbol{\xi}\|^2 )\,\mathbb{I}_{L^2}\big].
  \end{equation}
   When $\boldsymbol{\xi}$ and $R$ are small enough, we have
  \begin{equation}\label{eq:A_A1_prop}
  \|\mathcal{A}_1(\boldsymbol{\xi},R)\|<1,\quad\text{and}\;\big[\mathbb{I}-\mathcal{A}_1(\boldsymbol{\xi},R)\big]^{-1}=\sum_{n\in\N}(\mathcal{A}_1(\boldsymbol{\xi},R))^n.
  \end{equation}
   When $\boldsymbol{\xi}$ tends to $0$, we know that the remainder $R(\boldsymbol{\xi})$ tends also to $0$.
  This means that for $\boldsymbol{\xi}$ small enough, $\mathcal{A}_1(\boldsymbol{\xi}, R (\boldsymbol{\xi}))$ is invertible and $\mathcal{A}(\boldsymbol{\xi}, R (\boldsymbol{\xi}))$  also. We can then deduce from \eqref{eq:fred2} $\hat{\phi}^R(\boldsymbol{\xi})$ in terms of  $\hat{\phi}_{\delta,1}$ and $\hat{\phi}_{\delta,2}$.
  If we replace this expression in \eqref{eq:cond1}, we obtain 
  \begin{equation}\label{eq:syst_22}
  \left( (4\pi^2\|\boldsymbol{\xi}\|^2-R(\boldsymbol{\xi}))\mathbb{I}_2 + \|\boldsymbol{\xi}\|^2  B_2(\boldsymbol{\xi},R(\boldsymbol{\xi})) +  B_1(\boldsymbol{\xi})\right)\left[\begin{matrix}\alpha_1(\boldsymbol{\xi})\\\alpha_2(\boldsymbol{\xi})\end{matrix}\right]=0
  \end{equation}
  where $\forall\, i,j\in\{1,2\}$
  \begin{equation}\label{eq:Bi}
  \begin{array}{l}
  \|\boldsymbol{\xi}\|^2 [B_2(\boldsymbol{\xi},R(\boldsymbol{\xi}))]_{ij}:=(\,4\pi^2\boldsymbol{\xi}\cdot {\mathbb{B}}^\delta_{{\bf K}^*}\,\mathcal{A}(\boldsymbol{\xi}, R(\boldsymbol{\xi}))^{-1} \,\mathbb{P}^*\boldsymbol{\xi}\cdot {\mathbb{B}}^\delta_{{\bf K}^*}\hat{\phi}_{\delta,j} ,\hat{\phi}_{\delta,i}\,)_{H^1},\\[3pt]
    [B_1(\boldsymbol{\xi})]_{ij}:=(\,2\imath\pi \boldsymbol{\xi}\cdot {\mathbb{B}}^\delta_{{\bf K}^*}\hat{\phi}_{\delta,j} ,\hat{\phi}_{\delta,i}\,)_{H^1}
  \end{array}\end{equation}
  By using \eqref{eq:A_xiR}, \eqref{eq:A1_xiR} and \eqref{eq:A_A1_prop}, note that for $\boldsymbol{\xi}$ small enough
  \begin{equation}\label{eq:Bi_comp}
  \|\boldsymbol{\xi}\|^2 [B_2(\boldsymbol{\xi},R(\boldsymbol{\xi}))]_{ij}=\mathcal{O}(\|\boldsymbol{\xi}\|^2+R(\boldsymbol{\xi}))\|\boldsymbol{\xi}\|^2).
  \end{equation}
  Finally, there exists an eigenvector as in \eqref{eq:decomp_eigenvector} if and only if this $2\times 2$ system has a solution which is equivalent to
  \begin{equation}\label{eq:det}
  \text{det}\left( (4\pi^2\|\boldsymbol{\xi}\|^2-R(\boldsymbol{\xi}))\mathbb{I}_2 + \|\boldsymbol{\xi}\|^2  B_2(\boldsymbol{\xi},R(\boldsymbol{\xi})) +  B_1(\boldsymbol{\xi})\right)=0
  \end{equation}
  {\bf Step 3: Behaviour of the remainder} ~\\
  By Proposition \ref{sym_eigenvectors} and since $\hat{\phi}_{\delta,i}=e^{-2\imath\pi\mathbf{k}\cdot\mathbf{x}}{\phi}_{\delta,i}$, we have
  \[
  B_1(\boldsymbol{\xi})=\left[\begin{matrix}0&4\imath\pi v_\delta(\xi_1+\imath\xi_2)\\-4\imath\pi v_\delta(\xi_1-\imath\xi_2)
  \end{matrix}\right].
  \]
  The relation \eqref{eq:det} then rewrites
  \begin{equation}\label{eq:eqR}
  R^2-16\pi^2v_\delta^2\|\boldsymbol{\xi}\|^2=F(\boldsymbol{\xi},R),
  \end{equation}
  where $F$ is a smooth function with respect to $R$, lipschitz continuous which respect to $\boldsymbol{\xi}$ and by \eqref{eq:Bi_comp}, satisfies  for $\boldsymbol{\xi}$ and $R$ small enough
  \begin{equation}\label{eq:FR}
  F(\boldsymbol{\xi},R)=\mathcal{O}(\|\boldsymbol{\xi}\|^4+R\|\boldsymbol{\xi}\|^2)
  \end{equation}
  We suspect then that for $\boldsymbol{\xi}$ small enough, $R(\boldsymbol{\xi})^2\approx16\pi^2v_\delta^2\|\boldsymbol{\xi}\|^2$. Suppose that $v_\delta\neq 0$ and let us introduce $\eta_\pm(\boldsymbol{\xi})$ such that $R(\boldsymbol{\xi})=\pm 4\pi v_\delta\|\boldsymbol{\xi}\|(1+\eta_\pm(\boldsymbol{\xi}))$. From \eqref{eq:eqR} and \eqref{eq:FR}, we obtain if $v_\delta\neq 0$ that
  \begin{equation}\label{eq:eqR2}
  2\eta_\pm+\eta_\pm^2=G_\pm(\boldsymbol{\xi},\eta_\pm),
  \end{equation}
  where 
  $G_\pm$ is a smooth function with respect to $\eta_\pm$, lipschitz continuous which respect to $\boldsymbol{\xi}$ and satisfies by \eqref{eq:FR} for $\boldsymbol{\xi}$ and $\eta_\pm$ small enough
  \begin{equation}\label{eq:GR}
  G(\boldsymbol{\xi},\eta_\pm)=\mathcal{O}(\|\boldsymbol{\xi}\|^2+\eta_\pm\|\boldsymbol{\xi}\|)
  \end{equation}
  By the implicit function theorem, we deduce that for $\boldsymbol{\xi}$ small enough, $\eta_\pm$ is a Lipschitz function of $\boldsymbol{\xi}$ and $\eta_\pm\rightarrow 0$ when $\boldsymbol{\xi}\rightarrow 0$.\\
  {\bf Conclusion : } By definition \eqref{eq:defremainder} of the remainder, we have finally shown that if $\lambda$ is an eigenvalue of $A_\delta({\bf K}^*)$ where ${\bf K}^*\in\{{\bf K},{\bf K}'\}$ such that  $\lambda$ is a simple eigenvalue of $A_{\delta,1}({\bf K}^*)$ and of $A_{\delta,2}({\bf K}^*)$ but not of $A_{\delta,0}({\bf K}^*)$ and if $v_\delta$ defined \eqref{eq:v_delta} does not vanish, then for $\boldsymbol{\xi}$ small enough, there exists two eigenvalues of $A_\delta({\bf K}^*+\boldsymbol{\xi})$ having the behaviour \eqref{eq:Dirac points} with $\alpha^*=4\pi v_\delta$.
 	
\section{Technical results}\label{AppendixTechnique}
\begin{lemma}\label{LemmeEstimationNorme}
Let $\hat{\varphi}_\delta$ be defined by \eqref{DefinitionPhideltaHat}. Then,
$$
\left\| \hat{\varphi}_\delta\right\|_{L^2(\mathcal{C}_\delta^\sharp)} = \frac{\sqrt{3 \delta}}{\sqrt{2}}   + O(\delta^{3/2})
\quad\text{and}\quad
 \left\| \nabla\hat{\varphi}_\delta\right\|_{L^2(\mathcal{C}_\delta^\sharp)} = \frac{\sqrt{3 \delta}}{\sqrt{2}} \sqrt{\lambda_n}  + O(\delta^{3/2})
$$ 
\end{lemma}
\begin{proof}
We have, using the notation of Figure \ref{fig:periodicitycell}
$$
\left\| \hat{\varphi}_\delta(x)\right\|_{L^2(\mathcal{C}_\delta^\sharp)}^2 = \int_{J_\delta^A} \left|\hat{\varphi}_\delta(x)\right|^2 dx + \sum_{i=0}^2 \left( \int_{\tilde{e}_{i,\delta}} \left| \hat{\varphi}_\delta(x)\right|^2 dx + \int_{J_{i,\delta}^B}  \left|\hat{\varphi}_\delta(x)\right|^2 dx \right)
$$
where, by using \eqref{DefinitionPhideltaHat}
$$
\int_{J_\delta^A} \left|\hat{\varphi}_\delta(x)\right|^2 dx = \text{Meas}(J_\delta^A) = \frac{\sqrt{3} \delta^2}{8} \quad \mbox{ and } \int_{{J_{i,\delta}^B}}  \left|\hat{\varphi}_\delta(x)\right|^2 dx =0.
$$ 
Moreover, still using \eqref{DefinitionPhideltaHat}
$$
\int_{\tilde{e}_{i,\delta}} \left| \hat{\varphi}_\delta(x)\right|^2 dx  = \frac{L-\sqrt{3} \delta}{L} \int_0^L \frac{1}{L} \cos^2(\sqrt{\lambda_n} s) ds = \delta \frac{L-\sqrt{3} \delta}{ 2 L} 
$$
which allows to deduce the first result.
Similarly, by using \eqref{DefinitionPhideltaHat}
$$
\left\|\nabla \hat{\varphi}_\delta(x)\right\|_{L^2(\mathcal{C}_\delta^\sharp)}^2 =\sum_{i=0}^2  \int_{\tilde{e}_{i,\delta}} \left| \nabla \hat{\varphi}_\delta(x)\right|^2 dx =3 \delta \int_{\frac{ \sqrt{3}\delta}{2}}^{L -\frac{\sqrt{3}\delta}{2} }  (u'(g(t))^2 g'(t)^2 dt=3\delta  \frac{\lambda_n }{2} \frac{L }{L-\sqrt{3}\delta}
$$ 
which yields the second result.
% with
% \begin{multline*}
% \int_{\tilde{e}_{i,\delta}} \left| \nabla \hat{\varphi}_\delta(x)\right|^2 dx = \delta \int_{\frac{ \sqrt{3}\delta}{2}}^{L -\frac{\sqrt{3}\delta}{2} }  \left(\partial_t \left(u(g(t) \right))\right)^2 dt=\delta \int_{\frac{ \sqrt{3}\delta}{2}}^{L -\frac{\sqrt{3}\delta}{2} }  (u'(g(t))^2 g'(t)^2 dt \\=  \frac{\delta \lambda_n}{L} \frac{L }{L-\sqrt{3}\delta} \int_0^L \sin^2(\sqrt{\lambda_n} s) ds =\delta  \frac{\lambda_n }{2} \frac{L }{L-\sqrt{3}\delta}
% \end{multline*} 
% Therefore
% $$
% \left\|\nabla \hat{\varphi}_\delta(x)\right\|_{L_2(\mathcal{C}_\delta^\sharp)}^2 = \frac{3 \delta\lambda_n }{2} + O(\delta^2).
% $$ 
% As a result,
% $$
% \left\| \hat{\varphi}_\delta(x)\right\|_{H_1(\mathcal{C}_\delta^\sharp)} =\frac{\sqrt{3 \delta}}{\sqrt{2}} \sqrt{ 1 + \lambda_n}  + O(\delta^{3/2})
% $$ 
\end{proof}

\section{Proof of Lemmma~\ref{LemmeQuasiMode}}

Let $\tilde\phi_\delta$ be an eigenvector of $\mathcal{A}_{\delta,1}({\bf K})$ associated with the eigenvalue $\lambda_{n,\delta}$ such that $\|\tilde\phi_\delta \|_{L^2(\mathcal{C}_\delta^\sharp)}=1$. Then let 
$$ r_\delta = \varphi_\delta - \left(  \int_{\mathcal{C}_\delta^\sharp}  \varphi_\delta\overline{\tilde\phi_\delta} \right) \tilde\phi_\delta $$
Note that
$$ 
\int_{\mathcal{C}_\delta^\sharp} r_\delta \overline{\tilde\phi_\delta} dx =0,\quad
\text{and} 
\quad
\int_{\mathcal{C}_\delta^\sharp} \nabla r_\delta \cdot \overline{\nabla \varphi} - \lambda_{n,\delta} \int_{\mathcal{C}_\delta^\sharp} r_\delta\overline{ \varphi}  = L_\delta( \varphi),\;\forall\varphi\in H^1(\mathcal{C}_\delta^\sharp)
$$
where $L_\delta( \tilde\phi_\delta)=0$ and because of~\eqref{Quasi-modes}, we have
$$
 \sup_{ \varphi \in H^1(\mathcal{C}_\delta^\sharp),\| \varphi\|_{H^1} =1} \left|L_\delta( \varphi) \right|=\mathcal{O}(\sqrt{\delta}).
$$ 
Then, using the lemma \ref{LemmeInversibilite} below, we have
\begin{equation}\label{Step1}
     \|r_\delta \|_{H^1(\mathcal{C}_\delta^\sharp)} =\mathcal{O}(\sqrt{\delta}).
\end{equation}
We deduce using also \eqref{DefinitionPhiDelta} that
$$
\exists \theta_\delta\in[0,2\pi),\quad\int_{\mathcal{C}_\delta^\sharp}  \varphi_\delta \overline{\tilde\phi_\delta}=  e^{\imath\theta_\delta } + \mathcal{O}(\sqrt\delta). 
$$ 
Now, let $\phi_\delta =e^{\imath\theta_\delta } \tilde\phi_\delta$. It corresponds to the eigenmode of Lemma \ref{LemmeQuasiMode}. We have $r_\delta= \varphi_\delta - \phi_\delta  +\mathcal{O}(\sqrt\delta)$ with \eqref{Step1} allows to deduce the estimate of Lemma \ref{LemmeQuasiMode}.

\begin{lemma}\label{LemmeInversibilite}
Let $L_\delta$ be a linear form on $H^1(\mathcal{C}_\delta^\sharp)$ such that $L_\delta(\tilde\phi_\delta) =0$. Then, there exists a unique function $v_\delta \in H^1(\mathcal{C}_\delta^\sharp)$ such that
$$
\int_{\mathcal{C}_\delta^\sharp} \nabla v_\delta \cdot \overline{\nabla \varphi} dx - \lambda_{n,\delta} \int_{\mathcal{C}_\delta^\sharp} v_\delta \overline{\varphi} dx= L^\delta(\varphi),\;\forall\varphi\in H^1(\mathcal{C}_\delta^\sharp) \quad  \mbox{ and }  \int_{\mathcal{C}_\delta^\sharp} v_\delta \overline{\tilde\phi_\delta} dx =0.
$$
Moreover, for $\delta$ small enough, there exists a constant $C$ independent of $\delta$ such that 
$$
\| v_\delta \|_{H^1(\mathcal{C}_\delta^\sharp)} \leq C \sup\, \{\,\left|L_\delta( \varphi)\right|,\;\varphi \in H^1(\mathcal{C}_\delta^\sharp),\| \varphi\|_{H^1} =1\} .$$ 
\end{lemma}
\begin{proof}
We remind that $ {\lambda}_{n,\delta}$ is a simple and isolated eigenvalue of the self-adjoint operator $A_{\delta,1}$. As a result, since  $L_\delta$ vanishes on $\tilde\phi_\delta$, existence, uniqueness and stability of $v_\delta$ is immediate from Fredholm Alternative. It remains to prove that the stability constant is independent of $\delta$. Let 
$$
L^2_0 = \{ v  \in L^2(\mathcal{C}_\delta^\sharp),\;  \; \int_{\mathcal{C}_\delta^\sharp} v \overline{\tilde\phi_\delta} dx =0\}.
$$
We introduce the reduced operator $A_\delta^r: \mathcal{D}(A_\delta^r) =   \mathcal{D}(A_{\delta, 0}(K))\cap L^2_0  \rightarrow  L^2(\mathcal{C}_\delta^\sharp)$ defined by 
$$
\forall u \in \mathcal{D}(A_\delta^r), \quad A_\delta^r \,u= A_{\delta, 0}({\bf K})\, u 
$$ 
Note that if $u \in \mathcal{D}(A_\delta^r)$, $A_\delta^r u \in L^2_0$ since
$$
 \int_{\mathcal{C}_\delta^\sharp} A_\delta^r u \cdot  \overline{\tilde\phi_\delta} = - \int_{\mathcal{C}_\delta^\sharp} \nabla u \cdot  \nabla \overline{\tilde\phi_\delta} = \lambda_{n, \delta}  \int_{\mathcal{C}_\delta^\sharp} u \cdot  \overline{\tilde\phi_\delta} dx =0.  $$ 
In addition, the operator $A_\delta^r$ is a self-adjoint operator since it is symmetric and Im$(A_\delta^r + I)= L^2_0$. Indeed, for any $f \in L^2_0$, there exists a unique $v \in  \mathcal{D}(A_\delta^r)$ such that 
$$
- \Delta v + v = f \;\mbox{in} \; \Omega_\delta, 
$$
and, we can check that since $f \in L^2_0$, $v$ is also in  $L^2_0$:
$$
0=\int_{\mathcal{C}_\delta^\sharp} \nabla v \cdot \overline{ \nabla  \tilde\phi_\delta} + \int_{\mathcal{C}_\delta^\sharp} v  \overline{\tilde\phi_\delta} dx = (\lambda_{n,\delta} + 1)  \int_{\mathcal{C}_\delta^\sharp} v  \overline{\tilde\phi_\delta} dx
 $$
 Besides the assumption \eqref{DistanceCSpectre} ensures that for $\delta$ small enough, there exists a constant $C_2$ independent of $\delta$ such that any $\lambda \in [\lambda_{n,\delta} - C_2/2,\lambda_{n,\delta} + C_2/2 ]$  does not belong to $\sigma(A_\delta^r ) $ so that 
 $$
 \|(A_\delta^r- \lambda_{n,\delta} I)^{-1}  \| \leq \frac{1}{dist(\lambda_{n,\delta}, \sigma(A_\delta^r))} \leq \frac{2}{C_2}.
 $$  
 
 \end{proof}

\bibliographystyle{plain}
\bibliography{biblioFinal}
\end{document}